\documentclass[12pt,a4paper]{article}
\usepackage[titletoc,title]{appendix}
\usepackage{amsmath,amsfonts,amssymb,bm,amsthm}
\usepackage[colorlinks,citecolor=blue,linkcolor=blue]{hyperref}
\usepackage{tikz}
\usetikzlibrary{shapes,arrows}
\usetikzlibrary{intersections,shapes.arrows}
\numberwithin{equation}{section}
\newtheorem{theorem}{Theorem}[section]
\newtheorem{lemma}[theorem]{Lemma}
\newtheorem{corollary}[theorem]{Corollary}
\theoremstyle{definition}

\newtheorem{definition}[theorem]{Definition}
\theoremstyle{remark}
\newtheorem{remark}[theorem]{Remark}
\newtheorem{notation}[theorem]{Notation}

\usepackage[T1]{fontenc}
\usepackage[top=1in, bottom=1in, left=1in, right=1in]{geometry}
\usepackage{fancyhdr}
\newcommand{\ir}{\mathrm{i}}

\newcommand{\dr}{\mathrm{d}}
\newcommand{\dist}{\operatorname{dist}}
\usepackage{enumerate}
\usepackage{xcolor,cancel}
\usepackage{relsize}

\newcommand{\Ric}{\operatorname{Ric}}

\allowdisplaybreaks

\begin{document}

\title{Global propagator for the massless Dirac operator\\
and spectral asymptotics}
\author{Matteo Capoferri\thanks{MC:
Department of Mathematics,
University College London,
Gower Street,
London WC1E~6BT,
UK;
\text{matteo\-.capoferri.16@ucl.ac.uk}
\&
School of Mathematics, Cardiff University, Senghennydd Rd, Cardiff CF24 4AG, UK;
\text{CapoferriM@cardiff.ac.uk}, \url{http://www.mcapoferri.com}.
}
\and
Dmitri Vassiliev\thanks{DV:
Department of Mathematics,
University College London,
Gower Street,
London WC1E~6BT,
UK;
D.Vassiliev@ucl.ac.uk,
\url{http://www.ucl.ac.uk/\~ucahdva/}
}}

\renewcommand\footnotemark{}

\date{1 July 2022}

\maketitle
\begin{abstract}
We construct the propagator of the massless Dirac operator $W$ on a closed Riemannian 3-manifold as the sum of two invariantly defined oscillatory integrals, global in space and in time, with distinguished complex-valued phase functions. The two oscillatory integrals --- the positive and the negative propagators --- correspond to positive and negative eigenvalues of $W$, respectively. This enables us to provide a global invariant definition of the full symbols of the propagators (scalar matrix-functions on the cotangent bundle), a closed formula for the principal symbols and an algorithm for the explicit calculation of all their homogeneous components. Furthermore, we obtain small time expansions for principal and subprincipal symbols of the propagators in terms of geometric invariants. Lastly, we use our results to compute the third local Weyl coefficients in the asymptotic expansion of the eigenvalue counting functions of $W$.

\

{\bf Keywords:} Dirac operator, hyperbolic propagators, global Fourier integral operators, Weyl coefficients.

\

{\bf 2020 MSC classes: }
primary 35L45;  
secondary 
35Q41, 
58J40, 
58J45, 
35P20. 

\end{abstract}

\tableofcontents

\allowdisplaybreaks

\section{Statement of the problem}
\label{Statement of the problem}

Let $(M,g)$ be a connected oriented closed Riemannian 3-manifold.
Throughout this paper we denote by $\nabla$ the Levi-Civita connection,
by $\Gamma^\alpha{}_{\beta \gamma}$ the Christoffel symbols
and by
\begin{equation}
\label{riemannian density}
\textstyle
\rho(x):=\sqrt{\textcolor{black}{\det} g_{\alpha\beta}(x)}
\end{equation}
the Riemannian density.

Let us clarify straight away why, when dealing with the massless Dirac operator,
we restrict our analysis to the 3-dimensional case.
The reason is twofold: on the one hand, dimension three is physically meaningful in that it represents the first step towards a potential future analysis of the relativistic 3+1-dimensional setting, and on the other hand, our method requires the eigenvalues of the principal symbol of our operator to be simple, cf.~Section~\ref{Preliminary results for general first order systems}, which is not the case for the massless Dirac operator in higher dimensions.

Let $e_j$, $j=1,2,3$, be a positively oriented global framing,
i.e.~a set of three orthonormal smooth  vector fields\footnote
{
Observe that an orientable 3-manifold is automatically parallelizable
\cite{Kirby,Stiefel}.
} whose orientation agrees with the orientation of the manifold.
In chosen local coordinates $x^\alpha$, $\alpha=1,2,3$,
we will denote by $e_j{}^\alpha$ the $\alpha$-th component
of the $j$-th vector field.
Throughout this paper we use Greek letters for holonomic (tensor) indices
and Latin for anholonomic (frame) indices.
We adopt Einstein's convention of summation over repeated indices.

Let
\begin{equation}
\label{Pauli matrices basic}
s^1:=
\begin{pmatrix}
0&1\\
1&0
\end{pmatrix}
=s_1
\,,
\quad
s^2:=
\begin{pmatrix}
0&-i\\
i&0
\end{pmatrix}
=s_2
\,,
\quad
s^3:=
\begin{pmatrix}
1&0\\
0&-1
\end{pmatrix}
=s_3
\end{equation}
be the standard Pauli matrices and let
\begin{equation}
\label{Pauli matrices projection}
\sigma^\alpha:=
s^j\,e_j{}^\alpha
\end{equation}
be their projection along the framing.
The quantity $\sigma^\alpha$ is a vector-function with values
in the space of trace-free Hermitian $2\times2$ matrices.

\begin{definition}
\label{massless dirac definition}
We call \emph{massless Dirac operator} the operator 
\begin{equation}
\label{massless dirac definition equation}
W:=-i\sigma^\alpha
\left(
\frac\partial{\partial x^\alpha}
+\frac14\sigma_\beta
\left(
\frac{\partial\sigma^\beta}{\partial x^\alpha}
+\Gamma^\beta{}_{\alpha\gamma}\,\sigma^\gamma
\right)
\right)
:H^1(M;\mathbb{C}^2)\to L^2(M;\mathbb{C}^2).
\end{equation}
Here $H^1$ is the usual Sobolev space of functions which are square integrable together
with their first partial derivatives.
\end{definition}

In relativistic particle physics the massless Dirac equation is often
referred to as the Weyl equation, which explains our notation.
Our operator $W$ appears as the result of separating out the time variable in the relativistic Weyl equation,
see \cite{spin2} for details. 
Henceforth, we refer to the massless Dirac operator
simply as the Dirac operator,
which conforms with the terminology adopted in differential geometry.

\begin{remark}
The Dirac operator admits several equivalent definitions.
The most common is the geometric definition written in terms spinor bundles.
Our analytic Definition~\ref{massless dirac definition}
is equivalent to the standard geometric one, see
\cite[Appendix B]{sphere}.
\end{remark}

Definition~\ref{massless dirac definition}
depends on the choice of framing and this issue requires clarification.

Let
\begin{equation}
\label{special unitary matrix-function}
G:M\to\mathrm{SU}(2)
\end{equation}
be an arbitrary smooth special unitary matrix-function
and let $\widetilde W$ be the Dirac operator corresponding to a given framing.
Consider the transformation
\begin{equation}
\label{transformation of the Dirac operator under change of frame}
\widetilde W\mapsto G^*\widetilde WG:=W,
\end{equation}
where the star indicates Hermitian conjugation.
It turns out that $W$ is also a
Dirac operator, only corresponding to a different framing.

Let us now look at the matter the other way round.
Suppose that $\widetilde W$ and $W$ are two Dirac operators.
Does there exist a smooth matrix-function \eqref{special unitary matrix-function}
such that $W=G^*\widetilde WG\,$? If the operators $\widetilde W$ and $W$
are in a certain sense `close' then the answer is yes, but in general
there are topological obstructions and the answer is no.
This motivates the introduction of the concept of spin structure, see
\cite{spin1,spin2}.

The gauge transformation
\eqref{special unitary matrix-function},
\eqref{transformation of the Dirac operator under change of frame}
is the manifestation, at operator level,
of the freedom of pointwise rotating the framing in a smooth way,
\begin{equation}
\label{gauge transfomation in terms of O}
\widetilde{e}_j \mapsto O_j{}^k\, \widetilde{e}_k=:e_j, \qquad O\in C^\infty(M;\mathrm{SO}(3)),
\end{equation}
via the double cover
\[
\mathrm{SU}(2) \overset{2:1}{\rightarrow} \mathrm{SO}(3).
\]

The Dirac operator
\eqref{massless dirac definition equation}
is uniquely determined
by the metric and spin structure modulo an $\mathrm{SU}(2)$ gauge transformation.

The Dirac operator is symmetric
with respect to the $L^2$ inner product
\begin{equation}
\label{inner product}
\langle u,v\rangle
:=
\int_M u^*v\,\rho\,dx\,, \qquad u,v\in L^2(M;\mathbb{C}^2),
\end{equation}
where $dx=dx^1dx^2dx^3$.
Furthermore, a simple calculation shows that it is elliptic\footnote{
Ellipticity means that
the determinant of the principal symbol does not vanish on $T^*M\setminus\{0\}$.
}.

It is well known that the Dirac operator is self-adjoint
and its spectrum is discrete, accumulating to $+\infty$ and to $-\infty$.
Let $\lambda_k$ be the eigenvalues of $W$ and $v_k$ the
corresponding orthonormal eigenfunctions,
$k\in\mathbb{Z}$. The choice of particular enumeration is irrelevant for our
purposes, but what is important is that eigenvalues are enumerated with
account of their multiplicity. Note that
the Dirac operator has the special property that
it commutes with the antilinear operator of charge conjugation
\begin{equation*}
v=
\begin{pmatrix}
v_1\\
v_2
\end{pmatrix}
\mapsto
\begin{pmatrix}
-\overline{v_2}\\
\overline{v_1}
\end{pmatrix}
=:\mathrm{C}(v),
\end{equation*}
see \cite[Appendix A]{jst_part_b} for details,
and this implies that eigenvalues have even multiplicity.

\begin{definition}
\label{dirac propagator definition}
We define the \emph{Dirac propagator} as
\begin{equation}
\label{dirac propagator definition equation}
U(t)
:=
e^{-itW}.
\end{equation}
\end{definition}

The Dirac propagator is the (distributional) solution of
the hyperbolic Cauchy problem
\begin{subequations}
\label{hyperbolic Cauchy problem}
\begin{equation}
\label{main PDE  massless Dirac}
\left(
-i\,\frac{\partial}{\partial t}
+
W
\right)
U=0\,,
\end{equation}
\begin{equation}
\label{half wave equation initial condition}
U(0)=\operatorname{Id}\,.
\end{equation}
\end{subequations}
It is a time-dependent unitary operator which can written via functional calculus as
\begin{equation}
\label{dirac propagator definition equation detailed}
U(t)=
\sum_{\lambda_k} e^{-it\lambda_k}\,
v_k\,
\langle v_k\,,\,\cdot\ \rangle.
\end{equation}

The Dirac propagator can be written as a sum of three operators
\begin{equation*}
\label{dirac propagator sum of three}
U(t)=U^+(t)+U^0+U^-(t)
\end{equation*}
defined as
\begin{subequations}
\begin{equation}
\label{dirac propagator definition equation detailed positive}
U^+(t):=
\sum_{\lambda_k>0} e^{-it\lambda_k}\,
v_k\,
\langle v_k\,,\,\cdot\ \rangle,
\end{equation}
\begin{equation}
\label{dirac propagator definition equation detailed zero mode}
U^0:=
\sum_{\lambda_k=0}
v_k\,
\langle v_k\,,\,\cdot\ \rangle,
\end{equation}
\begin{equation}
\label{dirac propagator definition equation detailed negative}
U^-(t):=
\sum_{\lambda_k<0} e^{-it\lambda_k}\,
v_k\,
\langle v_k\,,\,\cdot\ \rangle.
\end{equation}
\end{subequations}
We call the operators
\eqref{dirac propagator definition equation detailed positive},
\eqref{dirac propagator definition equation detailed zero mode}
and
\eqref{dirac propagator definition equation detailed negative}
\emph{positive}, \emph{zero mode} and \emph{negative} propagators, respectively.
These are time-dependent partial isometries.
Note that the operator $U^0$ is nontrivial only if the Dirac operator
has zero modes (i.e.~if zero is an eigenvalue).

We define the \emph{positive} ($+$) and \emph{negative} ($-$) local counting functions as
\begin{equation}
\label{local counting functions}
N_\pm(y;\lambda):=
\begin{cases}
0 & \text{for }\lambda \le0,\\
\sum_{0<\pm\lambda_k<\lambda} [v_k(y)]^* \,v_k(y)& \text{for }\lambda >0.
\end{cases}
\end{equation}
Of course, integration over $M$ gives
\begin{equation}
\label{global counting functions}
N_\pm(\lambda)
:=
\int_M N_\pm(y;\lambda)\,\rho(y)\,dy
=
\begin{cases}
0 & \text{for }\lambda \le0,\\
\sum_{0<\pm\lambda_k<\lambda} 1& \text{for }\lambda >0.
\end{cases}
\end{equation}
The functions
\eqref{global counting functions}
are the `global' counting functions,
the only difference with the usual definition
\cite{SaVa} being that we count
the positive and negative eigenvalues separately.

Let $\hat{\mu}:\mathbb{R}\to \mathbb{C}$ be a smooth function such that $\hat{\mu}= 1$ in some neighbourhood of the origin and $\operatorname{supp} \hat{\mu}$ is sufficiently small. Here `sufficiently small' means that
$\operatorname{supp} \hat{\mu}\subset(-T_0,T_0)$,
where $T_0$ is the infimum of lengths of all the geodesic loops
originating from all the points of the manifold.

Following the notation of \cite{wave}, we write the Fourier transform as
\begin{equation}
\mathcal{F}_{\lambda\to t}[f](t)=\hat{f}(t)=\int_{-\infty}^{+\infty} e^{-it\lambda} f(\lambda) \,d\lambda
\end{equation}
and the inverse Fourier transform as
\begin{equation}
\mathcal{F}_{t\to \lambda}^{-1}[\hat f](\lambda)=f(\lambda)=\dfrac{1}{2\pi}\int_{-\infty}^{+\infty} e^{it\lambda}\hat{f}(t) \,dt.
\end{equation}
Accordingly, we put $\mu:=\mathcal{F}^{-1}[\hat \mu]$.

It is known \cite{DuGu,Ivr80,Ivr84,Ivr98,SaVa} that the mollified derivative of the positive (resp.~negative) counting function admits a complete asymptotic expansion in integer powers of $\lambda$:
\begin{equation}
\label{expansion for mollified derivative of counting function}
(N_\pm'*\mu)(y,\lambda)=
c^\pm_{2}(y)\,\lambda^{2}
+
c^\pm_{1}(y)\,\lambda
+
c^\pm_{0}(y)
+\dots
\quad
\text{as}
\quad
\lambda\to+\infty.
\end{equation}
Here $*$ stands for the convolution in the variable $\lambda$.

\begin{definition}
\label{definition of Weyl coefficients}
We call \emph{local Weyl coefficients} the smooth functions $c^\pm_{k}(y)$ appearing in the asymptotic expansions
\eqref{expansion for mollified derivative of counting function}.
\end{definition}

\begin{remark}
\begin{enumerate}[(i)]
\item
Our definition of Weyl coefficients
does not depend on the choice of mollifier $\mu$.
If $\widetilde{\mu}$ is another mollifier
with the same support properties, then 
\[
(N_\pm'*\mu)(y,\lambda)-(N_\pm'*\widetilde\mu)(y,\lambda)=O(\lambda^{-\infty})\quad \text{as} \quad \lambda\to+\infty.
\]

\item
Our definition of Weyl coefficients is, in a sense, unusual. The standard convention in the literature is to call local Weyl coefficients the functions appearing in the asymptotic expansion of the mollified counting function $N*\mu\,$, as opposed to its derivative.
The two definitions are, effectively, the same up to integrating factors,
\begin{multline}
\label{expansion for mollified counting function}
(N_\pm*\mu)(y,\lambda)=
\int_{-\infty}^\lambda
(N_\pm'*\mu)(y,\kappa)\,d\kappa
\\
=
\frac13\,
c^\pm_{2}(y)\,\lambda^{3}
+
\frac12\,
c^\pm_{1}(y)\,\lambda^2
+
c^\pm_{0}(y)\,\lambda
+\dots
\quad
\text{as}
\quad
\lambda\to+\infty,
\end{multline}
compare
\eqref{expansion for mollified derivative of counting function}
with
\eqref{expansion for mollified counting function}.
As a matter of convenience, we will stick with
Definition~\ref{definition of Weyl coefficients}
throughout this paper.

\item
It was shown in \cite{jst_part_b} that
\begin{equation}
\label{first two weyl coefficients}
c^\pm_{2}(y)=\frac1{2\pi^2}\,,
\qquad
c^\pm_{1}(y)=0.
\end{equation}

\item
The unmollified counting functions $N_\pm(y,\lambda)$
also admit asymptotic expansions as $\lambda\to+\infty$,
but here the situation is more delicate because these functions
are discontinuous and one encounters number-theoretic issues.
It is known \cite{CDV,jst_part_b} that
\begin{equation*}
\label{one term asymptotics unmollified}
N_\pm(y,\lambda)
=
\frac1{6\pi^2}\,
\lambda^{3}
+
O(\lambda^2)
\quad
\text{as}
\quad
\lambda\to+\infty
\end{equation*}
uniformly over $y\in M$ and, furthermore,
under appropriate assumptions on geodesic loops,
\begin{equation*}
\label{two term asymptotics unmollified}
N_\pm(y,\lambda)
=
\frac1{6\pi^2}\,
\lambda^{3}
+
o(\lambda^2)
\quad
\text{as}
\quad
\lambda\to+\infty.
\end{equation*}

\item 
An important topic in the spectral theory of first order elliptic systems is the issue of spectral asymmetry \cite{asymm1, asymm2, asymm3, asymm4}. Let us mention that to observe spectral asymmetry for our Dirac operator one as to go as far as the \emph{sixth} Weyl coefficients. This follows from the fact \cite{bismut, gilkey} that the eta function
\[
\eta(s):=\sum_{\lambda_k\ne 0}\dfrac{\operatorname{sgn} \lambda_k}{|\lambda_k|^s}=\int_0^{+\infty} \lambda^{-s} (N'_+(\lambda)-N'_-(\lambda))\,d \lambda
\]
is holomorphic in the complex half-plane $\operatorname{Re}s>-2$ and has its first pole at $s=-2$. The value of the residue of the eta function at $s=-2$, which was computed explicitly by Branson and Gilkey \cite{branson}, describes the difference 
\[
\int_M (c^+_{-3}(y)-c^-_{-3}(y))\,\rho(y)\,d y
\] 
between the sixth (global) Weyl coefficients.

\end{enumerate}
\end{remark}

Our paper has two main objectives.

\

\textbf{Objective 1\ }
Construct the propagators $U^\pm(t)$ explicitly,
modulo integral operators with infinitely smooth kernels,
and do so as a single invariantly defined oscillatory integral
global in space and in time.

\

\textbf{Objective 2\ }
Compute the third Weyl coefficient $c^\pm_{0}(y)$.

\

\begin{remark}
One cannot, in general, identify the third Weyl
coefficient by looking at the asymptotic behaviour of the
unmollified counting function.
In order to illustrate this point, let us consider the 3-torus
equipped with standard flat metric.
Already in this simple case the mathematical statement
\begin{equation*}
\label{two term asymptotics unmollified false}
N_\pm(y,\lambda)
=
\frac1{6\pi^2}\,
\lambda^{3}
+
c^\pm_{0}(y)\,\lambda
+
o(\lambda)
\quad
\text{as}
\quad
\lambda\to+\infty
\end{equation*}
is \emph{false}.
This fact can be established by writing down the eigenvalues
explicitly as in \cite[Appendix B]{jst_part_b}
and using standard results
\cite{heath-brown}
from analytic number theory.
\end{remark}

\section{Main results}
\label{Main results}

The study of Dirac operators in curved space, arguably the most important operators from the point of view of physical applications alongside the Laplacian, has a long and noble history in the mathematical literature. Excellent introductions to the subject can be found in \cite{Lawson} and \cite{friedrich}.

Due to the physical significance of the topic,
numerous researchers have contributed to our current understanding of the
spectrum of Dirac operators on Riemannian manifolds. One can ask, for example, how the eigenvalues behave under perturbations of the metric \cite{bourguignon,sphere,torus}, how the spectrum depends on the spin structure \cite{Bar2}, whether zero modes exist \cite{Bar3}, \emph{et cetera}.

Later in this paper we will be concerned with the study of the asymptotic behaviour of large (in modulus) eigenvalues of the Dirac operator on a closed 3-manifold. In the case of scalar elliptic operators, such as for example the Laplace--Beltrami operator, a wide range of classical techniques are available in the literature to compute spectral asymptotics. However, if one is interested in a first order system, whose spectrum is, in general, not semi-bounded, the heat kernel method can no longer be applied, at least in its original form, and even resolvent techniques require major modification \cite{ASV}. A very natural approach in this case is the so-called wave method, going back to Levitan \cite{levitan} and Avakumovic \cite{avakumovic}, which involves recovering information about the eigenvalue counting function from the behaviour of the wave propagator, see \eqref{dirac propagator definition equation detailed}. How this can be done will be explained in greater detail later on.  This partly motivates our interest in the Dirac propagator \eqref{dirac propagator definition equation}, which is also of interest on its own. Of course, the hyperbolic Cauchy problem \eqref{hyperbolic Cauchy problem} for $W$ lies at the heart of relevant applications in theoretical physics (e.g., the mathematical description of neutrinos/antineutrinos in curved space). 

In order to construct the propagator \eqref{dirac propagator definition equation} \emph{precisely}, one needs to know all eigenvalues and eigenfunctions of $W$, which is unrealistic for a generic Riemannian manifold. However, microlocal techniques allow one to construct the propagator \eqref{dirac propagator definition equation} \emph{approximately}, modulo an integral operator with smooth integral kernel. This fact is well-known and an extensive discussion can be found, for instance, in \cite{Hor}.

There are, however, several issues with this classical construction: 
(i) it is not invariant under changes of local coordinates, (ii) it is local in space and (iii) it is local in time. The last issue, locality in time, is especially serious: it is to do with obstructions associated with caustics. In practice, constructing a propagator locally in time means that for large times one has to use compositions
\[
U(t)=U(t-t_j)\circ U(t_j -t_{j-1})\circ \cdots \circ U(t_2 -t_1)\circ U(t_1).
\]
The propagator $U(t)$ is a special case of a Fourier integral operator and it is known that handling compositions of such operators is a daunting task.

Our goal is to construct the Dirac propagator explicitly, in a global -- i.e., as a single oscillatory integral -- and invariant  (under change of coordinates and gauge transformations) fashion. The key idea, originally proposed by Laptev, Safarov and Vassiliev in \cite{LSV} and further developed in \cite{SaVa}, is to use Fourier integral operators with \emph{complex-valued}, as opposed to real valued, phase function. Crucially, this allows one to circumvent the topological obstructions due to caustics.

Our work partly builds upon \cite{CDV} and \cite{jst_part_b}.
In \cite{CDV}, using the wave method, Chervova, Downes and Vassiliev obtained an explicit formula for the second Weyl coefficient of an elliptic self-adjoint first order pseudodifferential matrix operator, fixing thirty years of incorrect or incomplete publications in the subject, see \cite[Section~11]{CDV}. In \cite{jst_part_b} the same authors applied the results from \cite{CDV} to the Dirac operator. Unlike the current paper, the approach from \cite{CDV} is not geometric in nature and the complexity
of phase functions is not actually put to use. Note that some results from \cite{jst_part_b} were improved by Strohmaier and Li in \cite{strohmaier}, where the authors studied the second term of the mollified spectral counting function of Dirac type operators and characterised operators in this class with vanishing second Weyl coefficient.

A fully geometric global construction of the (scalar) wave propagator $e^{-it \sqrt{-\Delta}}$ on closed Riemannian manifolds, as a single oscillatory integral with complex-valued phase function, was recently proposed by the authors and Levitin in \cite{wave}, and subsequently extended to the Lorentzian setting in \cite{lorentzian}. The publication \cite{wave} is the starting point of the current paper.  \textcolor{black}{The extension of the results of the current paper to globally hyperbolic Lorentzian manifolds is carried out in \cite{lorentzian_dirac}.}

Our main results are as follows.

\begin{enumerate}

\item
We present a global construction of each of the two propagators, the positive propagator $U^+(t)$ and the negative propagator $U^-(t)$, as a single invariantly defined oscillatory integral, global in space and in time, with distinguished complex-valued phase function (Theorem~\ref{theorem propagator for A}\textcolor{black}{, Definition~\ref{definition LC phase functions with epsilon dirac}, Definition~\ref{definition full symbols dirac propagator}}). We provide a closed formula for the principal symbols of the propagators (Theorem~\ref{theorem small time expansion principal symbol}) and an algorithm for the calculation of the subprincipal symbols and all asymptotic components of lower degree of homogeneity in momentum (subsection~\ref{The algorithm}).

\item We give an explicit small time expansion of principal and subprincipal symbols of positive and negative propagators in terms of geometric invariants (Theorem~\ref{theorem invariant expression principal and subprincipal symbols dirac}).

\item We compute the third local Weyl coefficients in the asymptotic expansion of the two eigenvalue counting functions \eqref{local counting functions} (Theorem~\ref{theorem about third Weyl coefficient}).

\end{enumerate}

Along the way we prove a number of results about general first order elliptic systems and invariant representations of pseudodifferential operators on manifolds. Note that the third Weyl coefficients can, in principle and with some work,  be also obtained by a different method using results available in the literature, see~Remark~\ref{remark about third weyl coefficient}.

\

Our paper is structured as follows.

In Section~\ref{Preliminary results for general first order systems}
we explain how to construct explicitly positive and negative propagators for a general first order elliptic self-adjoint (pseudo)differential matrix operator, with a rigorous mathematical justification.

In Section~\ref{Invariant description} we deal with the delicate issue of invariant descriptions of pseudodifferential operators acting on scalar functions. In particular, we examine the relation between our $g$-subprincipal symbol and the standard notion of subprincipal symbol for operators acting on half-densities.

In Section~\ref{Global propagator for the massless Dirac operator} we apply the results from Section~\ref{Preliminary results for general first order systems} to the Dirac operator. A formula for the principal symbol of positive and negative Dirac propagators is provided in Section~\ref{Principal symbol of the global Dirac propagator}, whereas small time expansions for principal and subprincipal symbols of positive and negative propagators are obtained in Section~\ref{Explicit small time expansion of the symbol}. Our final results are expressed in terms of geometric invariants: curvature of the Levi-Civita connection associated with the metric $g$ and torsion of the Weitzenb\"ock connection generated by the framing defining the Dirac operator.

In Section~\ref{An application: spectral asymptotics} we use the results from Section~\ref{Explicit small time expansion of the symbol} to compute the third local Weyl coefficients for the Dirac operator.

Finally, in Section~\ref{Examples} we apply our techniques to two explicit examples: $M=\mathbb{S}^3$, where formulae are isotropic in momentum, and $M=\mathbb{S}^2 \times \mathbb{S}^1$, where they are not.

The paper is complemented by two appendices, containing background material and technical proofs.

\section{Preliminary results for general first order systems}
\label{Preliminary results for general first order systems}

In this section we will consider a broader class of first order systems and we will prove fairly general results, which will be later applied to the special case of the Dirac operator. In doing so, we will need some of the technology developed in \cite{CDV}. The setting of our analysis is somewhat different from that in \cite{CDV}, in that our operators are differential, as opposed to pseudodifferential (see also Remark~\ref{remark about pseudodifferential}), and act on scalar functions on a Riemannian manifold, as opposed to half-densities on a manifold with no metric structure. In particular, the change of the space in which the operator acts raises delicate issues concerning the invariance of the mathematical objects involved. For these reasons we provide here a modified version of some of the results from \cite{CDV}, adapted to the setting of our paper.

Throughout this section, $M$ will be a smooth connected closed Riemannian manifold of dimension $d\ge 2$.

Let $A$ be an elliptic symmetric (with respect to \eqref{inner product}) first order $m\times m$ matrix differential operator acting on $m$-columns of smooth complex-valued scalar functions $v\in C^\infty(M;\mathbb{C}^m)$ and let $A_\mathrm{prin}: T'M \to \mathrm{Herm}(m,\mathbb{C})$ be the principal symbol of $A$, where $T'M:=T^*M\setminus\{0\}$ and $\mathrm{Herm}(m,\mathbb{C})$ is the real vector space of $m\times m$ Hermitian matrices. 

We denote by $h^{(j)}(x,\xi)$
the eigenvalues of $A_\mathrm{prin}(x,\xi)$ enumerated in increasing order, with positive index $j=1,2,\ldots, m^+$ for positive $h^{(j)}(x,\xi)$ and negative index $j=-1,-2,\ldots, -m^-$ for negative $h^{(j)}(x,\xi)$.
We assume that the eigenvalues of the principal symbol $A_\mathrm{prin}$ are simple.
Clearly, $m=m^++m^-$, because the ellipticity condition
$\det A_\mathrm{prin}(x,\xi)\ne0$
ensures that all eigenvalues are nonzero. In fact, as our operator is differential, one can show \cite[Remark~2.1]{israel} that $m$ can only be even and that we have
\begin{equation}
\label{m plus and m minus equal m/2}
m^+=m^-=\frac{m}2\,.
\end{equation}
Furthermore, the eigenvalues $h^{(j)}$ of the principal symbol and the corresponding normalised eigenvectors $v^{(j)}$ possess the symmetry
\begin{equation}
\label{symmetry eigenvalues and eigenvectors}
h^{(-j)}(x,\xi)=-h^{(j)}(x,-\xi), \quad v^{(-j)}(x,\xi)=v^{(j)}(x,-\xi), 
\qquad
j=1,\ldots,\frac{m}2\,.
\end{equation}

Under the above assumptions the spectrum of $A$ is discrete and accumulates to $+\infty$ and to $-\infty$. We denote eigenvalues and orthonormalised \textcolor{black}{(smooth)} eigenfunctions of $A$ by $\lambda_k$ and $v_k$, respectively, enumerated with account of their multiplicity.

By replacing $W$ with $A$, one can define the `full' propagator $U_A(t)$ for $A$ via \eqref{dirac propagator definition equation detailed}, as well as the positive, zero mode and negative propagators via \eqref{dirac propagator definition equation detailed positive}--\eqref{dirac propagator definition equation detailed negative}, which we denote by $U^+_A(t)$, $U^0_A$ and $U^-_A(t)$, respectively.


Each eigenvalue $h^{(j)}(x,\xi)$ of the principal symbol can be interpreted as a Hamiltonian on the cotangent bundle. The corresponding Hamiltonian flow $(x^{(j)}(t;y,\eta), \xi^{(j)}(t;y,\eta))$, i.e.~the (global) solution to Hamilton's equations
\[
\dot{x}^{(j)}=h_\xi^{(j)}(x^{(j)},\xi^{(j)}), \qquad \dot{\xi}^{(j)}=-h_x^{(j)}(x^{(j)},\xi^{(j)})
\]
with initial condition $(x^{(j)}(0;y,\eta),\xi^{(j)}(0;y,\eta))=(y,\eta)$, generates a Lagrangian manifold to which one can, in turn, associate a \emph{global} Lagrangian distribution. See \cite[Section~2]{wave} and references therein for details. In particular, the singularities of the solution to the initial value problem
\begin{equation}
(-i \partial_t+A)v=0, \qquad \left. v\right|_{t=0}=v_0
\end{equation}
propagate along Hamiltonian trajectories generated by the eigenvalues of $A_\mathrm{prin}$.

\subsection{Positive and negative propagators: an abstract approach}
\label{Positive and negative propagators: an abstract approach}

Our aim is to show that $U^+_A(t)$ and $U^-_A(t)$ can be \emph{separately} approximated by a finite sum of global oscillatory integrals. 
Before doing so, let us state and prove an abstract preparatory theorem.

\begin{notation}
Let 
\[
v\in C^\infty(\mathbb{R}\times M_x\times M_y), \qquad (\lambda,x,y)\mapsto v(\lambda,x,y).
\]
We write
\begin{equation*}
\label{notation big O}
v=O(|\lambda|^{-\infty}) \quad \text{as }\lambda\to\pm\infty
\end{equation*}
if for every $\alpha>0$,
every $k\in \mathbb{N}$ and every linear partial differential operator $P$ with infinitely smooth coefficients of order $k$ on $M_x\times M_y$
there exists a positive constant $C_{\alpha,P}$ such that
\begin{equation*}
\label{estimate notation big O}
|P v|\le C_{\alpha,P}\, |\lambda|^{-\alpha}
\quad
\text{for\ \,}
\pm\lambda>1\,,
\end{equation*}
uniformly over $M_x\times M_y\,$.
\end{notation}

\begin{theorem}
\label{theorem positive and negative propagator abstract approach}
Let
$(T_-,T_+)\subseteq \mathbb{R}$ be an open interval
(possibly, the whole real line)
and let
$u^+(t,x,y)$, $u^-(t,x,y)$, $\widetilde{u}^+(t,x,y)$ and $\widetilde{u}^-(t,x,y)$ be elements of
$C^\infty(M_x\times M_y; \mathcal{D}'(T_-,T_+))$, satisfying
\begin{enumerate}
\item[(a)] $u^+(t,x,y)+u^-(t,x,y)=\widetilde u^+(t,x,y)+\widetilde u^-(t,x,y) \mod C^\infty((T_-,T_+)\times M_x\times M_y)\,$.
\end{enumerate}
Furthermore, assume that for every $\zeta \in C^\infty_0(T_-,T_+)$ we have
\begin{enumerate}
\item[(b)] $\mathcal{F}^{-1}_{t\to \lambda}[\zeta \,u^\pm]=O(|\lambda|^{-\infty}) \quad \text{as }\lambda\to \mp \infty$,
\item[(c)] $\mathcal{F}^{-1}_{t\to \lambda}[\zeta\,\widetilde u^\pm]=O(|\lambda|^{-\infty}) \quad \text{as }\lambda\to \mp \infty$.
\end{enumerate}
Then
\begin{equation}
u^\pm(t,x,y)=\widetilde u^\pm(t,x,y) \mod C^\infty((T_-,T_+)\times M_x\times M_y).
\end{equation}
\end{theorem}

\begin{proof}
Let $\zeta\in C_0^\infty(T_-,T_+)$. Multiplying (a) by $\zeta(t)$ we get
\begin{multline}
\label{proof preparatory theorem formula 1}
 \zeta(t)\,u^+(t,x,y)+\zeta(t)\,u^-(t,x,y)=\zeta(t)\,\widetilde u^+(t,x,y)+\zeta(t)\,\widetilde u^-(t,x,y) \\
 \mod C_0^\infty(\mathbb{R}\times M_x\times M_y).
\end{multline}
Applying the inverse Fourier transform
$\mathcal{F}^{-1}_{t\to \lambda}$
to \eqref{proof preparatory theorem formula 1},
letting $\lambda\to +\infty$ and using
\linebreak
assumptions (b) and (c) we obtain
\begin{equation}
\label{proof preparatory theorem formula 2}
\mathcal{F}^{-1}_{t\to\lambda}[\zeta\,u^+]=\mathcal{F}^{-1}_{t\to\lambda}[\zeta\,\widetilde u^+]+O(|\lambda|^{-\infty})\quad \text{as }\lambda\to+\infty.
\end{equation}
Here, when dealing with the remainder from
\eqref{proof preparatory theorem formula 1},
we used the fact that the Fourier transform of a compactly supported smooth function is rapidly decreasing. The compactness of $M$ ensures a uniform estimate in the spatial variables.

Furthermore, (b) and (c) immediately imply 
\begin{equation}
\label{proof preparatory theorem formula 3}
\mathcal{F}^{-1}_{t\to\lambda}[\zeta\,u^+]=\mathcal{F}^{-1}_{t\to\lambda}[\zeta\,\widetilde u^+]+O(|\lambda|^{-\infty})\quad \text{as }\lambda\to-\infty.
\end{equation}
Combining \eqref{proof preparatory theorem formula 2} and \eqref{proof preparatory theorem formula 3} we arrive at
\begin{equation*}
\label{proof preparatory theorem formula 4}
\mathcal{F}^{-1}_{t\to\lambda}[\zeta\,(u^+-\widetilde u^+)]
=
O(|\lambda|^{-\infty})\quad \text{as }|\lambda|\to+\infty,
\end{equation*}
which implies
\begin{equation*}
\label{proof preparatory theorem formula 5a}
\zeta\,(u^+-\widetilde u^+)\in C^\infty(\mathbb{R}\times M_x\times M_y).
\end{equation*}
As $\zeta\in C^\infty_0(T_-,T_+)$ in the above formula is arbitrary,
we conclude that
\begin{equation*}
\label{proof preparatory theorem formula 5b}
u^+-\widetilde u^+\in C^\infty((T_-,T_+)\times M_x\times M_y).
\end{equation*}

A similar argument gives
\begin{equation*}
\label{proof preparatory theorem formula 6}
u^--\widetilde{u}^- \in C^\infty((T_-,T_+)\times M_x\times M_y).
\end{equation*}
\end{proof}

\subsection{Construction of positive and negative propagators}
\label{Construction of positive and negative propagators}

\begin{theorem}
\label{theorem propagator for A}
\textcolor{black}{
Let $A$ be an elliptic self-adjoint first order pseudodifferential operator acting on $m$-columns of scalar functions over $M$,  whose principal symbol has simple eigenvalues.}
The positive and negative propagators \textcolor{black}{$U^+_A(t)$ and $U^-_A(t)$} can be written, modulo an infinitely smoothing operator, as a finite sum of oscillatory integrals, global in space and in time. More precisely, we have
\begin{equation}
\label{propagator approximated general theorem positive}
U^+_A(t)\overset{\mod \Psi^{-\infty}}{=}\sum_{j=1}^{m^+} U^{(j)}_A(t),
\end{equation}
\begin{equation}
\label{propagator approximated general theorem negative}
U^-_A(t)\overset{\mod \Psi^{-\infty}}{=}\sum_{j=1}^{m^-} U^{(-j)}_A(t),
\end{equation}
where
\begin{equation}
\label{oscillatory intergral for U^(j)}
U^{(j)}_A(t):=\frac{1}{(2\pi)^d}\int_{T'M} e^{i\varphi^{(j)}(t,x;y,\eta)}\,\mathfrak{a}^{(j)}(t;y,\eta)\,\chi^{(j)}(t,x;y,\eta)\,w^{(j)}(t,x;y,\eta)\,\left( \,\cdot\,\right)\rho(y)\,dy\,d\eta
\end{equation}
and
\begin{itemize}
\item 
by $\overset{\mod \Psi^{-\infty}}{=}$ we mean that the operator on the LHS is equal to the operator on the RHS up to an integral operator with infinitely smooth integral kernel;

\item $\left( \,\cdot\,\right)$ is meant for insertion of $f_0(y)$ when computing $(U^{(j)}_Af_0)(x)$;

\item
the phase function $\varphi^{(j)}\in C^\infty(\mathbb{R}\times M\times T'M; \mathbb{C})$ satisfies
\begin{enumerate}[(i)]
\item
$\left.\varphi^{(j)}\right|_{x=x^{(j)}}=0$,

\item
$\left. \varphi^{(j)}_{x^\alpha}\right|_{x=x^{(j)}}=\xi^{(j)}_\alpha$,

\item
$\left. \det\varphi^{(j)}_{x^\alpha\eta_\beta}\right|_{x=x^{(j)}}\ne0$,

\item
$\operatorname{Im} \varphi^{(j)}\ge 0$;
\end{enumerate}

\item
the symbol $\mathfrak{a}^{(j)}\in \mathrm{S}_{\mathrm{ph}}^0(\mathbb{R}\times T'M;\mathrm{Mat}(m;\mathbb{C}))$ is an element in the class of poly\-homogeneous symbols of order zero with values in $m\times m$ complex matrices, which means that $\mathfrak{a}^{(j)}$ admits an asymptotic expansion in components positively homogeneous in momentum,
\begin{equation}
\label{asympotic expansion symbol for A}
\mathfrak{a}^{(j)}(t;y,\eta)\sim \sum_{k=0}^{+\infty} \mathfrak{a}^{(j)}_{-k}(t;y,\eta), \qquad \mathfrak{a}^{(j)}_{-k}(t;y,\alpha\,\eta)=\alpha^{\textcolor{black}{-k}}\, \mathfrak{a}^{(j)}_{-k}(t;y,\eta), \quad\forall \alpha >0;
\end{equation}

\item the function $\chi^{(j)}\in C^\infty(\mathbb{R}\times M \times T'M)$ is a cut-off satisfying
\begin{enumerate}[(I)]
\item $\chi^{(j)}(t,x;y,\eta)=0$ on $\{(t,x;y,\eta) \,|\, |h^{(j)}(y,\eta)|\leq 1/2\}$,
\item $\chi^{(j)}(t,x;y,\eta)=1$ on the intersection of $\{(t,x;y,\eta) \,|\, |h^{(j)}(y,\eta)| \geq 1\}$ with some conical neighbourhood of $\{(t,x^{(j)}(t;y,\eta);y,\eta) \}$,
\item $\chi^{(j)}(t,x;y,\alpha\, \eta)=\chi^{(j)}(t,x;y,\eta)$ for $\alpha\geq 1$ on $\{ (t,x;y,\eta) \, | \, |h^{(j)}(y,\eta)|\geq 1   \}$;
\end{enumerate}

\item
the weight $w^{(j)}$ is defined by 
\begin{equation}
\label{weight definition general case}
w^{(j)}(t,x;y,\eta):= [\rho(x)\,\rho(y)]^{-\frac12} \left[{\det}^2( \varphi^{(j)}_{x^\alpha\eta_\beta}) \right]^\frac14,
\end{equation}
where the smooth branch of the complex root is chosen in such a way that 
\[
w^{\color{black}(j)}(0,y;y,\eta)=[\rho(y)]^{-1}.
\]
\end{itemize}
\end{theorem}

\begin{remark}
Note that the weight $w^{(j)}$ is the inverse of a smooth density in the variable $y$ and a smooth scalar function in all other variables.
The powers of the Riemannian density $\rho$ in \eqref{weight definition general case} are chosen in such a way that the symbol $\mathfrak{a}^{(j)}$ and
the integral kernel
\begin{equation}
\label{integral kernel of U^(j)}
u^{(j)}(t,x,y):=\frac{1}{(2\pi)^d}\int_{T'_yM} e^{i \varphi^{(j)}(t,x;y,\eta)}\,\mathfrak{a}^{(j)}(t;y,\eta)\,\chi^{(j)}(t,x;y,\eta)\,w^{(j)}(t,x;y,\eta)\,d\eta
\end{equation}
of the operator \eqref{oscillatory intergral for U^(j)} are scalar functions in all variables.
The fact that the symbol is a genuine scalar function on
$\mathbb{R}\times T'M$ is a crucial feature of our construction.

Taking the square and then extracting the fourth root in \eqref{weight definition general case} serves the purpose of making the weight invariant under inversion of a single coordinate $x^\alpha$ or a single coordinate $y^\alpha$. Note, however, that if one works on an orientable and oriented manifold, then one can simplify \eqref{weight definition general case} to read
\begin{equation*}
\label{weight definition general case oriented manifold}
w^{(j)}(t,x;y,\eta)= [\rho(x)\,\rho(y)]^{-\frac12} \left[{\det}\varphi^{(j)}_{x^\alpha\eta_\beta}\right]^\frac12.
\end{equation*}
\end{remark}
\begin{remark}
The existence of a phase function satisfying conditions (i)--(iv)
is a nontrivial matter. In fact, the space of phase function satisfying these conditions
is nonempty and path-connected, see \cite[Lemmata~1.4~and~1.7]{LSV}.
\end{remark}

\color{black}
\begin{remark}
Let us emphasise that a phase function $\varphi^{(j)}$ satisfying conditions (i)--(iv) from Theorem~\ref{theorem positive and negative propagator abstract approach} automatically satisfies
\begin{equation}
\label{eikonal equation along the flow}
\varphi^{(j)}_t(t,x^{(j)};y,\eta)+h^{(j)}(x^{(j)},\xi^{(j)})=0,
\end{equation}
see, e.g., \cite[Subsection~2.4.1]{SaVa}. 
The equation
\begin{equation}
\label{eikonal equation}
\varphi^{(j)}_t(t,x;y,\eta)+h^{(j)}(x,\nabla\varphi^{(j)}(t,x;y,\eta))=0
\end{equation}
is known in the literature as \emph{eikonal equation}. Note that when $x=x^{(j)}(t;y,\eta)$ formula \eqref{eikonal equation} turns into \eqref{eikonal equation along the flow}. In the classical approach to the construction of hyperbolic propagators,  \eqref{eikonal equation} is required to be satisfied in some open neighbourhood of 
\begin{equation}
\label{neighbourhood of the flow}
\{(t,x;y,\eta)\in \mathbb{R}\times M \times T'M \ |\ x=x^{(j)}(t;y,\eta)\}.
\end{equation}
This is a fundamental difference with the approach adopted in the current paper, where \eqref{eikonal equation} is only required to be satisfied `along the Hamiltonian flow', i.e.,  one only needs \eqref{eikonal equation along the flow}. Indeed,  there is no open neighbourhood of \eqref{neighbourhood of the flow} where the special phase functions that will be introduced and used in Section~\ref{Global propagator for the massless Dirac operator} --- the \emph{Levi-Civita phase functions} --- satisfy \eqref{eikonal equation}. Relaxing the requirements on our phase functions is needed in order to accommodate an imaginary part and, consequently, circumvent obstructions arising from caustics.
\end{remark}
\color{black}

\begin{proof}[Proof of Theorem~\ref{theorem propagator for A}]
Suppose that we have constructed the symbols $\mathfrak{a}^{(j)}$
appearing in the oscillatory integrals
\eqref{oscillatory intergral for U^(j)}
so that
\begin{equation}
\label{sum of ALL the U_A^(j)}
\widetilde{U}_A(t):=
\sum_jU^{(j)}_A(t)
=
\sum_{j=1}^{m^+} U^{(j)}_A(t)
+
\sum_{j=1}^{m^-} U^{(-j)}_A(t)
\end{equation}
satisfies
\begin{subequations}
\begin{equation}
\label{proof propagator general case 1}
\left(
-i\,\frac{\partial}{\partial t}
+
A
\right)
\widetilde{U}_A(t)\overset{\mod \Psi^{-\infty}}{=}0\,,
\end{equation}
\begin{equation}
\label{proof propagator general case 2}
\widetilde{U}_A(0)\overset{\mod \Psi^{-\infty}}{=} \mathrm{Id}\,.
\end{equation}
\end{subequations}
How to achieve this will be explained in subsection~\ref{The algorithm}.

Put
\begin{equation*}
\label{sum of positive u^(j)}
\widetilde{u}^+(t,x,y):=
\sum_{j=1}^{m^+}u^{(j)}(t,x,y),
\end{equation*}
\begin{equation*}
\label{sum of negative u^(j)}
\widetilde{u}^-(t,x,y):=
\sum_{j=1}^{m^-}u^{(-j)}(t,x,y),
\end{equation*}
so that the Schwartz kernel of the operator $\widetilde{U}_A(t)$ reads
\begin{equation*}
\label{sum of ALL u^(j)}
\widetilde{u}(t,x,y)
=
\widetilde{u}^+(t,x,y)
+
\widetilde{u}^-(t,x,y).
\end{equation*}

Let $u(t,x,y)$, $u^+(t,x,y)$ and $u^-(t,x,y)$
be the Schwartz kernels of the operators
$U_A(t)$, $U^+_A(t)$, and $U^-_A(t)$, respectively.

Formulae
\eqref{proof propagator general case 1}
and
\eqref{proof propagator general case 2}
imply
\begin{equation}
\label{u agrees with tilde u}
u(t,x,y)=\widetilde u(t,x,y) \mod C^\infty(\mathbb{R}\times M_x\times M_y;\mathrm{Mat}(m,\mathbb{C})).
\end{equation}
This fact can be established as follows. 

Let
\begin{equation*}
\label{proof propagator general case 3}
u_\infty(t,x,y):=u(t,x,y)-\widetilde u(t,x,y).
\end{equation*}
From the construction algorithm, we know that
\begin{equation}
\label{proof propagator general case 4}
\left[
\left(
-i\,\frac{\partial}{\partial t}
+
A^{(x)}
\right)
u_\infty
\right]
(t,x,y)
=
f(t,x,y),
\end{equation}
\begin{equation}
\label{proof propagator general case 5}
u_\infty(0,x,y)
=
\zeta(x,y),
\end{equation}
where $f\in C^\infty(\mathbb{R}\times M_x\times M_y;\mathrm{Mat}(m,\mathbb{C}))$ and $\zeta\in C^\infty(M_x\times M_y;\mathrm{Mat}(m,\mathbb{C}))$.
Here the superscript in $A^{(x)}$ indicates that the differential operator $A$ acts in the variable~$x$.
Using functional calculus, we can write the functions $u_\infty$, $f$ and $\zeta$ in terms of the eigenfunctions of $A$ as
\begin{equation}
\label{proof propagator general case 6}
u_\infty(t,x,y)=\sum_{j,k}a_{jk}(t)\,v_j(x)\,[v_k(y)]^*,
\end{equation}
\begin{equation}
\label{proof propagator general case 7}
f(t,x,y)=\sum_{j,k}b_{jk}(t)\,v_j(x)\,[v_k(y)]^*,
\end{equation}
\begin{equation}
\label{proof propagator general case 8}
\zeta(x,y)=\sum_{j,k}c_{jk}\,v_j(x)\,[v_k(y)]^*.
\end{equation}
Here the smooth functions $b_{jk}$ and the constants $c_{jk}$ are given, whereas the functions $a_{jk}$ are our unknowns. Substituting \eqref{proof propagator general case 6}--\eqref{proof propagator general case 8} into \eqref{proof propagator general case 4}, \eqref{proof propagator general case 5} we obtain the family of first order ODEs
\begin{equation*}
\label{proof propagator general case 9}
\left[
\left(
-i\,\frac{d}{d t}
+
\lambda_j
\right)
a_{jk}
\right]
(t)
=
b_{jk}(t),
\end{equation*}
\begin{equation*}
\label{proof propagator general case 10}
a_{jk}(0)
=
c_{jk},
\end{equation*}
whose solutions are
\begin{equation}
\label{proof propagator general case 11}
a_{jk}(t)
=
e^{-i\lambda_jt}
\left(
c_{jk}
+i
\int_0^t
e^{i\lambda_js}
\,b_{jk}(s)\,ds
\right).
\end{equation}
\color{black}

Let $\widehat\zeta$ be the operator with integral kernel $\zeta(x,y)$,
\[
\widehat\zeta:
\ v(x)\mapsto\int_M\zeta(x,y)\,v(y)\,\rho(y)\,dy\,.
\]
Then
\begin{equation}
\label{cjk}
c_{jk}
=
\langle v_j,\widehat\zeta v_k\rangle
=
\frac{1}{\lambda_j^l\lambda_k^n}
\langle A^l v_j,\widehat\zeta A^n v_k\rangle
=
\frac{1}{\lambda_j^l\lambda_k^n}
\langle v_j,(A^l\widehat\zeta A^n)v_k\rangle.
\end{equation}
The operator $A^l\widehat\zeta A^n$ is a pseudodifferential operator of order $-\infty$, so formula \eqref{cjk} and the fact that $\lambda_k \sim k^{1/d}$ when $k\to\infty$ allow one to conclude that the $c_{jk}$ decay faster than any power of $j$ and $k$ as $j,k\to \infty$. A similar argument shows that the $b_{jk}(t)$
and their time derivatives decay faster than any power of $j$ and $k$ as $j,k\to \infty$ uniformly over any bounded open interval in $\mathbb{R}$.
Formula \eqref{proof propagator general case 11} now tells us that the same is true for the $a_{jk}(t)$.
\color{black}
This\color{black}
, in turn,
\color{black}
implies that the series on the RHS of \eqref{proof propagator general case 6} defines a function $u_\infty(t,x,y)$ which is smooth in all variables. So we arrive at \eqref{u agrees with tilde u}, which gives us assumption (a) in Theorem~\ref{theorem positive and negative propagator abstract approach} with $(T_-,T_+)=\mathbb{R}$.

Resorting to standard stationary phase arguments -- see, e.g., \cite[Appendix~C]{SaVa} -- and using the properties (i)--(iv) of our phase functions, it is easy to see that ${u}^\pm$ and $\widetilde{u}^\pm$ satisfy assumptions (b) and (c) of Theorem~\ref{theorem positive and negative propagator abstract approach}. Hence, Theorem~\ref{theorem positive and negative propagator abstract approach} gives us \eqref{propagator approximated general theorem positive} and \eqref{propagator approximated general theorem negative}.

The fact that the construction is global in time is guaranteed by \cite[Lemma~1.2]{LSV}.
\end{proof}

\begin{remark}
If one is prepared to give up globality in time, Theorem~\ref{theorem propagator for A} and the corresponding proof can be adapted in a straightforward manner to the more customary case of real-valued -- as opposed to complex-valued -- phase functions. This is achieved by prescribing the phase functions to take values in $\mathbb{R}$, dropping condition (iv) and replacing everywhere in the statement and in the proof the time domain $\mathbb{R}$ with the interval $(T_-,T_+)$, where
\begin{equation}
\label{definition of time T plus general case}
T_+:=\min_j\,\inf\{t>0\,|\,\left.\det \varphi^{(j)}_{x^\alpha \eta_\beta}\right|_{x=x^{(j)}}=0,\ (y,\eta)\in T'M\}\,,
\end{equation}
\begin{equation}
\label{definition of time T minus general case}
T_-:=\max_j\,\sup\{t<0\,|\,\left.\det \varphi^{(j)}_{x^\alpha \eta_\beta}\right|_{x=x^{(j)}}=0,\ (y,\eta)\in T'M\}\,.
\end{equation}
The values of $T_\pm$ depend on the choice of particular real-valued phase functions,
but we always have $T_-<0<T_+\,$.
Observe that Theorem~\ref{theorem positive and negative propagator abstract approach} was formulated in such a way that it covers both the case of real-valued and complex-valued phase functions.
\end{remark}

The reader will have noticed that the zero mode propagator $U^0_A$ does not appear in our construction. This is due to the fact that, clearly, 
\[
U^0_A\overset{\mod \Psi^{-\infty}}{=}0.
\]

We end this subsection with the observation that,
thanks to the presence of the weight $w^{(j)}$
in formula \eqref{oscillatory intergral for U^(j)}, the scalar matrix-function $\mathfrak{a}^{(j)}_0$ does not depend on the choice of the phase functions $\varphi^{(j)}$. This motivates the following definition.

\begin{definition}
\label{definition principal symbol general case}
We call $\mathfrak{a}^{(j)}_0$ the \emph{principal symbol} of the Fourier integral operator \eqref{oscillatory intergral for U^(j)}.
\end{definition}

The above definition agrees with the standard definition of principal symbol of a Fourier integral operator expressed as a section of the Keller--Maslov bundle, see \cite[subsection~2.4]{LSV}.

\subsection{The algorithm}
\label{The algorithm}

The integral kernel \eqref{integral kernel of U^(j)}
of $U^{(j)}_A(t)$ can be constructed explicitly as follows.

\

{\textbf{Step 1}.} Choose a phase function $\varphi^{(j)}$ compatible with Theorem~\ref{theorem propagator for A}. We will see later on that for the special case of the Dirac operator we can identify a distinguished phase function, the \emph{Levi-Civita phase function}. Furthermore, set $\chi^{(j)}\equiv 1$. In fact, the purpose of the cut-off is to localise integration in a neighbourhood of the $h^{(j)}$-flow and away from the zero section: different choices of $\chi^{(j)}$ result in oscillatory integrals differing by an infinitely smooth function.

\

{\textbf{Step 2}.} Act with the operator $-i \partial_t+A^{(x)}$ on the oscillatory integral \eqref{integral kernel of U^(j)}. This produces a new oscillatory integral
\begin{equation}
\label{algorithm formula 1}
\frac{1}{(2\pi)^d}\int_{T'_yM} e^{i \varphi^{(j)}(t,x;y,\eta)}\,a^{(j)}(t,x;y,\eta)
\,w^{(j)}(t,x;y,\eta)\,d\eta
\end{equation}
whose amplitude $a^{(j)}\in C^\infty(\mathbb{R}\times M \times T'M;\mathrm{Mat}(m,\mathbb{C}))$ is given by
\begin{equation*}
\label{algorithm formula 2}
a^{(j)}:=e^{-i \varphi^{(j)}} [w^{(j)}]^{-1} \left(-i\partial_t+A^{(x)}\right)\left(e^{i \varphi^{(j)}}\,\mathfrak{a}^{(j)} \,w^{(j)} \right).
\end{equation*}

By making use of the fact that $\varphi^{(j)}$ and $w^{(j)}$ are positively homogeneous in momentum $\eta$ of degree 1 and 0, respectively, one can write down an asymptotic expansion for the amplitude $a^{(j)}$ in components positively homogeneous in momentum:
\begin{equation*}
\label{algorithm formula 4}
a^{(j)}(t,x;y,\eta)\sim \sum_{k=-1}^{+\infty} a^{(j)}_{-k}(t,x;y,\eta), \quad  a^{(j)}_{-k}(t,x;y,\alpha\,\eta)=\alpha^{-k}\, a^{(j)}_{-k}(t,x;y,\eta), \quad \forall\alpha>0.
\end{equation*}

\

{\bf Step 3}. As $u^{(j)}(t,x,y)$ is to be the (distributional) solution of the hyperbolic equation
\[
(-i \partial_t +A^{(x)})u^{(j)}(t,x,y)\overset{\mod C^\infty}{=}0,
\]
one would like to impose the condition $a^{(j)}(t,x,y,\eta)=0$. However, the amplitude $a^{(j)}$, unlike the symbol $\mathfrak{a}^{(j)}$, depends on $x$, and doing so would result in an unsolvable system of partial differential equations (PDEs). The current step consists in excluding the dependence of $a^{(j)}$ on $x$ by means of a procedure known as \emph{reduction of the amplitude}, to the end of reducing the system of PDEs to a system of ordinary differential equations instead.

Put\footnote{
Here $(\varphi^{(j)}_{x\eta})^{-1}$ is defined in accordance with
$
[(\varphi^{(j)}_{x\eta})^{-1}]_\alpha{}^\beta \, \varphi^{(j)}_{x^\beta \eta_\gamma} =\delta_\alpha{}^\gamma.
$
}
\begin{equation*}\label{operator L with j}
L^{(j)}_\alpha:=\left[(\varphi^{(j)}_{x\eta})^{-1}\right]_\alpha{}^\beta\,\dfrac{\partial}{\partial x^\beta}
\end{equation*}
and define
\begin{subequations}\label{operators mathfrak S with j}
\begin{gather}
\mathfrak{S}_0^{(j)}:=\left.\left( \,\cdot\, \right)\right|_{x=x^{(j)}}\,, \label{mathfrak S0 with j}\\
\mathfrak{S}_{-k}^{(j)}:=\mathfrak{S}_0^{(j)} \left[ i \, [w^{(j)}]^{-1} \frac{\partial}{\partial \eta_\beta}\, w^{(j)} \left( 1+ \sum_{1\leq |\boldsymbol{\alpha}|\leq 2k-1} \dfrac{(-\varphi^{(j)}_\eta)^{\boldsymbol{\alpha}}}{\boldsymbol{\alpha}!\,(|\boldsymbol{\alpha}|+1)}\,L^{(j)}_{\boldsymbol{\alpha}} \right) L^{(j)}_\beta  \right]^k\,,\label{mathfrak Sk with j}
\end{gather}
\end{subequations}
where $\boldsymbol{\alpha}\in \mathbb{N}^d$, $|\boldsymbol{\alpha}|=\sum_{j=1}^d \alpha_j$ and $(-\varphi^{(j)}_\eta)^{\boldsymbol{\alpha}}:=(-1)^{|\boldsymbol{\alpha}|}\, (\varphi^{(j)}_{\eta_1})^{\alpha_1}\dots ( \varphi^{(j)}_{\eta_d})^{\alpha_d}$. The operator
\eqref{mathfrak Sk with j} is well defined,
because the differential operators $L^{(j)}_\alpha$ commute \textcolor{black}{\cite[Lemma~A.2]{wave}}. Furthermore, the operators $\mathfrak{S}^{(j)}_{-k}$ are invariant under change of local coordinates $x$ and $y$.

\begin{remark}
\label{remark Lalpha}
Let $f:\mathbb{R}^d\to \mathbb{R}^d$, $x\mapsto \widetilde x$, be a (locally) invertible map.  Then the operators
\[
\widetilde L_\alpha:=[(\nabla_x f)^{-1}]_\alpha{}^\beta \dfrac{\partial}{\partial x^\beta}
\]
are the pushforward of partial derivatives $\partial/\partial \widetilde{x}^\alpha$ along $f^{-1}$. Hence, the operators $\widetilde L_\alpha$ commute because the partial derivatives $\partial/\partial \widetilde{x}^\alpha$ commute.  An adjustment of the above argument to our setting with $f=\varphi^{(j)}_\eta$ (and with account of the fact that $\varphi^{(j)}$ is complex-valued) provides an alternative explanation for the commutation of the operators $L^{(j)}_\alpha$.
\end{remark}

The \emph{amplitude-to-symbol operator} is defined as 
\begin{gather}
\mathfrak{S}^{(j)}:C^\infty(\mathbb{R}\times M \times T'M)\to C^\infty(\mathbb{R} \times T'M)\,,
\nonumber
\\
\label{definition amplitude to symbol with j}
\mathfrak{S}^{(j)}:=\sum_{j=0}^\infty \mathfrak{S}^{(j)}_{-k}\,.
\end{gather}

When acting on a function positively homogeneous in momentum, the operator $\mathfrak{S}^{(j)}_{-k}$ excludes the dependence on $x$ and decreases the degree of homogeneity by $k$.

The reduction of the amplitude is achieved by replacing the amplitude $a^{(j)}$ in \eqref{algorithm formula 1}
by
\begin{equation*}
\label{algorithm formula 5}
\mathfrak{S}^{(j)}a^{(j)}=:\mathfrak{b}^{(j)},
\end{equation*}
with
\[
\mathfrak{b}^{(j)}(t;y,\eta)\sim \sum_{k=-1}^{+\infty} \mathfrak{b}^{(j)}_{-k}(t;y,\eta)\,, \qquad \mathfrak{b}^{(j)}_{-k}=\sum_{l+s=k} \mathfrak{S}^{(j)}_{-l} \,a^{(j)}_{-s}\,.
\]
The oscillatory integral 
\begin{equation*}
\label{algorithm formula 6}
\frac{1}{(2\pi)^d}\int_{T'_yM} e^{i \varphi^{(j)}(t,x;y,\eta)}\,\mathfrak{b}^{(j)}(t;y,\eta)
\,w^{(j)}(t,x;y,\eta)\,d\eta
\end{equation*}
differs from \eqref{algorithm formula 1} only by an infinitely smooth function.

We refer the reader to \cite[Appendix~A]{wave} for further particulars and detailed proofs concerning the amplitude-to-symbol operator.

\

{\bf Step 4}. Set
\begin{equation}
\label{algorithm formula 7}
\mathfrak{b}^{(j)}_{-k}=0, \qquad k=-1,0,1,\ldots.
\end{equation}
Equations \eqref{algorithm formula 7}, combined with the initial conditions stemming from the constraint
\begin{equation}
\label{algorithm formula 8}
\sum_{j} U^{(j)}(0)\overset{\mod \Psi^{-\infty}}{=}\mathrm{Id},
\end{equation}
yield a hierarchy of (matrix) transport equations for the homogeneous components $\mathfrak{a}^{(j)}_{-k}$.

\begin{remark}
For the special case of the massless Dirac operator, the first few equations in the hierarchy \eqref{algorithm formula 7} are given by 
\eqref{transport equations formula 4}--\eqref{third transport equation dirac}.
\end{remark}

\

Let us make a few remarks warranted by formula \eqref{algorithm formula 8}.

The $m$ oscillatory integrals appearing on the RHS of \eqref{propagator approximated general theorem positive} and \eqref{propagator approximated general theorem negative} are not independent of one another, but they `mix' at $t=0$ via the initial condition \eqref{algorithm formula 8}. Now, satisfying \eqref{algorithm formula 8} involves representing the identity operator on $C^\infty(M;\mathbb{C}^m)$ in a somewhat nonstandard fashion, as
\begin{equation}
\label{indentity operator oscillatory integral general case}
\mathrm{Id}\overset{\mod{\Psi^{-\infty}}}{=}\sum_j \frac{1}{(2\pi)^d}\int_{T'M} e^{i\varphi^{(j)}(0,x;y,\eta)} \,\mathfrak{s}^{(j)}(y,\eta)\,\chi^{(j)}(0,x;y,\eta)\,w^{(j)}(0,x;y,\eta)\,(\,\cdot\,)\,\rho(y)\,dy\,d\eta,
\end{equation}
with $\mathfrak{s}^{(j)}\in S^0_{\mathrm{ph}}(T'M;\mathrm{Mat}(m;\mathbb{C}))$.

In terms of the symbols $\mathfrak{a}^{(j)}$, the initial condition \eqref{algorithm formula 8} reads 
\begin{equation*}
\mathfrak{a}^{(j)}(0;y,\eta)=\mathfrak{s}^{(j)}(y,\eta).
\end{equation*}
From the fact that the principal symbol of the identity operator is the identity matrix it follows that
\begin{equation}
\label{identity operator principal symbol general case}
\sum_j \mathfrak{a}^{(j)}_0(0;y,\eta)=\sum_j \mathfrak{s}^{(j)}_0(y,\eta)=\mathbf{1}_{m\times m}.
\end{equation}
Furthermore, one can show that
\[
\mathfrak{s}^{(j)}_0(y,\eta)=v^{(j)}(y,\eta)\,[v^{(j)}(y,\eta)]^*.
\]
However, obtaining formulae for subleading components $\mathfrak{s}^{(j)}_{-1}$ is already a challenging task, see \cite[subsection~4.2]{CDV}. In general, lower order components of $\mathfrak{s}^{(j)}$ depend in a nontrivial manner on the eigenvalues and eigenprojections of the matrix-function $A_\mathrm{prin}(x,\xi)$ and on the choice of phase functions~$\varphi^{(j)}$.

The invariant representation of the identity operator -- and, more generally, of pseudo\-differential operators -- on manifolds is not a well-studied subject. An initial analysis of the scalar case was carried out in \cite[Section~6]{wave}. For the case of the Dirac operator a more detailed examination of \eqref{indentity operator oscillatory integral general case} will be provided in subsection~\ref{identity operator dirac}.  \textcolor{black}{A more extensive analysis of \eqref{indentity operator oscillatory integral general case} for a general operator $A$ is carried out in \cite{part1, part2,diagonalisation}.}

\begin{remark}
\label{remark about pseudodifferential}
All statements and results presented in this section carry over verbatim to the case where $A$ is an elliptic symmetric first order $m\times m$ matrix \emph{pseudodifferential} -- as opposed to differential -- operator, with the following exceptions:
\begin{itemize}
\item 
formulae \eqref{m plus and m minus equal m/2} and \eqref{symmetry eigenvalues and eigenvectors} have to be dropped as they are no longer true;

\item 
`Step 2.' in subsection \ref{The algorithm} has to be modified to take into account the action of a pseudodifferential operator on an oscillatory integral in an \emph{invariant} manner, along the lines of \cite[Section~4.3]{battistotti}.
\end{itemize}
\end{remark}

\begin{remark}
Let us point out that in this section we did not use anywhere the fact that $M$ carries a Riemannian structure. 
If one replaces the Riemannian density \eqref{riemannian density} with an arbitrary positive density, all statements and results stay the same.
\end{remark}

\section{Invariant description of pseudodifferential operators acting on scalar functions}
\label{Invariant description}

In order to prepare ourselves to address the issue of initial conditions for our transport equations in the case of the Dirac operator, we need to discuss first the more general question of invariant
representation of a pseudodifferential operator. We devote a separate section to this, as we believe this matter to be of independent interest.
Note that we treat the case of a scalar operator merely for the sake of presentational
convenience: all the formulae and arguments in this subsection
remain unchanged for matrix pseudodifferential operators acting
on $m$-columns of scalar functions.

\begin{definition}
We call \emph{time-independent Levi-Civita phase function} the function
$\phi\in C^\infty(M\times T'M; \mathbb{C})$ defined by
\begin{equation}
\label{formula time-independent Levi-Civita phase function}
\phi(x;y,\eta):=\int_{\gamma} \zeta\, dz +\frac{i\epsilon}{2} h(y,\eta)\,\left[ \mathrm{dist}(x,y)\right]^2
\end{equation}
when $x$ lies in a geodesic neighbourhood of $y$ and continued smoothly elsewhere in such a way that $\mathrm{Im}\,\phi\ge 0$.
Here $\gamma$ is the (unique) shortest geodesic connecting $y$ to $x$, $\zeta$ is the parallel transport of $\eta$ along $\gamma$,
\begin{equation}
\label{eigenvalues massless dirac}
h(y,\eta):=\sqrt{g^{\alpha\beta}(y)\,\eta_\alpha\eta_\beta}\,,
\end{equation}
$\mathrm{dist}$ is the geodesic distance and $\epsilon$ is a positive parameter.
\end{definition}

Let $P$ be a pseudodifferential operator of order $p$ acting
on scalar functions over a Riemannian $d$-manifold.
The operator $P$ can be written, modulo an integral operator with smooth kernel, in the form
\begin{equation}
\label{pseudo P invariant}
P=\int_{T'M} e^{i\phi(x;y,\eta)} \,\mathfrak{p}(y,\eta)\,\chi_0(x;y,\eta)\,w_0(x;y,\eta)\,(\,\cdot\,)\,\rho(y)\,dy\,d\eta,
\end{equation}
where $\phi$ is the time-independent Levi-Civita phase function, $\mathfrak{p}\in S^m_{\mathrm{ph}}(T'M)$, $\chi_0$ is a cut-off localising integration to a neighbourhood of the diagonal and away from the zero section (see also (I)--(III) in Theorem~\ref{theorem positive and negative propagator abstract approach}) and 
\begin{equation}
w_0(x;y,\eta):= \left[\rho(x)\,\rho(y)\right]^{-\frac12}\left[{\det}^2 \phi_{x^\alpha\eta_\beta}(x;y,\eta) \right]^\frac14.
\end{equation}
Here the smooth branch of the complex root is chosen in such a way that $w_0(y;y,\eta)=[\rho(y)]^{-1}$.

\begin{remark}
Note that \eqref{pseudo P invariant} is, effectively, a special case of \eqref{oscillatory intergral for U^(j)} with $t=0$.
\end{remark}

Formula \eqref{pseudo P invariant} provides an invariant representation of the pseudodifferential operator~$P$.

\begin{definition}
We call \emph{full symbol} of the operator $P$ the scalar function
\[
\mathfrak{p}(y,\eta)\sim\sum_{k=-p}^{+\infty} \mathfrak{p}_{-k}(y,\eta).
\]
Furthermore, we call the homogeneous functions
$\mathfrak{p}_p$ and $\mathfrak{p}_{p-1}$
the  $g$-\emph{principal} and $g$-\emph{subprincipal} symbol,
respectively\footnote{Here `$g$' is a reference to the Riemannian metric used in the construction of the phase function $\phi$.}.
\end{definition}

The notions of principal and subprincipal symbols of a pseudodifferential operator are nowadays standard concepts in microlocal analysis. The former makes sense for operators acting either on scalar functions or on half-densities, whereas the latter is only defined for operators acting on half-densities. We refer the reader to \cite{Hor} for further details. Note that the concept of subprincipal symbol was introduced by Duistermaat and H\"ormander in \cite[Eqn.~(5.2.8)]{DuHo}.

It is easy to see that the concept of principal symbol $P_\mathrm{prin}$ and that of $g$-principal symbol $\mathfrak{p}_p$ coincide. As far as the subprincipal symbol is concerned, the situation is more complicated, in that before drawing a comparison we need to turn our operator into an operator acting on half-densities.

Put 
\begin{equation}
\label{P one half}
P_{1/2}:= \rho^{1/2}\,P\,\rho^{-1/2} 
\end{equation}
and let $P_\mathrm{sub}$ be the subprincipal symbol of the operator \eqref{P one half} defined in accordance with \cite[Eqn.~(5.2.8)]{DuHo}.

A natural question to ask is: what is the relation between 
$P_\mathrm{sub}$ and $\mathfrak{p}_{p-1}$?

\begin{theorem}
\label{theorem subprincipal symbols for P dirac}
The invariant quantities $P_\mathrm{sub}$ and $\,\mathfrak{p}_{p-1}$ are related as
\begin{equation}
\label{relation between invariant subps}
\mathfrak{p}_{p-1}=P_\mathrm{sub}+\frac{i}{2}(P_\mathrm{prin})_{y^\alpha\eta_\alpha}+\frac{i}{2}\,\Gamma^\alpha{}_{\beta\gamma}\left[ \eta_\alpha(P_\mathrm{prin})_{\eta_\beta}\right]_{\eta_\gamma}- \frac\epsilon2\,g_{\beta\gamma} \left[h \,(P_\mathrm{prin})_{\eta_\beta} \right]_{\eta_\gamma}.
\end{equation}
\end{theorem}

Theorem~\ref{theorem subprincipal symbols for P dirac} implies that,
in particular, the two notions of subprincipal symbol coincide when
the principal symbol does not depend on $\eta$, i.e.~when
$P$ is a pseudodifferential operator of the type
``multiplication by a scalar function plus an operator of order $-1$''.
Note that the identity operator, whose invariant representation was investigated in
\cite[Section~6]{wave},  falls into this class.

\begin{remark}
A tedious, yet straightforward, calculation shows that the RHS of \eqref{relation between invariant subps} is a scalar 
function on the cotangent bundle.
In fact, the second and third summands on the RHS of \eqref{relation between invariant subps}
admit an invariant representation in terms of the Laplace--Beltrami operator associated with the \emph{neutral metric} $n$ on the cotangent bundle $T^*M$, which, in local coordinates $(x^1, \ldots,x^d, \xi_1, \ldots, \xi_d)$, reads
\begin{equation}
n_{jk}(x,\xi)=
\begin{pmatrix}
-2\, \xi_\gamma \, \Gamma^\gamma{}_{\alpha\beta}(x)  & \delta_\alpha{}^\mu\\
\delta^\nu{}_\beta & 0
\end{pmatrix}, \qquad j,k \in \{1,\ldots, 2d\}.
\end{equation}
The adjective `neutral' refers to the fact that the metric $n$ has signature $(d,d)$. It turns out that the neutral metric is an effective tool in the development of an invariant theory
of pseudodifferential operators on Riemannian manifolds.
As the analysis of this matter requires a lengthy discussion and
would take us away from the core subject of our paper, 
we plan to address it in detail elsewhere. See also \cite{sharafutdinov}.
\end{remark}

\begin{proof}[Proof of Theorem~\ref{theorem subprincipal symbols for P dirac}]
Consider the pseudodifferential operator $P$ and turn it into an operator on half-densities $P_{1/2}$ via \eqref{P one half}. In what follows we work in an arbitrary coordinate system, the same for $x$ and $y$.

Dropping the cut-off, the integral kernel of $P_{1/2}$ now reads
\begin{equation}
\label{proof general subprincipal dirac temp 1}
\frac1{(2\pi)^d} \int_{T'_yM} e^{i\phi(x;y,\eta)} \,\mathfrak{p}(y,\eta)\, \sqrt{\det\phi_{x\eta}}\,d\eta\,.
\end{equation}
Our phase function \eqref{formula time-independent Levi-Civita phase function} admits the expansion
\begin{equation}
\label{proof general subprincipal dirac temp 2}
\phi(x;y,\eta)=(x-y)^\alpha\eta_\alpha
+\frac12\Gamma^\alpha{}_{\beta\gamma}\,\eta_\alpha(x-y)^\beta (x-y)^\gamma
+\frac{i\epsilon h}2g_{\alpha\beta}(x-y)^\alpha(x-y)^\beta + O(\|x-y\|^3),
\end{equation}
which implies that
\begin{equation}
\label{proof general subprincipal dirac temp 3}
\sqrt{\det\phi_{x\eta}}=1+
\frac12[
\Gamma^\alpha{}_{\alpha\beta}+i\epsilon h^{-1}\eta_\beta
](x-y)^\beta+O(\|x-y\|^2).
\end{equation}
Substituting \eqref{proof general subprincipal dirac temp 2} and \eqref{proof general subprincipal dirac temp 3} into \eqref{proof general subprincipal dirac temp 1}, we get
\begin{multline}
\label{proof general subprincipal dirac temp 4}
\frac{1}{(2\pi)^d}\int e^{i(x-y)^\alpha\eta_\alpha}
\Bigl\{
\mathfrak{p}_p
\\
+\left(
\frac12
\left[
i\Gamma^\alpha{}_{\beta\gamma}\,\eta_\alpha
-\epsilon hg_{\beta\gamma}
\right]
(x-y)^\beta (x-y)^\gamma
+\frac12
\left[
\Gamma^\alpha{}_{\alpha\beta}+i\epsilon h^{-1}\eta_\beta
\right]
(x-y)^\beta
 \right)
 \mathfrak{p}_p
 \\
+\mathfrak{p}_{p-1}
+O(\|\eta\|^{p-2})
\Bigr\}\,
d\eta\,.
\end{multline}
Excluding the $x$-dependence from the amplitude in \eqref{proof general subprincipal dirac temp 4} by acting with the operator
\begin{equation}
\label{operator Melrose}
\mathcal{S}_\mathrm{right}(\,\cdot\,):=\left.\left[\exp\left(i\frac{\partial^2}{\partial x^\mu\,\partial \eta_\mu}\right)(\,\cdot\,)\right]\right|_{x=y},
\end{equation}
we arrive at
\begin{multline}
\label{proof general subprincipal dirac temp 5}
\frac{1}{(2\pi)^d}\int e^{i(x-y)^\alpha\eta_\alpha}
\Bigl\{
\mathfrak{p}_p
\\
-\frac{i}{2}\left[ \eta_\alpha \,\Gamma^\alpha{}_{\beta\gamma}\,(\mathfrak{p}_p)_{\eta_\beta}\right]_{\eta_\gamma}+\frac\epsilon2 \left[h \,g_{\gamma\beta}\,(\mathfrak{p}_p)_{\eta_\beta} \right]_{\eta_\gamma}
 \\
+\mathfrak{p}_{p-1}
+O(\|\eta\|^{p-2})
\Bigr\}\,
d\eta\,.
\end{multline}
Computing the subprincipal symbol of \eqref{proof general subprincipal dirac temp 5} and using the fact that
$\mathfrak{p}_p=P_\mathrm{prin}=(P_{1/2})_\mathrm{prin}\,$,
we obtain \eqref{relation between invariant subps}. Note that the sign in front of the correction term
\[
\frac{i}{2}(P_\mathrm{prin})_{y^\alpha\eta_\alpha}
\]
is opposite to the usual one, see, for example, \cite[Eqn.~(A.3)]{spin2}. This is due to the fact that in this paper we use the right -- as opposed to left -- quantization.
\end{proof}

\section{Global propagator for the Dirac operator}
\label{Global propagator for the massless Dirac operator}

In this section we will start the analysis of the global propagator for the Dirac operator, specialising Theorem~\ref{theorem propagator for A} to the case $A=W$.

We denote by
\begin{equation}
\label{principal symbol dirac definition}
W_\mathrm{prin}(y,\eta):=\sigma^\alpha(y)\,\eta_\alpha
\end{equation}
the principal symbol of $W$
and by
\begin{equation}
\label{zero order part dirac definition}
W_0(x):=-\frac{i}4\sigma^\alpha(x)\sigma_\beta(x)
\left(
\frac{\partial\sigma^\beta}{\partial x^\alpha}(x)
+\Gamma^\beta{}_{\alpha\gamma}(x)\,\sigma^\gamma(x) \right)
\end{equation}
its zero order part, see Definition~\ref{massless dirac definition}.

The principal symbol $W_\mathrm{prin}(y,\eta)$ has eigenvalues $h^\pm=\pm h$, where
$h$ is given by \eqref{eigenvalues massless dirac},
compare with \eqref{symmetry eigenvalues and eigenvectors}.
This fact, which can be easily established by writing down \eqref{principal symbol dirac definition} in local coordinates, shows that the Dirac operator is indeed elliptic.

It is well-known that the Hamiltonian flow $(x^+(t;y,\eta),\xi^+(t;y,\eta))$ generated by $h$ is (co-)geodesic. The two flows $(x^\pm(t;y,\eta),\xi^\pm(t;y,\eta))$ are related as
\begin{equation}
\label{relation between two geodesic flows}
(x^-(t;y,\eta),\xi^-(t;y,\eta))
=
(x^+(t;y,-\eta),-\xi^+(t;y,-\eta)).
\end{equation}

Our goal is to write down explicitly the positive and negative propagators \eqref{dirac propagator definition equation detailed positive} and \eqref{dirac propagator definition equation detailed negative} in the form \eqref{oscillatory intergral for U^(j)} for a distinguished choice of phase functions. 

To this end, we give the following definition (see also \cite[Section~4]{wave}).

\begin{definition}
\label{definition LC phase functions with epsilon dirac}
We call \emph{positive} ($+$), resp.~\emph{negative} ($-$), \emph{Levi-Civita phase function} the infinitely smooth function $\varphi^{\pm}\in C^\infty(\mathbb{R}\times M \times T'M; \mathbb{C})$ defined by
\begin{equation}
\label{definition positive negative Levi-Civita phase function}
\varphi^\pm(t,x;y,\eta)=\int_{\gamma^\pm} \zeta^\pm \,dz+\frac{i\,\epsilon}{2}h(y,\eta) \dist^2(x,x^\pm(t;y,\eta))
\end{equation}
for $x$ in a geodesic neighbourhood of $x^\pm(t;y,\eta)$ and continued smoothly elsewhere in such a way that $\operatorname{Im} \varphi^\pm\ge0$. Here $\dist$ is the Riemannian geodesic distance, the path of integration $\gamma^\pm$ is the shortest geodesic connecting $x^\pm$ to $x$, $\zeta^\pm$ is the result of parallel transport of $\xi^\pm(t;y,\eta)$ along $\gamma^\pm$ and $\epsilon$ is a positive parameter.
\end{definition}

The positive and negative Levi-Civita phase functions are related as
\begin{equation}
\label{relation between the two Levi-Civita phase functions}
\varphi^-(t,x;y,\eta)
=
-
\overline
{
\varphi^+(t,x;y,-\eta)
}.
\end{equation}

Let us point out that the way one continues $\varphi^\pm$ outside a neighbourhood of the flow does not affect the singular part of the propagators. The choice of a different smooth continuation results in an error $\overset{\mod \Psi^{-\infty}}{=}0$, as one can show by a straight\-forward (non)stationary phase argument.

\begin{remark}
The time-independent phase function $\phi$ introduced in the previous section is the restriction to $t=0$
of the phase functions $\varphi^\pm$,
\begin{equation}
\phi(x;y,\eta)=\varphi^+(0,x;y,\eta)=\varphi^-(0,x;y,\eta).
\end{equation}
\end{remark}

It is easy to see that the positive and negative Levi-Civita phase functions satisfy conditions (i), (ii) and (iv) from Theorem~\ref{theorem propagator for A}.
Furthermore, \cite[Corollary 2.4.5]{SaVa} implies that condition (iii) is also
satisfied.
Hence, Theorem~\ref{theorem propagator for A} ensures that the integral kernel of $U^{\pm}$ can be written as a single oscillatory integral 
\begin{equation}
\label{integral kernel of U pm}
u^\pm(t,x,y):=\frac{1}{(2\pi)^3}\int_{T'_yM} e^{i \varphi^\pm(t,x;y,\eta)}\,\mathfrak{a}^\pm(t;y,\eta)\,\chi^\pm(t,x;y,\eta)\,w^\pm(t,x;y,\eta)\,d\eta,
\end{equation}
where $\varphi^\pm$ is the positive/negative Levi-Civita phase function.


\begin{definition}
\label{definition full symbols dirac propagator}
We define \emph{the full symbol of the positive} (resp.~\emph{negative}) \emph{propagator} to be the scalar matrix-function $\mathfrak{a}^+$ (resp.~$\mathfrak{a}^-$), obtained through the algorithm described in Section~\ref{The algorithm} with Levi-Civita phase functions.

We define the \emph{subprincipal symbol of the positive} (resp.~\emph{negative}) \emph{propagator} to be the scalar matrix-function $\mathfrak{a}^+_{-1}$ (resp.~$\mathfrak{a}^-_{-1}$) obtained the same way.
\end{definition}

As to the principal symbol, this object was defined earlier,
see Definition~\ref{definition principal symbol general case}.

We stress that the mathematical objects contained in the above definition are uniquely and invariantly defined. They only depend on the phase functions which, in turn, originate from the geometry of $M$ in a coordinate-free covariant manner, cf.~Definition~\ref{definition LC phase functions with epsilon dirac}.

To the best of our knowledge, there is no accepted definition of full symbol or subprincipal symbol for a Fourier integral operator available in the literature to date. The geometric nature of our construction allows us to provide invariant definitions of full and subprincipal symbol of the Dirac propagator, analyse them, and give explicit formulae. This paper, alongside \cite{wave}, aims to build towards an invariant theory for pseudodifferential and Fourier integral operators on manifolds.

\

Before moving on to computing the principal and subprincipal symbols of the positive (resp.~negative) Dirac propagator, an important remark is in order. In addition to what was discussed in Section~\ref{Preliminary results for general first order systems} for the general case, the construction of the Dirac propagator has to be consistent with the gauge transformation \eqref{special unitary matrix-function}, \eqref{transformation of the Dirac operator under change of frame}. In particular, the action of the gauge transformation needs to be carefully accounted for by the construction process.

The transformation \eqref{transformation of the Dirac operator under change of frame} leads to the transformation 
\begin{equation*}
\label{gauge transformation at the level of the symbol}
\mathfrak{a}^\pm(t;y,\eta)\mapsto G^*(x)\, \mathfrak{a}^\pm(t;y,\eta)\, G(y).
\end{equation*}
in the oscillatory integral \eqref{integral kernel of U pm}.
Note that this introduces an $x$-dependence
which has to be handled by means of amplitude-to-symbol reduction \eqref{definition amplitude to symbol with j}.

\subsection{Transport equations}
\label{Transport equations dirac}

By acting with the Dirac operator $W$ on \eqref{integral kernel of U pm} in the variable $x$ and dropping the cut-off, we obtain
\begin{equation*}
\label{transport equations formula 1}
W u^\pm(t,x,y)=\frac{1}{(2\pi)^3}\int_{T'_yM} e^{i \varphi^\pm(t,x;y,\eta)}\,a^\pm(t;y,\eta)\,w^\pm(t,x;y,\eta)\,d\eta,
\end{equation*}
where
\begin{equation*}
\label{transport equations formula 2}
\begin{split}
a&=-i e^{-i\varphi^\pm} (w^\pm)^{-1} \partial_t \left(e^{i\varphi^\pm}\,\mathfrak{a}^\pm  \,w^\pm\right)+ \left[-i e^{-i\varphi^\pm} (w^\pm)^{-1}\sigma^\alpha \partial_{x^\alpha} \left(e^{i\varphi^\pm} \,w^\pm \right) + W_0\right]\mathfrak{a}^\pm\\
&= \left(\varphi^\pm_t+\sigma^\alpha \varphi^\pm_{x^\alpha}\right)\mathfrak{a}^\pm -i \mathfrak{a}^\pm_t+ \left[-i (w^\pm)^{-1}\left( w^\pm_t+\sigma^\alpha w^\pm_{x^\alpha} \right)+W_0 \right]\mathfrak{a}^\pm.
\end{split}
\end{equation*}
Put
\begin{equation}
\label{transport equations formula 3}
a\sim\sum_{k=-1}^{+\infty} a_{-k},
\end{equation}
where 
\begin{equation}
\label{transport equations formula 4}
a^\pm_1:=\left(\varphi^\pm_t+W_\mathrm{prin}(x,\varphi_x^\pm)\right)\mathfrak{a}^\pm_0
\end{equation}
and
\begin{equation}
\label{transport equations formula 5}
a^\pm_{-k}:= \left(\varphi^\pm_t+W_\mathrm{prin}(x,\varphi_x^\pm)\right)\mathfrak{a}^\pm_{-k-1} -i (\mathfrak{a}^\pm_{-k})_t+ \left[-i (w^\pm)^{-1}\left( w^\pm_t+\sigma^\alpha w^\pm_{x^\alpha} \right)+W_0 \right]\mathfrak{a}^\pm_{-k}
\end{equation}
for $k\ge 0$.
Note that the $a^\pm_{-k}$, $k\ge -1$, are positively homogeneous in momentum of degree $-k$.

Our transport equations read
\begin{gather}
\mathfrak{S}^\pm_0 a^\pm_1=0,
\label{first transport equation dirac}
\\
\mathfrak{S}^\pm_{-1} a^\pm_1+ \mathfrak{S}^\pm_{0}a^\pm_0=0,
\label{second transport equation dirac}
\\
\mathfrak{S}^\pm_{-2} a^\pm_1+ \mathfrak{S}^\pm_{-1}a^\pm_0+\mathfrak{S}^\pm_0 a^\pm_{-1}=0,
\label{third transport equation dirac}
\\\nonumber
\dots
\end{gather}

Recalling that $v^\pm$ are the normalised eigenvectors of $W_\mathrm{prin}$ corresponding to the eigenvalues $\pm h$, denote by
\begin{equation}
\label{spectral projections principal symbol}
P^\pm(y,\eta):= v^\pm(y,\eta)\,[v^\pm(y,\eta)]^*
\end{equation}
the spectral projections along the eigenspaces spanned by $v^\pm$.
Of course,
\begin{equation}
\label{spectral decomposition principal symbol dirac}
W_\mathrm{prin}=h\, (P^+-P^-),
\end{equation}
\begin{equation}
\label{spectral decomposition identity dirac}
\mathrm{Id}=P^++P^-,
\end{equation}
and
\begin{equation}
\label{projections in terms of principal symbol dirac}
P^\pm=\frac12\left( \mathrm{Id}\pm \frac{W_\mathrm{prin}}{h} \right).
\end{equation}

Let us label the transport equations with nonnegative integer numbers in increasing order, so that \eqref{first transport equation dirac} is the zeroth transport equation, \eqref{second transport equation dirac} is the first transport equation and so on. Direct inspection of \eqref{transport equations formula 4} and \eqref{transport equations formula 5} reveals that

\begin{itemize}
\item multiplication of the $n$-th transport equation by $P^\mp(x^\pm, \xi^\pm)$ on the left allows one to determine
\begin{equation}
\label{transport equations formula 6}
P^{\mp}(x^\pm,\xi^\pm)\mathfrak{a}^\pm_{-n}(t;y,\eta), 
\qquad n\ge0,
\end{equation}
algebraically;

\item multiplication of the $(n+1)$-th transport equation by $P^\pm(x^\pm, \xi^\pm)$ on the left and the use of \eqref{transport equations formula 6} allows one to determine
\begin{equation}
\label{transport equations formula 7}
P^{\pm}(x^\pm,\xi^\pm)\mathfrak{a}^\pm_{-n}(t;y,\eta), 
\qquad n\ge0,
\end{equation}
upon solving a matrix ordinary differential equation in the variable $t$.
\end{itemize}
Summing up \eqref{transport equations formula 6} and \eqref{transport equations formula 7} one obtains $\mathfrak{a}^\pm_{-k}(t;y,\eta)$, in view of \eqref{spectral decomposition identity dirac}.

\subsection{Pseudodifferential operators $U^\pm(0)$}
\label{identity operator dirac}

This subsection is devoted to the examination of operators $U^\pm(0)$.
We need to examine these operators because,
as explained in subsection~\ref{The algorithm}, their full symbols determine the initial conditions
$\mathfrak{a}^\pm_{-k}(0;y,\eta)$ for our transport equations.

We have
\begin{equation}
\label{4 December 2019 equation 1}
U^\pm(0)=\theta(\pm W),
\end{equation}
where
\begin{equation*}
\label{4 December 2019 equation 2}
\theta(\lambda)
:=
\begin{cases}
1\quad\text{for}\quad\lambda>0,
\\
0\quad\text{for}\quad\lambda\le0.
\end{cases}
\end{equation*}
We see that the operators $U^\pm(0)$ are self-adjoint pseudodifferential operators
of order zero, orthogonal projections onto the positive/negative eigenspaces of the operator $W$.
The operator
$\mathrm{Id}-U^+(0)-U^-(0)$
is the orthogonal projection onto the nullspace of the operator $W$, hence
\begin{equation*}
\label{sum of propagators at time zero}
U^+(0)+U^-(0)\overset{\mod \Psi^{-\infty}}{=} \mathrm{Id}.
\end{equation*}

The principal symbols of the operators $U^\pm(0)$ read
\begin{equation}
\label{principal symbols of propagators at time zero}
[U^\pm(0)]_\mathrm{prin}=P^\pm(y,\eta),
\end{equation}
where $P^\pm$ are the orthogonal projections onto the positive/negative eigenspaces
of the principal symbol of the operator $W$, see \eqref{spectral projections principal symbol}.

The analysis of the \emph{full} symbol of $U^\pm(0)$ is a delicate task which was investigated,
to a certain extent and in a somewhat different setting, in \cite{CDV}.
In order to develop the ideas from \cite{CDV} we have to address a number of issues.
\begin{itemize}
\item
We are now dealing with scalar fields as opposed to half-densities.
\item
We are now making full use of Riemannian structure.
\item
We are now working in the special setting of a system of two equations in dimension three with trace-free principal symbol.
\item
Unlike \cite{CDV,jst_part_b}, we are aiming to evaluate the actual matrix-functions
$[U^\pm(0)]_\mathrm{sub}$ and not only their traces.
\end{itemize}

In order to calculate the subprincipal symbols of the pseudodifferential operators $U^\pm(0)$ 
we will need the following auxiliary result.

\begin{theorem}
\label{theorem covariant derivatives gauge transformation}
Fix a point $y\in M$ and let $\{\widetilde e\}_{j=1}^3$ be a framing on $M$.
Let $G\in C^\infty(M;SU(2))$
be a gauge transformation such that $G(y)=\mathrm{Id}$ and let
\begin{equation}
\label{new framing generated by G}
{e}_j{}^\alpha:=\frac12 \operatorname{tr}(s_j\,G^*\,s^k\,G)\,\widetilde{e}_k{}^\alpha.
\end{equation}
Then
\begin{equation}
\label{covariant derivative of gauge transformation G}
\nabla_\alpha G(y)=-\frac{i}{2} \left[\overset{*}{K}{}_{\alpha\beta}(y)-\overset{*}{\widetilde{K}}{}_{\alpha\beta}(y)\right]\,\sigma^\beta(y),
\end{equation}
where $K$ (resp.~$\widetilde K$) is the contorsion tensor of the Weitzenb\"ock connection
(see Appendix~\ref{Geometric properties of the Weitzenb\"ock connection})
associated with the framing $\{ e_j\}_{j=1}^3$
(resp.~$\{\widetilde e_j\}_{j=1}^3$), the star stands for the Hodge dual
applied in the first and third indices, see formula \eqref{hodge star contorsion}, and $\sigma^\alpha(y)$ is defined by \eqref{Pauli matrices projection}.
\end{theorem}

\begin{proof}
The proof is provided in Appendix~\ref{appendix proof covariant derivative gauge transformation}.
\end{proof}

\begin{remark}
\label{remark about w and tilde w}
Let $\{\widetilde e\}_{j=1}^3$ and $\{e\}_{j=1}^3$ be a pair of framings related in accordance with
\eqref{new framing generated by G},
and let $\widetilde W$ and $W$ be the corresponding Dirac operators,
see Definition~\ref{massless dirac definition}.
Then
\begin{equation}
\label{relation between w and tilde w}
W=G^*\widetilde WG.
\end{equation}
\end{remark}

The following theorem is the main result of this subsection.

\begin{theorem}
\label{theorem about subprincipal symbols of propagators at zero}
We have
\begin{equation}
\label{4 December 2019 equation 6}
[U^\pm(0)]_\mathrm{sub}(y,\eta)
=
\pm
\frac1{4(h(y,\eta))^3}\,
\overset{*}{T}{}^{\alpha\beta}(y)\,
\eta_\alpha\eta_\beta
\,\mathrm{Id}\,,
\end{equation}
where $T$ is the torsion tensor of the Weitzenb\"ock connection
(see Appendix~\ref{Geometric properties of the Weitzenb\"ock connection})
associated with the framing $\{ e_j\}_{j=1}^3$
encoded within the Dirac operator $W$
(see Definition~\ref{massless dirac definition})
and the star stands for the Hodge dual
applied in the second and third indices, see formula \eqref{hodge star torsion}.
\end{theorem}

\begin{proof}
Let us fix a point $y\in M$ and choose normal geodesic coordinates $x$ centred at $y$ such that
$e_j{}^\alpha(y)=\delta_j{}^\alpha\,$. Consider the (local) operator with constant coefficients
\begin{equation}
\label{dirac with constant coefficients}
\widetilde W:=-is^\alpha\,\frac{\partial}{\partial x^\alpha}\,,
\end{equation}
where the $s^\alpha$ are the standard Pauli matrices \eqref{Pauli matrices basic}.
Let us choose a smooth special unitary $2\times2$ matrix-function $G$ such that
\begin{equation*}
\label{G(0) is identity}
G(0)=\mathrm{Id},
\end{equation*}
\begin{equation*}
\label{relation between two principal symbols}
[W]_\mathrm{prin}=[G^*\widetilde WG]_\mathrm{prin}+O(\,\|\eta\|\,\|x\|^2\,)\,,
\end{equation*}
compare with \eqref{relation between w and tilde w}.
It is easy to see that such a matrix-function $G(x)$ exists and is defined uniquely modulo $O(\|x\|^2)$.

Let us now compare the subprincipal symbols of the pseudodifferential operators
$\theta(\pm W)$ and $\theta(\pm G^*\widetilde WG)$,
with $G^*\widetilde WG$ understood as an operator acting in Euclidean space
(constant metric tensor $g_{\alpha\beta}(x)=\delta_{\alpha\beta}$).
It can be shown that at the origin we have
\[
[W]_\mathrm{sub}(0,\eta)=[G^*\widetilde WG]_\mathrm{sub}(0,\eta).
\]
Thus, the proof of the
Theorem~\ref{theorem about subprincipal symbols of propagators at zero} has been
reduced to the case when we are in Euclidean space and the operator $W$
is given by formulae
\eqref{relation between w and tilde w} and \eqref{dirac with constant coefficients}.

We have
\begin{equation}
\label{dima1}
\theta(\pm\widetilde W)
=
\frac1{(2\pi)^3}
\int_{T'\mathbb{R}^3}
e^{i(x-z)^\alpha\eta_\alpha}\,P^\pm(\eta)\,(\,\cdot\,)\,dz\,d\eta\,,
\end{equation}
where
\begin{equation}
\label{dima2}
P^\pm(\eta)
=
\frac12
\left(
\mathrm{Id}
\pm
\frac1{\|\eta\|}s^\beta\eta_\beta
\right).
\end{equation}
Formulae \eqref{dima1} and \eqref{dima2} imply that
\begin{equation*}
\label{dima3}
\theta(\pm G^*\widetilde WG)
=
\frac1{(2\pi)^3}
\int_{T'\mathbb{R}^3}
e^{i(x-z)^\alpha\eta_\alpha}\,Q^\pm(x,z,\eta)\,(\,\cdot\,)\,dz\,d\eta\,,
\end{equation*}
where
\begin{equation*}
\label{dima4}
Q^\pm(x,z,\eta)
=
G^*(x)\,P^\pm(\eta)\,G(z)
=
\frac12\,
G^*(x)
\left(
\mathrm{Id}
\pm
\frac1{\|\eta\|}s^\beta\eta_\beta
\right)
G(z)\,.
\end{equation*}
Excluding the $z$-dependence from the amplitude $Q^\pm$ by acting with the operator
\begin{equation*}
\label{operator Melrose other way round}
\mathcal{S}_\mathrm{left}(\,\cdot\,):=\left.\left[\exp\left(-i\frac{\partial^2}{\partial z^\mu\,\partial \eta_\mu}\right)(\,\cdot\,)\right]\right|_{z=x},
\end{equation*}
compare with \eqref{operator Melrose},
we arrive at
\begin{equation*}
\label{dima5}
\theta(\pm G^*\widetilde WG)
=
\frac1{(2\pi)^3}
\int_{T'\mathbb{R}^3}
e^{i(x-z)^\alpha\eta_\alpha}\,\mathcal{Q}^\pm(x,\eta)\,(\,\cdot\,)\,dz\,d\eta\,,
\end{equation*}
where
\begin{equation}
\label{dima6}
\mathcal{Q}^\pm(x,\eta)
=
\mathcal{Q}_0^\pm(x,\eta)
+
\mathcal{Q}_{-1}^\pm(x,\eta)
+O(\|\eta\|^{-2})\,,
\end{equation}
\begin{equation}
\label{dima7}
\mathcal{Q}_0^\pm(x,\eta)
=
\frac12\,
G^*(x)
\left(
\mathrm{Id}
\pm
\frac1{\|\eta\|}s^\beta\eta_\beta
\right)
G(x)\,,
\end{equation}
\begin{equation}
\label{dima8}
\mathcal{Q}_{-1}^\pm(x,\eta)
=
-\,
\frac i2\,
G^*(x)
\left(
\mathrm{Id}
\pm
\frac1{\|\eta\|}s^\beta\eta_\beta
\right)_{\eta_\mu}
G_{x^\mu}(x)\,.
\end{equation}

In the Euclidean setting the standard formula 
\cite[Eqn.~(5.2.8)]{DuHo}
for the subprincipal symbol reads
\begin{equation}
\label{dima9}
[\theta(\pm G^*\widetilde WG)]_\mathrm{sub}
=
\mathcal{Q}_{-1}^\pm
+
\frac i2(\mathcal{Q}_0^\pm)_{x^\mu\eta_\mu}\,.
\end{equation}
Substituting \eqref{dima7} and \eqref{dima8} into \eqref{dima9} and setting $x=0$, we get
\begin{multline}
\label{dima10}
[\theta(\pm G^*\widetilde WG)]_\mathrm{sub}
=
\pm
\frac i4
\left[
G^*_{x^\mu}
\left(
\frac1{\|\eta\|}s^\beta\eta_\beta
\right)_{\eta_\mu}
-
\left(
\frac1{\|\eta\|}s^\beta\eta_\beta
\right)_{\eta_\mu}
G_{x^\mu}
\right]
\\
=
\pm
\frac{i(
\delta_\beta{}^\mu\|\eta\|^2-\eta_\beta\,\eta^\mu)
}
{4\|\eta\|^3}
\left[
G^*_{x^\mu}
s^\beta
-
s^\beta
G_{x^\mu}
\right].
\end{multline}
Theorem~\ref{theorem covariant derivatives gauge transformation}
tells us that
$G_{x^\mu}=\frac{i}{2}\overset{*}{K}{}_{\mu\nu}\,s^\nu$.
Substituting this into \eqref{dima10}, and using standard properties of Pauli matrices and \eqref{contorsion star vs torsion star}, we get
\begin{equation*}
\begin{split}
\label{dima11}
[\theta(\pm G^*\widetilde WG)]_\mathrm{sub}
&
=
\pm
\frac{
\delta_\beta{}^\mu\|\eta\|^2-\eta_\beta\,\eta^\mu
}
{8\|\eta\|^3}
\left[
s^\nu
s^\beta
+
s^\beta
s^\nu
\right]
\overset{*}{K}{}_{\mu\nu}
\\
&
=
\pm
\frac{
1
}
{4\|\eta\|^3}
\left(
\overset{*}{K}{}^\gamma{}_\gamma \delta_{\mu\nu}
-\overset{*}{K}{}_{\mu\nu}
\right)\eta^\mu\eta^\nu
\operatorname{Id}
\\
&
=
\pm
\frac{1}{4\|\eta\|^3}
\overset{*}{T}{}_{\mu\nu}
\,\eta^\mu\,\eta^\nu
\operatorname{Id}.
\end{split}
\end{equation*}

The above argument combined with \eqref{4 December 2019 equation 1} yields \eqref{4 December 2019 equation 6}.
\end{proof}

Observe that formula \eqref{4 December 2019 equation 6} implies
\begin{equation*}
\label{4 December 2019 equation 6 trace}
\operatorname{tr}\,
[U^\pm(0)]_\mathrm{sub}(y,\eta)
=
\pm
\frac1{2(h(y,\eta))^3}\,
\overset{*}{T}{}^{\alpha\beta}(y)\,
\eta_\alpha\eta_\beta\,,
\end{equation*}
which agrees with
\cite[formula (1.20)]{CDV} and \cite[formula (4.1) with $\mathbf{c}=+1$]{jst_part_b}.

\section{Principal symbol of the global Dirac propagator}
\label{Principal symbol of the global Dirac propagator}

In this section we provide an explicit geometric characterisation of the principal symbols of the positive and negative Dirac propagators.

\begin{theorem}
\label{theorem small time expansion principal symbol}
The principal symbols of the positive and negative Dirac propagators are
\begin{equation}
\label{principal symbols dirac propagator geometric formula}
\mathfrak{a}^\pm_0(t;y,\eta)=\zeta^\pm(t;y,\eta)\,[v^\pm(y,\eta)]^*,
\end{equation}
where $\zeta^\pm(t;y,\eta)$ is the parallel transport of $v^\pm(y,\eta)$ along $x^\pm$ with respect to the spin connection,
i.e.
\begin{equation}
\label{definition of parallel transport of spinor}
\left( \frac{d}{dt} +[\dot{x}^\pm]^\alpha\,\frac14 \sigma_\beta \left(\frac{\partial\sigma^\beta}{\partial x^\alpha}+ \Gamma^\beta{}_{\alpha\gamma}\sigma^\gamma \right) \right)
\zeta^\pm=0,
\qquad
\zeta^\pm|_{t=0}=v^\pm.
\end{equation}
\end{theorem}

\begin{proof}
It is known \cite{safarov}\cite[\textcolor{black}{Subsection~3.4}]{nicoll} that the principal symbols $\mathfrak{a}_0^\pm$ are independent of the choice of the phase function and read
\begin{equation}
\label{principal symbol safarov-nicoll formula}
\mathfrak{a}_0^\pm(t;y,\eta)=v^\pm(x^\pm,\xi^\pm)\,[v^\pm(y,\eta)]^* \,e^{-i\,\int_0^t q^\pm(x^\pm(\tau;y,\eta),\xi^\pm(\tau;y,\eta))\,d \tau},
\end{equation}
where
\begin{equation}
\label{q appearing in principal symbols}
q^\pm=[v^\pm]^* \,W_{\mathrm{sub}} \,v^\pm-\frac{i}{2}\{[v^\pm]^*,W_\mathrm{prin}-h^\pm,v^\pm \}-i\, [v^\pm]^*\{v^\pm,h^\pm\}\,,
\end{equation}
and
\begin{equation}
\label{subprincipal symbol dirac}
W_{\mathrm{sub}}(y):=W_0(y)+\frac{i}2 \sigma^\alpha(y) \,\Gamma^\beta{}_{\alpha\beta}(y)+
\frac{i}{2}[W_\mathrm{prin}(y,\eta)]_{y^\alpha\eta_\alpha}.
\end{equation}
In formula \eqref{q appearing in principal symbols} curly brackets denote the Poisson bracket
\begin{equation*}
\label{poisson brackets}
\{B,C\}:=B_{y^\alpha} C_{\eta_\alpha}- B_{\eta_\alpha} C_{y^\alpha}
\end{equation*}
and the generalised Poisson bracket
\begin{equation*}
\label{generalised poisson brackets}
\{ B,C,D\}:=B_{y^\alpha} C D_{\eta_\alpha}- B_{\eta_\alpha} C D_{y^\alpha}
\end{equation*}
on matrix-functions on the cotangent bundle.
In formula \eqref{subprincipal symbol dirac} the second term on the RHS
is the result of switching to half-densities, see \eqref{P one half}.

Introducing the shorthand $q^\pm(t):=q^\pm(x^\pm(t;y,\eta), \xi^\pm(t;y,\eta))$, the task at hand is to show that 
\[
\zeta^\pm(t;y,\eta)=e^{-i\,\int_0^t q^\pm(\tau)\,d \tau}\,v^\pm(x^\pm,\xi^\pm).
\]
More explicitly, we need to show that
\begin{equation}
\label{proof principal symbols task pt}
e^{i\,\int_0^t q^\pm(\tau)\,d \tau}\left( \frac{d}{dt} +[\dot{x}^\pm]^\alpha\,\frac14 \sigma_\beta \left(\frac{\partial\sigma^\beta}{\partial x^\alpha}+ \Gamma^\beta{}_{\alpha\gamma}\sigma^\gamma \right) \right)\left[ e^{-i\,\int_0^t q^\pm(\tau)\,d \tau}\,v^\pm(x^\pm,\xi^\pm)\right]=0,
\end{equation}
where we premultiplied our expression by $e^{i\,\int_0^t q^\pm(\tau)\,d \tau}$ for the sake of convenience.

We shall prove \eqref{principal symbols dirac propagator geometric formula} for $\mathfrak{a}_0^+$, which corresponds to the upper choice of signs in \eqref{proof principal symbols task pt}. The proof for $\mathfrak{a}_{0}^-$ is analogous.

Let us begin by computing
\begin{equation}
\label{proof principal symbol dirac 1}
\begin{split}
e^{i\,\int_0^t q^+(\tau)} \dfrac{d}{dt}\left( e^{-i\,\int_0^t q^+(\tau)\,d \tau}\,v^+(x^+,\xi^+)\right)
&
=
-i q^+(t) \,v^++ v^+_{x^\alpha}[\dot{x}^+]^\alpha+v^+_{\xi_\alpha} [\dot{\xi}^+]_\alpha\\
&
=
-i q^+(t) \,v^+ + \{v^+,h\}.
\end{split}
\end{equation}
To this end, let us choose geodesic normal coordinates centred at $x^+(t;y,\eta)=0$ and such that $[\xi^+(t;y,\eta)]_\alpha=\delta_{3\alpha}$. Furthermore, up to a global rigid rotation of the framing, we can assume that
\begin{equation*}
\label{proof principal symbol dirac 2}
e_j{}^\alpha(0)=\delta_j{}^\alpha.
\end{equation*}

In our special coordinate system we have
\begin{equation}
\label{proof principal symbol dirac 2bis}
v^+(0,\xi^+)=
\begin{pmatrix}
1\\0
\end{pmatrix},
\qquad
v^-(0,\xi^+)=
\begin{pmatrix}
0\\1
\end{pmatrix},
\end{equation}
and we can expand our framing about $x^+=0$ as
\begin{equation}
\label{proof principal symbol dirac 3}
\begin{pmatrix}
e_1{}^1(x)&e_1{}^2(x)&e_1{}^3(x)\\
e_2{}^1(x)&e_2{}^2(x)&e_2{}^3(x)\\
e_3{}^1(x)&e_3{}^2(x)&e_3{}^3(x)
\end{pmatrix}
=
\begin{pmatrix}
1&l^3(x)&-l^2(x)\\
-l^3(x)&1&l^1(x)\\
l^2(x)&-l^1(x)&1
\end{pmatrix}+O(\|x\|^2) \quad \text{as }x\to 0,
\end{equation}
where $l^k(x)=O(\|x\|)$, $k=1,2,3$.

The fact that $([v^+]^*v^+)(x,\xi)=1$ implies
\begin{equation*}
\label{proof principal symbol dirac 7}
\{[v^+]^*,P^+,v^+\}(0,\xi^+)=
[v^+_{x^\alpha}]^*\,v^+\,[v^+]^*\,v^+_{\xi^\alpha}
-
[v^+_{\xi_\alpha}]^*\,v^+\,[v^+]^*\,v^+_{x^\alpha}=0,
\end{equation*}
which, in turn, yields
\begin{equation}
\label{proof principal symbol dirac 8}
\{[v^+]^*,W_\mathrm{prin},v^+\}=h\,\{[v^+]^*,2P^+-\mathrm{Id},v^+\}=-h\,\{[v^+]^*,v^+\}.
\end{equation}

A standard perturbation argument gives us
\begin{equation}
\label{proof principal symbol dirac 4}
h\,\{[v^+]^*,v^+\}(0,\xi^+)=\left.-\frac i2 \left( \dfrac{\partial l^1}{\partial x^1}+\dfrac{\partial l^2}{\partial x^2}\right)\right|_{x=0}
\end{equation}
and
\begin{equation}
\label{proof principal symbol dirac 5}
\{v^+,h\}(0,\xi^+)=\left.\frac{i}2 
\begin{pmatrix}
0\\
\dfrac{\partial l^1}{\partial x^3}+i\dfrac{\partial l^2}{\partial x^3}
\end{pmatrix}\right|_{x=0}.
\end{equation}

Furthermore, combining \eqref{proof principal symbol dirac 3} with \eqref{Pauli matrices projection} and \eqref{subprincipal symbol dirac}, we get
\begin{equation}
\label{proof principal symbol dirac 6}
W_\mathrm{sub}(0)=-\left. \frac12\left(\dfrac{\partial l^1}{\partial x^1}+\dfrac{\partial l^2}{\partial x^2}+\dfrac{\partial l^3}{\partial x^3} \right)\right|_{x=0}\,\mathrm{Id}.
\end{equation}

Substituting \eqref{proof principal symbol dirac 2bis}, \eqref{proof principal symbol dirac 8} and \eqref{proof principal symbol dirac 4}--\eqref{proof principal symbol dirac 6} into \eqref{q appearing in principal symbols}, and then \eqref{q appearing in principal symbols} and \eqref{proof principal symbol dirac 5} into 
\eqref{proof principal symbol dirac 1}, we conclude that
\begin{equation}
\label{proof principal symbol dirac 9}
e^{i\,\int_0^t q^+(\tau)} \dfrac{d}{dt}\left( e^{-i\,\int_0^t q^+(\tau)\,d \tau}\,v^+(x^+,\xi^+)\right)=\left.\dfrac{i}{2}\begin{pmatrix}
\dfrac{\partial l^3}{\partial x^3}
\\
0
\end{pmatrix}\right|_{x=0}
+\left.\dfrac{i}{2}\begin{pmatrix}
0\\
\dfrac{\partial l^1}{\partial x^3}+i\dfrac{\partial l^2}{\partial x^3}
\end{pmatrix}\right|_{x=0}.
\end{equation}

Similarly, in our special coordinate system we have
\begin{equation}
\label{proof principal symbol dirac 10}
\begin{split}
\left.[\dot{x}^+]^\alpha\,\frac14 \sigma_\beta \left(\frac{\partial\sigma^\beta}{\partial x^\alpha}+ \Gamma^\beta{}_{\alpha\gamma}\sigma^\gamma \right) v^+\right|_{x=0,\ \xi=\xi^+}
&
=
\left.\frac14 \sigma_\beta \left(\frac{\partial\sigma^\beta}{\partial x^3}\right) \begin{pmatrix}
1\\0
\end{pmatrix}
\right|_{x=0}
\\
&
=
\left.-\frac{i}{2}
\begin{pmatrix}
\dfrac{\partial l^3}{\partial x^3}
\\
\dfrac{\partial l^1}{\partial x^3}+i\dfrac{\partial l^2}{\partial x^3}
\end{pmatrix}
\right|_{x=0}.
\end{split}
\end{equation}

Summing up \eqref{proof principal symbol dirac 9} and \eqref{proof principal symbol dirac 10} we arrive at
\eqref{proof principal symbols task pt}.
\end{proof}

\section{Explicit small time expansion of the symbol}
\label{Explicit small time expansion of the symbol}

Even though the presence of gauge degrees of freedom represents an additional challenge in the analysis of the propagator, one can put this freedom to use and exploit it to obtain a small time expansion for the propagator.

Our strategy goes as follows. 
\begin{enumerate}
\item 
Compute the principal and subprincipal symbols of the positive (resp.~negative) propagator for a conveniently chosen framing;

\item 
Using the gauge transformation \eqref{gauge transfomation in terms of O}, \eqref{transformation of the Dirac operator under change of frame}, switch to an arbitrary framing with the same orientation\footnote{Recall that in our paper the orientation is prescribed from the beginning.};

\item
Express the final result in terms of geometric invariants.

\end{enumerate}

\subsection{Special framing}

Let us fix an arbitrary point $y\in M$ and let $V_j\in T_yM$, $j=1,2,3$ be defined by
\begin{equation*}
\label{special framing at zero dirac}
V_j:=e_j(y).
\end{equation*}

\begin{definition}[Levi-Civita framing]
\label{definition Levi-Civita framing}
Let $\mathcal{U}$ be a geodesic neighbourhood of $y$. For $x\in \mathcal{U}$, let $\widetilde e_j^{\mathrm{loc}}(x)$, $j=1,2,3$, be the parallel transport of $V_j$ along the shortest geodesic connecting $y$ to $x$.
We define the \emph{Levi-Civita framing generated by $\{e_j\}_{j=1}^3$ at $y$} to be the equivalence class of framings coinciding with $\{\widetilde e_j^{\mathrm{loc}}\}_{j=1}^3$ in a neighbourhood of $y$.
\end{definition}

With slight abuse of notation, in the following we will identify the Levi-Civita framing with one of its representatives, denoted by $\{\widetilde e_j\}_{j=1}^3$. The choice of a particular representative does not affect our results.

Using the Levi-Civita framing is especially convenient due to the following property.

\begin{lemma}
\label{lemma Levi-Civita framing}
In normal coordinates centred at $y$, the Levi-Civita framing admits the following expansion:
\begin{equation}
\label{expansion Levi-Civita framing Dirac}
\widetilde e_j{}^\alpha(x)=e_j{}^\alpha(y)+\frac16 e_j{}^\beta(y)\,R^\alpha{}_{\mu\beta\nu}(y)\,x^\mu x^\nu +O(\|x\|^3), \qquad j=1,2,3,
\end{equation}
where $R$ is the Riemann curvature tensor\footnote{\color{black}The Riemann curvature tensor $R$ has components ${R^\kappa}_{\lambda\mu\nu}$ defined in accordance with 
\[
{R^\kappa}_{\lambda\mu\nu}:=
dx^\kappa(R(\partial_\mu\,,\partial_\nu)\,\partial_\lambda)
=
\partial_\mu{\Gamma^\kappa}_{\nu\lambda}
-\partial_\nu{\Gamma^\kappa}_{\mu\lambda}
+{\Gamma^\kappa}_{\mu\eta}{\Gamma^\eta}_{\nu\lambda}
-{\Gamma^\kappa}_{\nu\eta}{\Gamma^\eta}_{\mu\lambda}\,.
\]
}.
\end{lemma}

\begin{proof}
In normal geodesic coordinates centred at $y$, the unique geodesic connecting $y$ to $x$ can be written as
\begin{equation}
\label{8 March 2019 formula 2}
\gamma^\alpha(t)=\frac{x^\alpha}{\|x\|_E}\,t,
\end{equation}
where $\|\,\cdot\,\|_E$ is the Euclidean norm, so that $\gamma(\|x\|_E)=x$. Assuming $t$ and $\|x\|_E$ to be small and of the same order, let us perform an expansion in powers of $t$ of $\widetilde e_j$.

The parallel transport equation defining the framing $\{\widetilde e_j\}_{j=1}^3$ reads
\begin{equation}
\label{8 March 2019 formula 3}
\dot{\widetilde{e}}_j{}^\alpha(\gamma(t))=-\dot{\gamma}^\beta(t)\, \Gamma^\alpha{}_{\beta\mu}(\gamma(t))\,\, \widetilde{e}_j{}^\mu(\gamma(t)), \qquad j=1,2,3.
\end{equation}
Since $\widetilde{e}_j(0)=V_j$ and $\Gamma(0)=0$, at linear order in $t$ we have $\dot{\widetilde e}_j(\gamma(t))=O(t),$ which implies
\begin{equation}
\label{8 March 2019 formula 4}
\widetilde e_j(\gamma(t))=V_j+O(t^2).
\end{equation}
Substituting \eqref{8 March 2019 formula 4} into \eqref{8 March 2019 formula 3}, we get
\begin{equation*}
\label{8 March 2019 formula 5}
\dot{\widetilde e}_j{}^\alpha(\gamma(t))=-\frac{x^\beta\,x^\nu}{\|x\|_E^2} \partial_\nu\Gamma^\alpha{}_{\beta\mu}(0)\,V_j{}^\mu\,t+O(t^2),
\end{equation*}
so that
\begin{equation*}
\label{8 March 2019 formula 6}
\widetilde e_j{}^\alpha(\gamma(t))=V_j-\frac{1}{2}\frac{x^\beta x^\nu}{\|x\|_E^2} \partial_\nu\Gamma^\alpha{}_{\beta\mu}(0) \, V_j{}^\mu\,t^2+O(t^3)
\end{equation*}
and
\begin{equation}
\label{8 March 2019 formula 7}
\widetilde e_j{}^\alpha(x)=\widetilde e_j{}^\alpha(\gamma(\|x\|_E))=V_j{}^\alpha-\frac12\partial_\nu\Gamma^\alpha{}_{\beta\mu}(0)\,V_j{}^\mu\, x^\beta x^\nu +O(\|x\|^3), \qquad j=1,2,3.
\end{equation}
Formula \eqref{expansion Levi-Civita framing Dirac} follows at once from \eqref{8 March 2019 formula 7} and the elementary identity
\begin{equation}
\label{relation beteween christoffel and riemann dirac}
\partial_\nu\Gamma^\alpha{}_{\beta\mu}(0)=-\frac13 (R^\alpha{}_{\beta\mu\nu}+R^\alpha{}_{\mu\beta\nu})(0).
\end{equation}
\end{proof}

\begin{corollary}
\label{corollary projection Pauli matrices}
In normal coordinates $x$ centred at $y$, the Pauli matrices
$\widetilde\sigma^\alpha(x)$
projected along the Levi-Civita framing
(see \eqref{Pauli matrices projection}) satisfy
\begin{equation}
\label{Pauli matrices projected along Levi Civita framing dirac}
\widetilde\sigma^\alpha(y)=\sigma^\alpha(y),
\qquad
[\widetilde\sigma^\alpha]_{x^\beta}(y)=0,
\qquad
[\widetilde\sigma^\alpha]_{x^\mu x^\nu}(y)=\frac16
\left[
R^\alpha{}_{\nu\beta\mu}(y)+R^\alpha{}_{\mu\beta\nu}(y)
\right]
\sigma^\beta(y).
\end{equation}
\end{corollary}

\begin{proof}[Proof of Corollary~\ref{corollary projection Pauli matrices}]
Formula \eqref{Pauli matrices projected along Levi Civita framing dirac} follows immediately from \eqref{expansion Levi-Civita framing Dirac}.
\end{proof}

\begin{corollary}
\label{corollary taylor expansion W0 dirac}
Let $\widetilde{W}$ be the Dirac operator \eqref{massless dirac definition equation} corresponding to the choice of the Levi-Civita framing. Then, in normal coordinates centred at $y$, its zero order part $\widetilde{W}_0$ (see formula~\eqref{zero order part dirac definition}) admits the following expansion:
\begin{equation}
\label{derivative of W_0 at zero dirac}
\widetilde{W}_0(x)=\frac{i}{4} \Ric_{\alpha\beta}(y)\,\widetilde{\sigma}^\beta(y)\, x^\alpha +O(\|x\|^2).
\end{equation}
\textcolor{black}{Here $\Ric$ is the Ricci tensor, $\Ric_{\alpha\beta}:=R^\gamma{}_{\alpha\gamma\beta}$.}
\end{corollary}

\begin{proof}
Formula \eqref{derivative of W_0 at zero dirac} is obtained by expanding the RHS of  \eqref{zero order part dirac definition} in powers of $x$ in normal coordinates centred at $y$, substituting \eqref{relation beteween christoffel and riemann dirac} and \eqref{Pauli matrices projected along Levi Civita framing dirac} in and performing a lengthy but straightforward calculation. It is a somewhat nontrivial fact that the coefficient of the linear term in \eqref{derivative of W_0 at zero dirac} turns out to be trace-free.
\end{proof}

\subsection{Small time expansion of the principal symbols}

The first step towards computing small time expansions for principal and subprincipal symbols of $W$ is to obtain an expression for these objects in a neighbourhood of a given point $y\in M$ for the choice of the Levi-Civita framing generated by our framing $\{e_j\}_{j=1}^3$ at $y$. Observe that, as we are after a small time expansion of the symbols, it is enough to restrict our attention to a small open neighbourhood of $y$.

In the following, we will denote with a tilde objects associated with the Dirac operator $\widetilde{W}$ corresponding to the choice of the Levi-Civita framing.

\begin{theorem}
\label{theorem small time expansion principal symbols dirac}
For the choice of the Levi-Civita framing, the positive and negative principal symbols are independent of $t$ and read
\begin{equation}
\label{small time expansion principal LC dirac}
\widetilde {\mathfrak{a}}^{\pm}_0(t;y,\eta)=\widetilde{P}^{\pm}(y,\eta).
\end{equation}
\end{theorem}

\begin{proof}
In accordance with Theorem~\ref{theorem small time expansion principal symbol}, the principal symbols are determined by the eigenvectors of $\widetilde{W}_\mathrm{prin}$ and their parallel transport with respect with the spin connection along the Hamiltonian trajectories.
Hence, it suffices to show that 
\begin{equation}
\label{proof small time expansion principal LC dirac 1}
\widetilde{\zeta}^\pm(t;y,\eta)=\widetilde{v}^\pm(y,\eta).
\end{equation}
Once this is achieved, \eqref{small time expansion principal LC dirac} follows from the fact that $\widetilde{W}_\mathrm{prin}(y,\eta)=W_\mathrm{prin}(y,\eta)$.

In normal coordinates centred at $y$ the parallel transport equation \eqref{definition of parallel transport of spinor} reads
\begin{equation}
\label{new proof principal symbols 1}
\left[ \frac{d}{dt} +[\dot{x}^\pm]^\alpha\,\frac14 \widetilde{\sigma}_\beta(x^\pm)\left( \frac{\partial\widetilde{\sigma}^\beta}{\partial x^\alpha}(x^\pm)+\Gamma^\beta{}_{\alpha\gamma}(x^\pm)\widetilde{\sigma}^\gamma(x^\pm)\right)
\right]
\widetilde{\zeta}^\pm=0, \qquad  \widetilde{\zeta}^\pm|_{t=0}=\widetilde{v}^\pm.
\end{equation}
We claim that 
\begin{equation}
\label{new proof principal symbols 2}
[\dot{x}^\pm]^\alpha\left( \frac{\partial\widetilde{\sigma}^\beta}{\partial x^\alpha}(x^\pm)+\Gamma^\beta{}_{\alpha\gamma}(x^\pm)\widetilde{\sigma}^\gamma(x^\pm)\right)=0.
\end{equation}
In fact, we have
\[
[\dot{x}^\pm]^\alpha\left( \frac{\partial\widetilde{\sigma}^\beta}{\partial x^\alpha}(x^\pm)+\Gamma^\beta{}_{\alpha\gamma}(x^\pm)\widetilde{\sigma}^\gamma(x^\pm)\right)=
[\dot{x}^\pm]^\alpha (\partial_{x^\alpha} \widetilde{e}_j{}^\beta+\Gamma^\beta{}_{\alpha\gamma}\, \widetilde{e}_j{}^\gamma)(x^\pm) \,s^j
\]
and
\[
[\dot{x}^\pm]^\alpha (\partial_{x^\alpha} \widetilde{e}_j{}^\beta+\Gamma^\beta{}_{\alpha\gamma}\, \widetilde{e}_j{}^\gamma)(x^\pm)=0 \quad \text{for}\quad j=1,2,3
\]
in view of Definition~\ref{definition Levi-Civita framing} and the properties of the Hamiltonian flows $x^\pm$, i.e.\ that $x^+(\,\cdot\,;y,\eta)$ is geodesic and relation \eqref{relation between two geodesic flows}.
By substituting \eqref{new proof principal symbols 2} into \eqref{new proof principal symbols 1} we arrive at \eqref{proof small time expansion principal LC dirac 1}.
\end{proof}

\subsection{Small time expansion of the subprincipal symbols}

Let us now turn our attention to the subprincipal symbols $\widetilde{\mathfrak{a}}_{-1}^\pm$. 

Unlike the principal symbols, the subprincipal symbols depend on the choice of phase functions. As here we are only interested in small time expansions and the injectivity radius $\operatorname{Inj}(M,g)$ is strictly positive, we can work, without loss of generality, in a neighbourhood of $y$ with no conjugate points to $y$. The absence of conjugate points allows us to construct positive and negative propagators for small times by means of the algorithm described in subsection~\ref{The algorithm} for the choice of \emph{real-valued} Levi-Civita phase functions
\begin{equation*}
\label{real valued LC phase functions dirac}
\varphi^\pm(t,x;y,\eta)=\int_{\gamma^\pm} \zeta^\pm \,dz,
\end{equation*}
cf.~Definition~\ref{definition LC phase functions with epsilon dirac} for $\epsilon=0$.

In the remainder of this subsection we adopt the same coordinates for $x$ and $y$ and we choose normal geodesic coordinates centred at $y$. We remind the reader that, in such coordinates, 
\begin{equation}
\label{x plus minus normal coordinates dirac}
[x^\pm]^\alpha(t;0,\eta)=\pm\frac{\eta^\alpha}{h}\,t.
\end{equation}

According to \cite[Eqns.~(8.7) and (8.12)]{wave} and \eqref{relation between the two Levi-Civita phase functions}, we have
\begin{equation}
\label{real valued LC phase functions dirac expansion}
\varphi^\pm(t,x;0,\eta)=x^\alpha\eta_\alpha \mp h\,t\pm \frac{\textcolor{black}{t}}{3h} R^\alpha{}_\mu{}^\beta{}_\nu(0) \eta_\alpha\eta_\beta\,x^\mu x^\nu +O(\|x\|^4+t^4)
\end{equation}
and
\begin{equation}
\label{weight real valued LC phase functions dirac expansion}
w^\pm(t,x;0,\eta)=1+\frac1{12}\operatorname{Ric}_{\mu\nu}(0)\,x^\mu\,x^\nu\mp\frac{t}{3h}\operatorname{Ric}^\mu{}_\nu(0)\,\eta_\mu\,x^\nu +O(\|x\|^3+|t|^3).
\end{equation}
Recall that the weight $w$ is defined by \eqref{weight definition general case}. 

As explained in subsection~\ref{Transport equations dirac}, the subprincipal symbols are determined by the first and the second transport equations, \eqref{second transport equation dirac} and \eqref{third transport equation dirac}. More precisely, if we are interested in expansions with remainder $O(t^2)$, we need to determine \eqref{third transport equation dirac} up to zeroth order in $t$ and \eqref{second transport equation dirac} up to first order in $t$. 

To this end, we begin by observing that formulae \eqref{real valued LC phase functions dirac expansion} and \eqref{weight real valued LC phase functions dirac expansion}, see also \eqref{operators mathfrak S with j}, imply that the differential evaluation operators $\mathfrak{S}^\pm_{-2}$ and $\mathfrak{S}^\pm_{-1}$ admit the following expansions in normal coordinates centred at $y$.

\begin{lemma}
\label{lemma expansion fractur S minus 1 and S minus 2}
We have
\begin{enumerate}[(a)]
\item 
\begin{equation}
\label{expansion S_-2 order zero dirac}
\mathfrak{S}^\pm_{-2}=\left.\frac12 \left[i \dfrac{\partial^2}{\partial x^\alpha\partial \eta_\alpha}  \right]^2 \left( \,\cdot\,\right)\right|_{t=0, \,x=0}+O(t),
\end{equation}

\item
\begin{equation}
\label{expansion S_-1 order one dirac}
\mathfrak{S}^\pm_{-1}=i\,\mathfrak{S}_0^\pm\left(\dfrac{\partial^2}{\partial x^\alpha\partial \eta_\alpha}\pm\frac{t}{2}\, h_{\eta_\alpha\eta_\beta} \dfrac{\partial^2}{\partial x^\alpha\partial x^\beta} \right) +O(t^2).
\end{equation}
\end{enumerate}
\end{lemma}

\begin{proof}
(a) It is an immediate consequence of \eqref{x plus minus normal coordinates dirac}, \eqref{weight real valued LC phase functions dirac expansion}
and
\begin{equation}
\label{proof taylor expansion fraktur sigmas dirac 1}
L^\pm_\alpha=\dfrac{\partial}{\partial x^\alpha}+O(\|x\|+|t|).
\end{equation}

(b) Substituting \eqref{weight real valued LC phase functions dirac expansion} into \eqref{operators mathfrak S with j} with $k=1$ and recalling that $\left.\varphi^\pm_{\eta}\right|_{x=x^\pm}=0$, we get
\begin{equation}
\label{2 March 2019 forumla 3}
\mathfrak{S}^{\pm}_{-1}=\mathfrak{S}_0^{\pm}\left[ i \dfrac{\partial^2}{\partial x^\alpha\partial \eta_\alpha}-\frac{i}{2} \varphi^{\pm}_{\eta_\alpha\eta_\beta}L^{\pm}_\alpha L^{\pm}_\beta \right]+O(t^2).
\end{equation}
Formula \eqref{proof taylor expansion fraktur sigmas dirac 1} and the fact that
\begin{equation}
\label{varphi pm second derivative in eta dirac}
\left.\varphi^\pm_{\eta_\alpha\eta_\beta}\right|_{x=x^\pm}=\mp t\, h_{\eta_\alpha\eta_\beta}+O(t^3)
\end{equation}
yield \eqref{expansion S_-1 order one dirac}.
\end{proof}

In order to be able to compute the subprincipal symbols, we need to determine the initial condition $\widetilde{\mathfrak{a}}^\pm_{-1}|_{t=0}$ first.

\begin{lemma}
\label{lemma initial condition subprincipal symbol dirac}
For the choice of real-valued Levi-Civita phase functions, the positive and negative subprincipal symbols $\widetilde{\mathfrak{a}}^\pm_{-1}$ vanish at $t=0$:
\begin{equation}
\label{initial condition subprincipal symbol dirac}
\widetilde{\mathfrak{a}}^\pm_{-1}(0;y,\eta)=0.
\end{equation}
\end{lemma}

\begin{proof}
The subprincipal symbols are scalar functions, so it enough to establish \eqref{initial condition subprincipal symbol dirac} in one specific coordinate system. Let us choose normal coordinates centred at $y=0$ such that $\widetilde{e}_j{}^\alpha(0)=\delta_j{}^\alpha$. We observe that the torsion of the Weitzenb\"ock connection generated by the Levi-Civita framing at $y$ vanishes at $y$, as a consequence of the fact that the first derivatives of the framing are zero, cf.~\eqref{expansion Levi-Civita framing Dirac} and \eqref{weitzenbock connection coefficients definition}--\eqref{torsion definition}. Therefore, Theorem~\ref{theorem about subprincipal symbols of propagators at zero} tells us that
\begin{equation}
\label{4 December 2019 equation 6 bis}
[U^\pm(0)]_\mathrm{sub}(0,\eta)=0.
\end{equation}
A straightforward perturbation argument shows that 
\begin{equation}
\label{proof lemma initial conditions subps dirac 1}
(v^\pm)_{x^\alpha}(0,\eta)=0.
\end{equation}
Substituting \eqref{4 December 2019 equation 6 bis} and \eqref{proof lemma initial conditions subps dirac 1} into \eqref{relation between invariant subps} with $P=U^\pm(0)$ and $\epsilon=0$ and using the fact that Christoffel symbols vanish at $y$, we arrive at \eqref{initial condition subprincipal symbol dirac}.
\end{proof}

We are now in a position to examine the first transport equation.

\begin{lemma}
\label{lemma projection along other eigenspace Dirac}
The projection onto the negative (resp.~positive) eigenspace of $\widetilde{W}_\mathrm{prim}$ of the subprincipal symbol of the positive (resp.~negative) propagator is given by
\begin{multline}
\label{projection on other eigenspace subps dirac}
\widetilde{P}^\mp(x^\pm,\xi^\pm)\widetilde{\mathfrak{a}}_{-1}^\pm(t;y,\eta)=\pm it\,\widetilde{P}^\mp(y,\eta)\left[\frac1{8h^3}\,\Ric_{\alpha\beta}(y)\,\eta^\alpha\eta^\beta-\frac1{4h}\Ric_{\alpha\beta}(y)\,\eta^\alpha \widetilde{P}^\pm_{\eta_\beta}(y,\eta) \right]
\\+O(t^2).
\end{multline}
\end{lemma}

\begin{proof}
We will establish formula \eqref{projection on other eigenspace subps dirac} by expanding the first transport equation \eqref{second transport equation dirac} up to first order in $t$ and then acting with $\widetilde{P}^\mp$ on the left. Recall that $a^\pm_{-k}$ is defined by \eqref{transport equations formula 5}.

Working in normal coordinates centred at $y$ and using \eqref{real valued LC phase functions dirac expansion}--\eqref{weight real valued LC phase functions dirac expansion}, we obtain
\begin{equation}
\label{17 June 2019 formula 2}
\begin{aligned}
\mathfrak{S}^\pm_0 \widetilde{a}^\pm_{0}(t;0,\eta)&=
\Bigl\{
\left(\varphi^\pm_t+\widetilde{W}_\mathrm{prin}(x,\varphi_x^\pm)\right)\widetilde{\mathfrak{a}}^\pm_{-1} -i (\widetilde{\mathfrak{a}}^\pm_{-0})_t + \left[-i (w^\pm)^{-1}\left( w^\pm_t+\sigma^\alpha w^\pm_{x^\alpha} \right)+\widetilde{W}_0 \right]\widetilde{\mathfrak{a}}^\pm_{-0}
\Bigr\}
\Bigr|_{x=x^\pm}
\\
&=(\widetilde{W}_\mathrm{prin}(x^\pm,\xi^\pm)\mp h)\,\widetilde{\mathfrak{a}}^\pm_{-1}(t;0,\eta)+\frac{i\,t}{3h^2}\Ric_{\alpha\nu}(0)
\left(\eta^\alpha\eta^\nu \pm\frac12 h\,\eta^\alpha \,\widetilde{\sigma}^\nu(0)\right)\widetilde{P}^\pm(y,\eta)\\
&\pm \frac{t\,\eta^\alpha}{h}(\widetilde{W}_0)_{x^\alpha}(0)\,\widetilde{P}^\pm(y,\eta)+O(t^2).
\end{aligned}
\end{equation}
Furthermore, in view of Theorem~\ref{theorem small time expansion principal symbol} and Lemma~\ref{lemma expansion fractur S minus 1 and S minus 2}(b), we have
\begin{equation}
\label{17 June 2019 formula 3}
\begin{aligned}
\mathfrak{S}^\pm_{-1}\widetilde{a}^\pm_{1}(t;0,\eta)
&
=
\left[\dfrac{\partial^2}{\partial x^\alpha\partial \eta_\alpha}\pm\frac{t}{2}\, h_{\eta_\alpha\eta_\beta} \dfrac{\partial^2}{\partial x^\alpha\partial x^\beta} \right] \left. \left(\varphi^\pm_t+\widetilde{W}_{\mathrm{prin}}(x,\varphi^\pm_x)\right)\widetilde{\mathfrak{a}}_0^\pm\right|_{x=x^*}+O(t^2)\\
&=
-\frac{2\ir t}{3h^2}\Ric_{\alpha\nu}(0) \,\eta^\alpha\eta^\nu \widetilde{P}^\pm
\pm\ir t \left[\frac{2}{3h}R^{\mu}{}_\alpha{}^\nu{}_\beta(0)\,\eta_\mu\eta_\nu\,\widetilde{\sigma}^\beta(0)\widetilde{P}^\pm\right]_{\eta_\alpha}\\
&\pm it
\frac{\eta^\beta}{h}\left[(\widetilde{W}_\mathrm{prin})_{x^\alpha x^\beta}(0,\eta)\widetilde{P}^\pm\right]_{\eta_\alpha}
\\
&+\ir \,t \left[
 \frac{h_{\eta_\alpha\eta_\beta}}{3h}\,R^{\mu}{}_\alpha{}^\nu{}_\beta(0)\,\eta_\mu\eta_\nu\pm\frac{1}{2}h_{\eta_\alpha\eta_\beta}(\widetilde{W}_\mathrm{prin})_{x^\alpha x^\beta}(0,\eta) \right]\widetilde{P}^\pm+O(t^2).
\end{aligned}
\end{equation}
Adding up \eqref{17 June 2019 formula 2} and \eqref{17 June 2019 formula 3} and projecting along $\widetilde{P}^{\mp}$,
we arrive at
\begin{equation}
\label{17 June 2019 formula 4}
\begin{split}
\widetilde{P}^{\mp}\widetilde{\mathfrak{a}}_{-1}^\pm(t;0,\eta)&= \frac{i t}{h}\widetilde{P}^\mp\left\{ 
\frac{1}{12h}\Ric_{\alpha\nu}(0)\,\eta^\alpha \widetilde{\sigma}^\nu(0) \widetilde{P}^\pm-\frac{i\eta^\alpha}{2h}(\widetilde{W}_0)_{x^\alpha}(0)\,\widetilde{P}^\pm
\right.\\
&
+\left[ \frac{1}{3h}R^\mu{}_\alpha{}^\nu{}_\beta(0) \,\eta_\mu\eta_\nu\,\widetilde{\sigma}^\beta(0) \,\widetilde{P}^\pm \right]_{\eta_\alpha}+\frac{\eta^\beta}{2h} \left[(\widetilde{W}_\mathrm{prin})_{x^\alpha x^\beta}(0,\eta)\widetilde{P}^\pm\right]_{\eta_\alpha}
\\
&+\frac14 h_{\eta_\alpha\eta_\beta}\,(\widetilde{W}_\mathrm{prin})_{x^\alpha x^\beta}(0,\eta)\,\widetilde{P}^\pm \left.\right\}+O(t^2).
\end{split}
\end{equation}
Let us compute the summands in \eqref{17 June 2019 formula 4} separately. To this end, let us put
\begin{equation*}
\label{17 June 2019 formula 5}
\begin{aligned}
&A_1:=\frac{1}{12h}\Ric_{\alpha\nu}(0)\,\eta^\alpha \widetilde{\sigma}^\nu(0) \widetilde{P}^\pm,
&&A_2:=-\frac{i\eta^\alpha}{2h}(\widetilde{W}_0)_{x^\alpha}(0)\,\widetilde{P}^\pm,\\
&A_3:=\left[ \frac{1}{3h}R^\mu{}_\alpha{}^\nu{}_\beta(0) \,\eta_\mu\eta_\nu\,\widetilde{\sigma}^\beta(0) \,\widetilde{P}^\pm \right]_{\eta_\alpha},
&&A_4:=\frac{\eta^\beta}{2h} \left[(\widetilde{W}_\mathrm{prin})_{x^\alpha x^\beta}(0,\eta)\widetilde{P}^\pm\right]_{\eta_\alpha},\\
&A_5:=\frac14 h_{\eta_\alpha\eta_\beta}\,(\widetilde{W}_\mathrm{prin})_{x^\alpha x^\beta}(0,\eta)\,\widetilde{P}^\pm.
\end{aligned}
\end{equation*}

\begin{itemize}
\item $A_1$: It ensues from elementary properties of $\widetilde{P}^\pm$ that
\begin{equation}
\label{P sigma P dirac}
\begin{split}
\widetilde{P}^\mp \widetilde{\sigma}^\alpha \widetilde{P}^\pm
&=
\widetilde{P}^\mp [\widetilde{W}_\mathrm{prin} \widetilde{P}^\pm]_{\eta_\alpha}-\widetilde{P}^\mp \widetilde{W}_\mathrm{prin} \widetilde{P}^\pm_{\eta_\alpha}=\pm 2h\,\widetilde{P}^\mp \widetilde{P}^\pm_{\eta_\alpha}.
\end{split}
\end{equation}
Hence
\begin{equation}
\label{formula for A1 spbps mp dirac}
\widetilde{P}^\mp A_1=\widetilde{P}^\mp \left(\pm\frac{1}{6}\Ric_{\alpha\beta}(0)\,\eta^\alpha\widetilde{P}^\pm_{\eta_\beta} \right).
\end{equation}

\item $A_2$: Combining Corollary~\ref{corollary taylor expansion W0 dirac} with the identity
\begin{equation}
\label{6 April 2019 formula 2}
h_{\eta_\alpha\eta_\beta}=\dfrac{h^2\,\delta^{\alpha\beta}-\eta^\alpha\eta^\beta}{h^3}
\end{equation}
and using \eqref{P sigma P dirac}, we get
\begin{equation}
\label{formula for A2 spbps mp dirac}
\widetilde{P}^\mp A_2=\widetilde{P}^\mp \left(\pm \frac14\Ric_{\alpha\beta}(0)\,\eta^\alpha\widetilde{P}^\pm_{\eta_\beta} \right).
\end{equation}

\item $A_3$: We have
\begin{equation*}
\label{17 June 2019 formula 13}
\begin{aligned}
A_3&=\left[ \frac{1}{3h}R^\mu{}_\alpha{}^\nu{}_\beta(0) \,\eta_\mu\eta_\nu\,\widetilde{\sigma}^\beta(0) \,\widetilde{P}^\pm \right]_{\eta_\alpha}=\frac{1}{3h}R^\mu{}_\alpha{}^\nu{}_\beta(0)\,\widetilde{\sigma}^\beta(0)\left[\eta_\mu\eta_\nu \,\widetilde{P}^\pm \right]_{\eta_\alpha}\\
&=-\frac{1}{3 h}\Ric_{\mu\nu}(0)\,\eta^\mu \,\widetilde{\sigma}^\nu(0)\,\widetilde{P}^\pm \pm\frac{1}{6\,h^2} \Ric_{\mu\beta}(0)\,\eta^\mu\eta^\beta\, \mathrm{Id},
\end{aligned}
\end{equation*}
so that, by \eqref{P sigma P dirac}, 
\begin{equation}
\label{formula for A3 spbps mp dirac}
\widetilde{P}^\mp A_3=\widetilde{P}^\mp \left( \mp\frac23\Ric_{\alpha\beta}(0)\,\eta^\alpha\widetilde{P}^\pm_{\eta_\beta}\pm\frac{1}{6\,h^2}\Ric_{\alpha\beta}(0)\,\eta^\alpha\eta^\beta\right).
\end{equation}

\item $A_4$: Recalling \eqref{Pauli matrices projected along Levi Civita framing dirac}, we have
\begin{equation*}
\label{17 June 2019 formula 11}
\begin{aligned}
A_4&=
\frac{\eta^\beta}{2h}(\widetilde{\sigma}^\mu)_{x^\alpha x^\beta}(0) \left[\eta_\mu\,\widetilde{P}^\pm\right]_{\eta_\alpha}\\
&=-\frac{1}{12h}\Ric_{\alpha\beta}(0)\,\eta^\alpha \widetilde{\sigma}^\beta(0)\,\widetilde{P}^\pm 
\mp
\frac{1}{24 \,h^2} \Ric_{\alpha\beta}(0)\,\eta^\alpha\eta^\beta\, \mathrm{Id},
\end{aligned}
\end{equation*}
so that, by \eqref{P sigma P dirac}, 
\begin{equation}
\label{formula for A4 spbps mp dirac}
\widetilde{P}^\mp A_4=\widetilde{P}^\mp \left( \mp\frac16\Ric_{\alpha\beta}(0)\,\eta^\alpha\widetilde{P}^\pm_{\eta_\beta}\mp\frac{1}{24\,h^2}\Ric_{\alpha\beta}(0)\,\eta^\alpha\eta^\beta\right).
\end{equation}

\item $A_5$: In view of \eqref{6 April 2019 formula 2} and \eqref{Pauli matrices projected along Levi Civita framing dirac}, we have
\begin{equation*}
\label{17 June 2019 formula 6}
\begin{aligned}
A_5&=\frac14 \left(\frac{\delta^{\alpha\beta}}{h}-\frac{\eta^\alpha\eta^\beta}{h^3} \right)\frac16 \left[ R^\mu{}_{\beta\nu\alpha}+R^\mu{}_{\alpha\nu\beta}\right](0)\,\widetilde{\sigma}^\nu(0)\,\eta_\mu\,\widetilde{P}^\pm\\
&=\frac1{12\,h}\,\Ric_{\mu\nu}(0)\,\eta^\mu\widetilde{\sigma}^\nu(0)\, \widetilde{P}^\pm,
\end{aligned}
\end{equation*}
so that, by \eqref{P sigma P dirac}, 
\begin{equation}
\label{formula for A5 spbps mp dirac}
\widetilde{P}^\mp A_5=\widetilde{P}^\mp \left(\pm \frac16\Ric_{\alpha\beta}(0)\,\eta^\alpha\widetilde{P}^\pm_{\eta_\beta}\right).
\end{equation}
\end{itemize}

Substituting \eqref{formula for A1 spbps mp dirac}, \eqref{formula for A2 spbps mp dirac}, \eqref{formula for A3 spbps mp dirac}, \eqref{formula for A4 spbps mp dirac} and  \eqref{formula for A5 spbps mp dirac} into \eqref{17 June 2019 formula 4} we arrive at \eqref{projection on other eigenspace subps dirac}.
\end{proof}

Let us now move to the second transport equation.
\begin{lemma}
\label{lemma projection along same eigenspace Dirac}
The projection onto the positive (resp.~negative) eigenspace of $\widetilde{W}_\mathrm{prin}$ of the subprincipal symbol of the positive (resp.~negative) propagator is given by
\begin{multline}
\label{projection on same eigenspace subps dirac}
\widetilde{P}^\pm(x^\pm,\xi^\pm)\widetilde{\mathfrak{a}}_{-1}^\pm(t;y,\eta)=\mp it\,\widetilde{P}^\pm\left[\frac1{24 h}\mathcal{R}(0)
+
\frac{1}{8h^3}\Ric_{\alpha\beta}(0)\,\eta^\alpha\,\eta^\beta
+
\frac1{4h}\Ric_{\alpha\beta}(0)\,\eta^\alpha\,\widetilde{P}^\pm_{\eta_\beta}
\right]\\
+O(t^2).
\end{multline}
\end{lemma}

\begin{proof}
We will establish formula \eqref{projection on same eigenspace subps dirac} by computing the second transport equation \eqref{third transport equation dirac} up to zeroth order in $t$ and then acting with $\widetilde{P}^\pm$ on the left.

With account of Lemma~\ref{lemma expansion fractur S minus 1 and S minus 2}, we have 
\begin{equation}
\label{15 April 2019 formula 30}
\begin{aligned}
\left.\mathfrak{S}^\pm_{-2}\widetilde{a}^\pm_{1}\right|_{t=0}&=\left.-\frac12 \dfrac{\partial^4}{\partial x^\alpha\partial\eta_\alpha \partial x^\beta \partial \eta_\beta} \left( \varphi_t^\pm+\widetilde{W}_\mathrm{prin}(x,\varphi_x^\pm) \right) \mathfrak{a}_0^\pm\right|_{x=0,t=0}\\
&=
-\frac12 \dfrac{\partial^4}{\partial x^\alpha\partial\eta_\alpha \partial x^\beta \partial \eta_\beta}\left(
\mp h\pm\frac1{3\,h} \,R^\gamma{}_\mu{}^\rho{}_\nu(0)\,\eta_\gamma\,\eta_\rho \,x^\mu\, x^\nu+O(\| x\|^3)
\right.\\
&\left.+\widetilde{\sigma}^\alpha(0)(\eta_\alpha+O(\| x \|^3))
\widetilde{P}^\pm\Bigr)
\right|_{x=0,t=0}\\
&=-\dfrac{\partial^2}{\partial\eta_\alpha \partial \eta_\beta} \left(\pm
\frac1{3\,h} \,R^\gamma{}_\alpha{}^\rho{}_\beta(0)\,\eta_\gamma\,\eta_\rho +\frac12 (\widetilde{W}_\mathrm{prin})_{x^\alpha x^\beta}(0,\eta)
\right) \widetilde{P}^\pm,
\end{aligned}
\end{equation}
\begin{equation}
\label{15 April 2019 formula 31}
\begin{aligned}
\left.\mathfrak{S}^\pm_{-1}\widetilde{a}^\pm_{0}\right|_{t=0}&=i\,\dfrac{\partial^2}{\partial x^\alpha\partial\eta_\alpha}\left[ \left( \mp h+\widetilde{W}_\mathrm{prin}(x,\eta) \right) \widetilde{\mathfrak{a}}^\pm_{-1}(0;y,\eta)-i (\widetilde{\mathfrak{a}}^\pm_{0})_t\right.\\
&\left.\left.+\left(\pm\frac{i \eta_\mu}{3h}\Ric^\mu{}_\nu(0)\,x^\nu+-\frac{i}6\operatorname{Ric}_{\mu\nu}(0)\,\widetilde{\sigma}^\mu(0)\,x^\nu+O(\|x\|^2)\right) \widetilde{P}^\pm+\widetilde{W}_0(x)\widetilde{P}^\pm\right]\right|_{t=0,x=0}\\
&=\mp\dfrac{\partial}{\partial\eta_\alpha}\left( \frac{\eta_\mu}{3h^{(j)}}\Ric^\mu{}_\alpha(0)\widetilde{P}^\pm \right)
+\frac16 \operatorname{Ric}_{\alpha\mu}(0)\,\widetilde{\sigma}^\mu(0) \,\widetilde{P}^\pm_{\eta_\alpha}
+i\, (\widetilde{W}_0)_{x^\alpha}(0)\widetilde{P}^\pm_{\eta_\alpha}
\end{aligned}
\end{equation}
and
\begin{equation}
\label{15 April 2019 formula 32}
\begin{aligned}
\left.\mathfrak{S}^\pm_{0}\widetilde{a}^\pm_{-1}\right|_{t=0}&=(\mp h+\widetilde{W}_\mathrm{prin}(0,\eta))\widetilde{\mathfrak{a}}^\pm_{-2}(0)-i\,(\widetilde{\mathfrak{a}}^\pm_{-1})_t|_{t=0}.
\end{aligned}
\end{equation}
In carrying out the above calculations we used Theorem~\ref{theorem small time expansion principal symbols dirac} and Lemma~\ref{lemma initial condition subprincipal symbol dirac}. Note that, when multiplying on the left by $\widetilde{P}^\pm$, the terms containing $\widetilde{\mathfrak{a}}^\pm_{-2}$ disappear. Summing up \eqref{15 April 2019 formula 30}, \eqref{15 April 2019 formula 31} and \eqref{15 April 2019 formula 32}, and projecting along $\widetilde{P}^\pm$, we obtain
\begin{equation}
\label{15 April 2019 formula 33}
\begin{aligned}
(\widetilde{P}^\pm \widetilde{\mathfrak{a}}^\pm_{-1})_t\,(0;y,\eta)=&i \widetilde{P}^\pm\dfrac{\partial^2}{\partial\eta_\alpha \partial \eta_\beta} \left[\left(
\pm \frac1{3h} \,R^\gamma{}_\alpha{}^\rho{}_\beta(0)\,\eta_\gamma\,\eta_\rho + \frac12(\widetilde{W}_\mathrm{prin})_{x^\alpha x^\beta}(0,\eta)
\right) \widetilde{P}^\pm\right]\\
&\pm i \widetilde{P}^\pm\frac{\partial}{\partial\eta_\alpha}\left[ \frac{\eta_\mu}{3h}\Ric^\mu{}_\alpha(0)\widetilde{P}^\pm \right]
-\frac{i}6 \operatorname{Ric}_{\alpha\mu}(0)\,\widetilde{\sigma}^\mu(0) \,\widetilde{P}^\pm_{\eta_\alpha}
\\&+\widetilde{P}^\pm (\widetilde{W}_0)_{x^\alpha}(0)\widetilde{P}^\pm_{\eta_\alpha}+O(t).
\end{aligned}
\end{equation}
Using the identity
\begin{equation*}
\pm\dfrac{\partial}{\partial \eta_\beta}\left[ \frac1{3h} \,R^\gamma{}_\alpha{}^\rho{}_\beta(0)\,\eta_\gamma\,\eta_\rho  \widetilde{P}^\pm\right]=\mp \frac{\eta_\mu}{3h}\Ric^\mu{}_\alpha(0)\widetilde{P}^\pm \pm \frac1{3h} \,R^\gamma{}_\alpha{}^\rho{}_\beta(0)\,\eta_\gamma\,\eta_\rho  \widetilde{P}^\pm_{\eta_\beta},
\end{equation*}
formula \eqref{15 April 2019 formula 33} becomes
\begin{equation}
\label{15 April 2019 formula 33 bis}
\begin{aligned}
(\widetilde{P}^\pm \widetilde{\mathfrak{a}}^\pm_{-1})_t\,(0;y,\eta)=&\frac{i}2 \widetilde{P}^\pm\dfrac{\partial^2}{\partial\eta_\alpha \partial \eta_\beta} \left[(\widetilde{W}_\mathrm{prin})_{x^\alpha x^\beta}(0,\eta) \widetilde{P}^\pm\right] \pm i \widetilde{P}^\pm\frac{\partial}{\partial\eta_\alpha}\left[  \frac1{3h} \,R^\gamma{}_\alpha{}^\rho{}_\beta(0)\,\eta_\gamma\,\eta_\rho\,\widetilde{P}^\pm_{\eta_\beta} \right]
\\
&
-\frac{i}6 \operatorname{Ric}_{\alpha\mu}(0)\,\widetilde{\sigma}^\mu(0) \,\widetilde{P}^\pm_{\eta_\alpha}+\widetilde{P}^\pm (\widetilde{W}_0)_{x^\alpha}(0)\widetilde{P}^\pm_{\eta_\alpha}+O(t).
\end{aligned}
\end{equation}
Let us put
\begin{equation}
\label{13 April 2019 forumla 2} 
\begin{aligned}
&B_1:=\frac{i}2\left[(\widetilde{W}_\mathrm{prin})_{x^\alpha x^\beta}(0,\eta) \widetilde{P}^\pm\right]_{\eta_\alpha\eta_\beta},
&&B_2:=\frac{i}{3h} \,R^\gamma{}_\alpha{}^\rho{}_\beta(0)\,\eta_\gamma\,\eta_\rho\,\widetilde{P}^\pm_{\eta_\beta},\\
&B_3:=-\frac{i}6 \operatorname{Ric}_{\alpha\mu}(0)\,\widetilde{\sigma}^\mu(0) \,\widetilde{P}^\pm_{\eta_\alpha}+(\widetilde{W}_0)_{x^\alpha}(0)\widetilde{P}^\pm_{\eta_\alpha}.
\end{aligned}
\end{equation}

\begin{itemize}
\item $B_1$: It follows from \eqref{projections in terms of principal symbol dirac}, Corollary~\ref{corollary projection Pauli matrices} and \eqref{6 April 2019 formula 2} that
\begin{equation}
\label{13 April 2019 forumla 4}
\begin{split}
\widetilde{P}^\pm\,B_1
&=\frac{i}{6}\widetilde{P}^\pm\left[ 
 R^\mu{}_{\beta\nu\alpha}\,\eta_\mu\,\widetilde{\sigma}^\nu\frac12 \left( \mathrm{Id}\pm\frac{\eta_\rho\, \widetilde{\sigma}^\rho}{h} \right)\right]_{\eta_\alpha\eta_\beta}=\pm\frac{i}{12}\widetilde{P}^\pm
 R^\mu{}_{\beta\nu\alpha}\,\widetilde{\sigma}^\nu\, \widetilde{\sigma}^\rho
 \left(
 \frac{\eta_\mu\,\eta_\rho}{h}
  \right)_{\eta_\alpha\eta_\beta}\\
&=\pm\left(-\frac{i}{12 h}\mathcal{R}(0)+\frac{i}{12h^3}\Ric_{\alpha\beta}(0)\,\eta^\alpha\eta^\beta\right)\,\widetilde{P}^\pm.
\end{split}
\end{equation}

\item $B_2$: Differentiating \eqref{projections in terms of principal symbol dirac} with respect to $\eta_\beta$ yields
\begin{equation}
\begin{split}
\label{14 April 2019 formula 1}
\widetilde{P}^\pm_{\eta_\beta}
&
=
\pm
\frac{1}{2h} (\widetilde{W}_\mathrm{prin})_{\eta_\beta}
\mp
\frac{\eta^\beta}{2h^3}\widetilde{W}_\mathrm{prin}
\\
&
=
\pm\frac{1}{2}\left( \frac{\widetilde{\sigma}^\beta}{h}-\frac{\eta^\beta\,\eta_\rho \,\widetilde{\sigma}^\rho}{h^3} \right).
\end{split}
\end{equation}

Substituting \eqref{14 April 2019 formula 1} into $B_2$ in \eqref{13 April 2019 forumla 2} we obtain
\begin{equation}
\label{14 April 2019 formula 2}
\begin{split}
\pm\widetilde{P}^\pm\,B_2&=\frac{i}6 \widetilde{P}^\pm\,R^{\mu}{}_\alpha{}^\nu{}_\beta \left[\widetilde{\sigma}^\beta \left(\frac{\eta_\mu\eta_\nu}{h^2} \right)_{\eta_\alpha} 
+
\widetilde{\sigma}^\rho\left(\frac{\eta_\mu\eta_\nu\eta^\beta\eta_\rho}{h^4} \right)_{\eta_\alpha}
\right]\\
&=-\frac{i}{6\,h^2} \widetilde{P}^\pm\,\Ric^{\mu}{}_\beta(0)\,\eta_\mu\,\widetilde{\sigma}^\beta(0).
\end{split}
\end{equation}

\item $B_3$: By means of Corollary~\ref{corollary taylor expansion W0 dirac} and formula \eqref{14 April 2019 formula 1} we get
\begin{equation}
\label{14 April 2019 formula 4}
\begin{split}
B_3&=\left( \pm \frac{i}{4} \Ric_{\alpha\beta}(0)\,\sigma^\alpha\mp \frac{i}{6} \Ric_{\alpha\beta}(0)\,\sigma^\alpha \right) \frac{1}{2}\left( \frac{\sigma^\beta}{h}-\frac{\eta^\beta \,\sigma^\rho\eta_\rho}{h^3} \right)\\
&=\pm \frac{i}{24} \Ric_{\alpha\beta}(0)\,\sigma^\alpha \left( \frac{\sigma^\beta}{h}-\frac{\eta^\beta \,\sigma^\rho\eta_\rho}{h^3} \right)\\
&=\pm \left(\frac{i}{24h} \mathcal{R}(0)\,\mathrm{Id}-\frac{i}{24h^3}\Ric_{\alpha\beta}(0)\, \eta^\beta \eta_\rho \sigma^\alpha(0)\,\sigma^\rho(0)\right).
\end{split}
\end{equation}
Now, since $\widetilde{P}^\pm\,\widetilde{\sigma}^\rho(0)\eta_\rho=\widetilde{P}^\pm\,\widetilde{W}_\mathrm{prin}(0,\eta)=\pm h\, \widetilde{P}^\pm$ and $\widetilde{\sigma}^\alpha\widetilde{\sigma}^\rho=-\widetilde{\sigma}^\rho \widetilde{\sigma}^\alpha+2\,\delta^{\alpha\rho}\,\mathrm{Id}$, formula \eqref{14 April 2019 formula 4} implies
\begin{equation}
\label{14 April 2019 formula 5}
\begin{split}
\widetilde{P}^\pm B_3
&=\widetilde{P}^\pm\left(\pm  
\frac{i}{24\,h} \mathcal{R}(0)
+\frac{i}{24 \,h^2}\Ric_{\alpha\beta}(0)\,\eta^\alpha\,\widetilde{\sigma}^\beta
\mp\frac{i}{12 h^3}\Ric_{\alpha\beta}(0)\,\eta^\alpha\eta^\beta
\right).
\end{split}
\end{equation}
\end{itemize}
Summing up \eqref{13 April 2019 forumla 4}, \eqref{14 April 2019 formula 2} and \eqref{14 April 2019 formula 5} we arrive at
\begin{equation}
\label{14 April 2019 formula 5 bis}
(\widetilde{P}^\pm \widetilde{\mathfrak{a}}^\pm_{-1})_t\,(0;y,\eta)=i \widetilde{P}^\pm \left(\mp\frac{1}{24h}\mathcal{R}(0)-\frac1{8h^2} \operatorname{Ric}_{\alpha\beta}(0)\,\eta^\alpha \,\widetilde{\sigma}^\beta(0) \right) +O(t).
\end{equation}
A straightforward calculation shows that
\begin{equation*}
\widetilde{P}^\pm \widetilde{\sigma}^\alpha=\pm \widetilde{P}^\pm\left( \frac{\eta^\alpha}{h}+2h \widetilde{P}^\pm_{\eta_\alpha}\right).
\end{equation*}
Substituting the above expression into \eqref{14 April 2019 formula 5 bis} and integrating in time with initial condition \eqref{initial condition subprincipal symbol dirac}, we obtain \eqref{projection on same eigenspace subps dirac}.

\end{proof}

The pieces of information from Lemma~\ref{lemma projection along other eigenspace Dirac} and Lemma~\ref{lemma projection along same eigenspace Dirac} can be combined to give the following result.

\begin{theorem}
\label{theorem small time expansion subprincipal symbols levi-civita dirac}
For the choice of the Levi-Civita framing, the subprincipal symbols of the positive and negative propagators admit the following small time expansion:
\begin{equation}
\label{small time expansion subprincipal symbols levi-civita dirac}
\widetilde{\mathfrak{a}}_{-1}^\pm(t;y,\eta)=\mp it
\left(
\frac{1}{24\,h}\mathcal{R(}y)\,\widetilde{P}^\pm(y,\eta)
-
\frac{1}{8 h^2}\operatorname{Ric}_{\alpha\beta}(y)\,\eta^\alpha\,(\widetilde{W}_\mathrm{prin})_{\eta_\beta}(y,\eta)
\right)
+
O(t^2)\,.
\end{equation}
\end{theorem}

\begin{proof}
Summing up formulae \eqref{projection on other eigenspace subps dirac} and \eqref{projection on same eigenspace subps dirac}, we obtain
\begin{equation*}
\begin{split}
\widetilde{\mathfrak{a}}_{-1}^\pm(t;y,\eta)&=\mp
\frac{it}{24\,h}\mathcal{R(}y)\,\widetilde{P}^\pm(y,\eta)
\mp\frac{it}{8h^4}\operatorname{Ric}_{\alpha\beta}(y)\,\eta^\alpha\,\eta^\beta\,\widetilde{W}_\mathrm{prin}(y,\eta)\\
&
-
\frac{it}{4h}\operatorname{Ric}_{\alpha\beta}(y)\,\eta^\alpha\,\widetilde{P}^\pm_{\eta_\beta}(y,\eta)
+
O(t^2)\,.
\end{split}
\end{equation*}
The substitution of \eqref{14 April 2019 formula 1} into the RHS of the above equation gives \eqref{small time expansion subprincipal symbols levi-civita dirac}.
\end{proof}

Note that if the manifold is Ricci-flat then $\widetilde{\mathfrak{a}}_{-1}^\pm(t;y,\eta)=O(t^2)$.

\subsection{Invariant reformulation}

In the previous subsections we derived the quite elegant and compact formulae \eqref{small time expansion principal LC dirac} and \eqref{small time expansion subprincipal symbols levi-civita dirac}, which were obtained under the assumption that the chosen framing is the Levi-Civita framing at $y$. Now the task at hand is to obtain similar formulae for the Dirac operator $W$ corresponding to an arbitrary framing $\{e_j\}_{j=1}^3$.

Given a framing $\{e_j\}_{j=1}^3$ and a point $y\in M$, there exits a special unitary matrix-function $G$, defined in a neighbourhood of $y$, such that $\{e_j\}_{j=1}^3$ and the Levi-Civita framing $\{\widetilde{e}_j\}_{j=1}^3$ generated by $\{e_j\}_{j=1}^3$ at $y$ are related in accordance with 
\begin{equation}
\label{new framing generated by G Levi-civita}
{e}_j{}^\alpha(x)=\frac12 \operatorname{tr}(s_j\,G^*(x)\,s^k\,G(x))\,\widetilde{e}_k{}^\alpha(x), \qquad G(y)=\mathrm{Id},
\end{equation}
cf.~\eqref{new framing generated by G} and \eqref{gauge transfomation in terms of O}. The symbols $\widetilde{\mathfrak{a}}^\pm$ and $\mathfrak{a}^\pm$ are related as
\begin{equation}
\label{relation between mathfrak a and tilde mathfrak a}
\mathfrak{a}^\pm=\mathfrak{S}^\pm[G^*(x) \, \widetilde{\mathfrak{a}}^\pm \,G(y)],
\end{equation}
cf.\ Section~\ref{Global propagator for the massless Dirac operator}. Note that on the RHS of \eqref{relation between mathfrak a and tilde mathfrak a} the transformed symbol is acted upon by amplitude-to-symbol operators \eqref{definition amplitude to symbol with j}. The latter are needed because the gauge transformation $G$ introduces an $x$-dependence in the amplitude, which has to be excluded.

Working in normal coordinates centred at $y$, formula \eqref{relation between mathfrak a and tilde mathfrak a}, combined with \eqref{x plus minus normal coordinates dirac} and \eqref{small time expansion principal LC dirac}, implies
\begin{equation}
\label{transformation principal symbol under gauge transformation dirac}
\begin{split}
\mathfrak{a}_0^\pm &=G^*(x^\pm)P^\pm\\
&=P^\pm \pm\frac{t\,\eta^\alpha}{h} G^*_{x^\alpha}(y) P^\pm +\frac{t^2}{2}\frac{\eta^\alpha\eta^\beta}{h^2} G^*_{x^\alpha x^\beta}(y)P^\pm
+O(t^3)\\
&=P^\pm \pm\frac{t\,\eta^\alpha}{h} \,\nabla_\alpha
G^*(y) \,P^\pm +\frac{t^2}{2}\frac{\eta^\alpha\eta^\beta}{h^2} \nabla_\alpha\nabla_\beta G^*(y)P^\pm
+O(t^3).
\end{split}
\end{equation}
Similarly, by means of \eqref{x plus minus normal coordinates dirac} and Lemma~\ref{lemma expansion fractur S minus 1 and S minus 2}, from \eqref{relation between mathfrak a and tilde mathfrak a} we get
\begin{equation}
\label{transformation subprincipal symbol under gauge transformation dirac}
\begin{split}
\mathfrak{a}_{-1}^\pm
&
=
\mathfrak{S}^\pm_{-1} [G^*(x)\,\mathfrak{\widetilde a}^\pm_{0}]+\mathfrak{S}^\pm_0[G^*(x)\,\mathfrak{\widetilde a}^\pm_{-1}]
\\
&
=
i\,G^*_{x^\alpha}(y)\,P^\pm_{\eta_\alpha}
\pm 
it\,G^*_{x^\alpha x^\beta}(y) 
\left(h_{\eta_\beta}\, P^\pm_{\eta_\alpha}+\frac12 h_{\eta_\alpha\eta_\beta}\,P^\pm \right)\\
&+ \widetilde{\mathfrak{a}}^\pm_{-1}+O(t^2)
\\
&
=
i\,\nabla_\alpha G^*(y)\,P^\pm_{\eta_\alpha}
\pm 
it\,\nabla_\alpha\nabla_\beta G^*(y) 
\left(h_{\eta_\beta}\, P^\pm_{\eta_\alpha}+\frac12 h_{\eta_\alpha\eta_\beta}\,P^\pm \right)\\
&+ \widetilde{\mathfrak{a}}^\pm_{-1}+O(t^2).
\end{split}
\end{equation}

The last step towards expressing \eqref{transformation principal symbol under gauge transformation dirac} and \eqref{transformation subprincipal symbol under gauge transformation dirac} invariantly is writing $\nabla G$ and $\nabla \nabla G$ in terms of geometric invariants.  Theorem~\ref{theorem covariant derivatives gauge transformation} tells us that
\begin{equation}
\label{covariant derivative gauge transformation special case}
\nabla_\alpha G(y)=-\frac{i}{2} \overset{*}{K}{}_{\alpha\beta}(y)\,\sigma^\beta(y).
\end{equation}
The following theorem provides an expression for the second covariant derivatives of the gauge transformation.

\begin{theorem}
\label{theorem covariant derivatives gauge transformation Levi-civita}
Let us fix a point $y$ and let the special unitary matrix-function $G$ be such that our framing $\{e_j\}_{j=1}^3$ and the Levi-Civita framing $\{\widetilde e_j\}_{j=1}^3$ generated by $\{e_j\}_{j=1}^3$ at $y$ are related in accordance with \eqref{new framing generated by G Levi-civita} in a neighbourhood of $y$.
Then we have
\begin{equation}
\label{double covariant derivative of gauge transformation G Levi-civita}
\nabla_\alpha\nabla_\beta \,G(y)=-\frac{i}{4}
\left(
\nabla_\alpha \overset{*}{K}{}_{\beta\mu}(y)
+
\nabla_\beta \overset{*}{K}{}_{\alpha\mu}(y)
\right)
\sigma^\mu(y)
-\frac{1}{4} \overset{*}{K}{}_{\alpha\mu}(y)\overset{*}{K}{}_\beta{}^\mu(y)\operatorname{Id},
\end{equation}
where $K$ is the contorsion tensor of the Weitzenb\"ock connection (see Appendix~\ref{Geometric properties of the Weitzenb\"ock connection}) associated with the framing $\{ e_j\}_{j=1}^3$ and the star stands for the Hodge dual
applied in the first and third indices (see formula \eqref{hodge star contorsion}).
\end{theorem}

\begin{proof}
The proof is given in Appendix~\ref{appendix proof double covariant derivative gauge transformation}.
\end{proof}

\begin{remark}
Note that, remarkably, the curvature of the Levi-Civita connection does not appear in the RHS of \eqref{double covariant derivative of gauge transformation G Levi-civita}.
\end{remark}

Substituting \eqref{covariant derivative gauge transformation special case} and \eqref{double covariant derivative of gauge transformation G Levi-civita} into \eqref{transformation principal symbol under gauge transformation dirac} and \eqref{transformation subprincipal symbol under gauge transformation dirac}  we arrive at the following result.

\begin{theorem}
\label{theorem invariant expression principal and subprincipal symbols dirac}
Let $W$ be the Dirac operator \eqref{massless dirac definition equation}. Then the the principal and subprincipal symbols of the positive and negative propagators admit the following small time expansions:
\begin{equation}
\begin{split}
\mathfrak{a}_0^\pm 
&
=
\left[\operatorname{Id} \pm \frac{it}{2}h_{\eta_\alpha}\,\overset{*}{K}{}_{\alpha\beta}\,(W_\mathrm{prin})_{\eta_\beta}\right] P^\pm\\
&
+\frac{t^2}{8}\frac{\eta^\alpha\eta^\beta}{h^2} \left[i\bigl(
\nabla_\alpha \overset{*}{K}{}_{\beta\mu}(y)
+
\nabla_\beta \overset{*}{K}{}_{\alpha\mu}(y)
\bigr)
(W_\mathrm{prin})_{\eta_\mu}-\overset{*}{K}{}_{\alpha\mu}(y)\overset{*}{K}{}_\beta{}^\mu(y)\right] P^\pm +O(t^3),
\end{split}
\end{equation}
\begin{equation}
\begin{split}
\mathfrak{a}_{-1}^\pm
&
=
-\frac12 \overset{*}{K}{}_{\alpha\beta}\,(W_\mathrm{prin})_{\eta_\beta}\,P^\pm_{\eta_\alpha}
\\
&\mp it
\left(
\frac{1}{24\,h}\mathcal{R}\,P^\pm
-
\frac{1}{8 h^2}\operatorname{Ric}_{\alpha\beta}\,\eta^\alpha\,({W}_\mathrm{prin})_{\eta_\beta}
\right)
\\
&\mp
\frac{t}4\,\left(
\nabla_\alpha \overset{*}{K}{}_{\beta\mu}
+
\nabla_\beta \overset{*}{K}{}_{\alpha\mu}
\right)
(W_\mathrm{prin})_{\eta_\mu}
\left(h_{\eta_\beta}\, P^\pm_{\eta_\alpha}+\frac12 h_{\eta_\alpha\eta_\beta}\,P^\pm \right)\\
&\mp\frac{it}{4} \overset{*}{K}{}_{\alpha\mu}\overset{*}{K}{}_\beta{}^\mu
\left(h_{\eta_\beta}\, P^\pm_{\eta_\alpha}+\frac12 h_{\eta_\alpha\eta_\beta}\,P^\pm \right)\\
&+O(t^2),
\end{split}
\end{equation}
where $\overset{*}{K}$ denotes the Hodge dual in the first and third indices of the contorsion tensor of the Weitzenb\"ock connection associated with the framing $\{e_j\}_{j=1}^3$.
\end{theorem}

\section{An application: spectral asymptotics}
\label{An application: spectral asymptotics}

In this section we will compute the third Weyl coefficient for the Dirac operator. In doing so we will use the same notation as in Section~\ref{Statement of the problem} --- recall in particular formulae \eqref{expansion for mollified derivative of counting function}, \eqref{local counting functions} and the definition of the function $\mu$.

\begin{theorem}
\label{theorem about third Weyl coefficient}
The third local Weyl coefficients for the Dirac operator are
\begin{equation}
\label{theorem about third Weyl coefficient formula}
c_0^\pm(y)=-\frac1{48\pi^2}\mathcal{R}(y),
\end{equation}
where $\mathcal{R}$ is scalar curvature. 
\end{theorem}

\begin{proof}
Let us fix a point $y\in M$ and choose normal geodesic coordinates $x$ centred at~$y$.
Let us also choose a Levi-Civita framing $\{\widetilde e_j\}_{j=1}^3$,
see Definition~\ref{definition Levi-Civita framing};
here we make use of the fact that Weyl coefficients
do not depend on the choice of framing.

We have
\begin{equation}
\label{hyperbolic approach equation 1 plus}
(N_+'*\mu)(y,\lambda)
\,=\,
\mathcal{F}^{-1}
\left[
\mathcal{F}
\left[
(N_+'*\mu)
\right]
\right](y,\lambda)
\,=\,
\mathcal{F}^{-1}
\left[
\operatorname{tr}u_+(t,y,y)\,\hat\mu(t)
\right],
\end{equation}
\begin{equation}
\label{hyperbolic approach equation 1 minus}
(N_-'*\mu)(y,\lambda)
\,=\,
\mathcal{F}^{-1}
\left[
\mathcal{F}
\left[
(N_-'*\mu)
\right]
\right](y,\lambda)
\,=\,
\mathcal{F}^{-1}
\left[
\operatorname{tr}\overline{u_-(t,y,y)}\,\hat\mu(t)
\right],
\end{equation}
where $u_\pm$ is the Schwartz kernel
of the propagator $U^{\pm}$
and $\,\operatorname{tr}\,$ stands for the matrix trace.
Note that at each point of the manifold
the quantity $\,\operatorname{tr}u_\pm(t,y,y)\,$
is a distribution in the variable $t$ and
the construction presented in preceding sections allows us to write down this distribution
explicitly, modulo a smooth function.

Our task is to substitute
\eqref{integral kernel of U pm}
into the right-hand sides of
\eqref{hyperbolic approach equation 1 plus}
and
\eqref{hyperbolic approach equation 1 minus}
and expand the resulting quantities in powers of $\lambda$ as $\lambda\to+\infty$.
Thus, the problem reduces to the analysis of explicit integrals in four variables,
$\eta_1,\eta_2,\eta_3$ and $t$,
depending on the parameter~$\lambda\,$.
In what follows we drop the $y$ in our intermediate calculations.

The construction presented in preceding sections
tells us that the only singularity of the distribution $\,\operatorname{tr}u_\pm(t,y,y)\,\hat\mu(t)\,$ is at $t=0$.
Hence, in what follows, we can assume that the support of $\hat\mu$ is arbitrarily small.
In particular, this allows us to use the real-valued ($\epsilon=0$) Levi-Civita phase functions $\varphi^\pm$.

Theorems \ref{theorem small time expansion principal symbols dirac}
and \ref{theorem small time expansion subprincipal symbols levi-civita dirac}
imply that
\begin{equation}
\label{small time expansion of trace of principal symbol}
\operatorname{tr}\widetilde {\mathfrak{a}}^{\pm}_0(t;\eta)=
1\,,
\end{equation}
\begin{equation}
\label{small time expansion of trace of subprincipal symbol}
\operatorname{tr}\widetilde {\mathfrak{a}}^{\pm}_{-1}(t;\eta)=
\mp\,
\frac{i}{24\,\|\eta\|}\,\mathcal{R}\,t+O(t^2)\,.
\end{equation}

Formula \cite[(B.11)]{wave} reads $\,\varphi^+(t,\eta)=-\|\eta\|\,t+O(t^4)\,$, which,
in view of \eqref{relation between the two Levi-Civita phase functions}, implies
\begin{equation}
\label{first three Weyl coefficients theorem proof equation 4}
\varphi^\pm(t,\eta)=\mp\|\eta\|\,t+O(t^4)\,.
\end{equation}

Using formulae
\eqref{hyperbolic approach equation 1 plus}--\eqref{first three Weyl coefficients theorem proof equation 4}
and arguing as in \cite[Appendix B]{wave}, we conclude that
\begin{multline}
\label{first three Weyl coefficients theorem proof equation 12}
(N_\pm'*\mu)(y,\lambda)
=
\frac{S_{2}}{(2\pi)^{4}}\int_{\mathbb{R}^{2}}
\left(
r^{2}-\dfrac{1}{24}\,\mathcal{R}
\right)
e^{i(\lambda-r)t}\,\hat\mu(t)\,dr\,dt
\\
+\,O(\lambda^{-1})
\quad
\text{as}
\quad
\lambda\to+\infty,
\end{multline}
where $S_2=4\pi$ is the surface area of the 2-sphere.
But
\begin{equation*}
\label{integral in r and t weyl dirac}
\frac1{2\pi}
\int_{\mathbb{R}^{2}}
r^m\,
e^{i(\lambda-r)t}\,\hat\mu(t)\,dr\,dt
=\lambda^m,\qquad m=0,1,2,\ldots,
\end{equation*}
so \eqref{first three Weyl coefficients theorem proof equation 12} can be rewritten as
\begin{equation*}
\label{first three Weyl coefficients theorem proof equation 12 continued}
(N_\pm'*\mu)(y,\lambda)
=
\frac1{2\pi^2}\,\lambda^2
-\frac1{48\pi^2}\,\mathcal{R}(y)
+\,O(\lambda^{-1})
\quad
\text{as}
\quad
\lambda\to+\infty.
\end{equation*}
\end{proof}

\begin{remark}
Let us compare the spectrum of the Dirac operator with the spectrum of the Laplacian.
Working on the same 3-manifold, let $\Delta$ be the Laplace--Beltrami operator and let $N(y,\lambda)$
be the local counting function for the operator $\sqrt{-\Delta}\,$. Then
\[
(N'*\mu)(y,\lambda)=
c_{2}(y)\,\lambda^{2}
+
c_{1}(y)\,\lambda
+
c_{0}(y)
+\dots
\quad
\text{as}
\quad
\lambda\to+\infty,
\]
where the values of the first three Weyl coefficients are provided by \cite[Theorem B.2]{wave}.
Comparing these with
\eqref{first two weyl coefficients}
and
\eqref{theorem about third Weyl coefficient formula},
we conclude that
\[
c^\pm_{2}(y)=c_{2}(y)\,,
\qquad
c^\pm_{1}(y)=c_{1}(y)=0\,,
\qquad
c^\pm_{0}(y)=-\frac12\,c_{0}(y)\,.
\]
We see that the large (in modulus) eigenvalues of the Dirac operator
are distributed approximately the same way as the eigenvalues of the operator $\sqrt{-\Delta}\,$,
differing only in the third Weyl coefficient.
\end{remark}

\begin{remark}
\label{remark about third weyl coefficient}
There are, of course, alternative ways of computing the third Weyl coefficients. One can, for example, calculate $c_0^\pm$ by examining the quantities
\[
\operatorname{Tr} e^{-W^2 \,t} \quad\text{and}\quad\operatorname{Tr} W\,e^{-W^2 \,t},
\]
which are related to the counting functions via the Mellin transform, as in \cite{DuGu, branson}. See also \cite{chamsedine}.
\end{remark}

\section{Examples}
\label{Examples}

In this section we present two explicit examples, which show how our constructions work in practice and  which give us an opportunity to double-check our formulae.

The specific choice of examples is motivated by the fact that the first, $M=\mathbb{S}^3$, is isotropic in momentum whereas the second, $M=\mathbb{S}^2\times \mathbb{S}^1$, is anisotropic in momentum.

\subsection{The case $M=\mathbb{S}^3$}
\label{The case 3-sphere}

Let $\mathbb{R}^4$ be Euclidean space equipped with Cartesian coordinates $\textbf{x}^\alpha$, $\alpha=1,2,3,4$, and put
\[
\widehat{\textbf{e}}_4=\begin{pmatrix}
0\\0\\0\\1
\end{pmatrix}.
\]
Consider the 3-sphere\footnote{We shifted the sphere so as to place the south pole at the origin.}
\[
\mathbb{S}^3:=\{\textbf{x}+\widehat{\textbf{e}}_4 \in \mathbb{R}^4\,|\,\|\textbf{x}\|=1\}
\]
with orientation prescribed in accordance with \cite[Appendix~A]{sphere}, equipped with the standard round metric $g$ and with the global framings $\{V_{\pm,k}\}_{k=1}^3$ defined as the restriction to $\mathbb{S}^3$ of the vector fields in $\mathbb{R}^4$
\begin{equation}
\label{31 January 2019 formula 8}
\begin{aligned}
&\mathbf{V}_{\pm,1}:=(1-\textbf{x}^4)\dfrac{\partial}{\partial \textbf{x}^1}\mp \textbf{x}^3\frac{\partial}{\partial \textbf{x}^2} \pm \textbf{x}^2\dfrac{\partial}{\partial \textbf{x}^3}+\textbf{x}^1\dfrac{\partial}{\partial \textbf{x}^4},\\
&\mathbf{V}_{\pm,2}:=
\pm \textbf{x}^3\dfrac{\partial}{\partial \textbf{x}^1}+(1-\textbf{x}^4)\frac{\partial}{\partial \textbf{x}^2}\mp \textbf{x}^1 \dfrac{\partial}{\partial \textbf{x}^3}+\textbf{x}^2\dfrac{\partial}{\partial \textbf{x}^4},\\
&\mathbf{V}_{\pm,3}:=
\mp \textbf{x}^2\dfrac{\partial}{\partial \textbf{x}^1}\pm \textbf{x}^1\frac{\partial}{\partial \textbf{x}^2} +(1-\textbf{x}^4)\dfrac{\partial}{\partial \textbf{x}^3}+\textbf{x}^3\dfrac{\partial}{\partial \textbf{x}^4}.\\
\end{aligned}
\end{equation}
It is easy to check that the vector fields \eqref{31 January 2019 formula 8} are tangent to $\mathbb{S}^3$, so that they restrict to smooth vector fields on the 3-sphere.
Note that \eqref{31 January 2019 formula 8} is an adaptation of \cite[Eqn.~(C.1)]{sphere} to the case at hand.

Let us introduce coordinates on $\mathbb{S}^3$ with the north pole excised by stereographically projecting it onto the hyperplane tangent to the 3-sphere at the south pole. The stereo\-graphic map
is given by
\begin{equation*}
\label{31 January 2019 formula 2}
\sigma:
\mathbb{R}^3
\to 
\mathbb{S}^3
\setminus 
\begin{pmatrix}
0\\0\\0\\2
\end{pmatrix} ,
\qquad
\begin{pmatrix}
u\\v\\w
\end{pmatrix}
\mapsto
\begin{pmatrix}
\textbf{x}^1\\\textbf{x}^2\\\textbf{x}^3\\\textbf{x}^4
\end{pmatrix}
=
\dfrac{1}{1+f^2}
\begin{pmatrix}
u\\
v\\
w\\
2f^2
\end{pmatrix},
\end{equation*}
where
\begin{equation*}
\label{31 January 2019 formula 3}
f^2:=\frac14 \,(u^2+v^2+w^2).
\end{equation*}
It is easy to see that the coordinate system $(u,v,w)$ has positive orientation.

In stereographic coordinates the metric reads 
\begin{equation}
\label{31 January 2019 formula 5}
g=\dfrac{1}{(1+f^2)^2}\left[\dr u^2+\dr v^2+\dr w^2  \right]
\end{equation}
and our framings are given by
\begin{equation}
\label{1 February 2019 formula 3}
\begin{aligned}
&2V_{\pm,1}=
(2 - 2f^2+u^2)\frac{\partial}{\partial u}
+
 (u v \mp 2 w)\frac{\partial}{\partial v}
+
(u w\pm 2v)\frac{\partial}{\partial w},
\\
&
2V_{\pm,2}=
(u v \pm 2 w)\frac{\partial}{\partial u}
+
(2 - 2f^2+v^2)\frac{\partial}{\partial v}
+  
(v w\mp 2u)\frac{\partial}{\partial w},
\\
&
2V_{\pm,3}=
(u w\mp 2v)\frac{\partial}{\partial u}
+
(v w\pm 2u)\frac{\partial}{\partial v}
+
(2 - 2f^2+w^2)\frac{\partial}{\partial w}.
\end{aligned}
\end{equation} 
A straightforward calculation shows that $\{V_{\pm,k}\}_{k=1}^3$ are positively oriented framings formed by (orthonormal) smooth Killing vector fields with respect to the metric $g$.

The framings $\{V_{\pm,k}\}_{k=1}^3$ define, via \eqref{massless dirac definition equation}, two Dirac operators $W_\pm$ related in accordance with
\begin{equation*}
W_-=G^*W_+G,
\end{equation*}
where
\begin{equation}
\label{5 February 2019 formula 2}
G:=
\dfrac{1}{4(1+f^2)}
\begin{pmatrix}
u^2+v^2+(w-2\ir)^2 & 4(v-\ir u)\\
-4(v+\ir u)& u^2+v^2+(w+2\ir)^2
\end{pmatrix}
.
\end{equation}
is the $SU(2)$ gauge transformation relating the two framings via \eqref{new framing generated by G Levi-civita} with $\widetilde{e}_k=V_{+,k}$ and $e_k=V_{-,k}$.

Let us deal with $W_+$ first. On account of the symmetries of the 3-sphere, we will write formulae for principal and subprincipal symbols of the propagator of $W_+$ at the south pole ($y=(0,0,0)$) for the choice of momentum $\overline{\eta}=(0,0,1)$. 

The principal symbol $(W_+)_\mathrm{prin}$ has eigenvalues $h^\pm(y,\eta)=\pm\|\eta\|$, whose Hamiltonian flows in stereographic coordinates read
\begin{equation}
z^\pm(t;0,\eta)=\pm 2\tan(t/2) \frac{\eta}{\|\eta\|}, \qquad \xi^\pm(t;0,\eta)=\cos^2(t/2) \,\eta,
\end{equation}
see also formula \eqref{relation between two geodesic flows}. Direct inspection of the parallel transport equation \eqref{definition of parallel transport of spinor} reveals that the parallel transport of 
\[
v^+(0,\overline{\eta})=\begin{pmatrix}
1\\0
\end{pmatrix},\qquad v^-(0,\overline{\eta})=\begin{pmatrix}
0\\1
\end{pmatrix}
\]
along $z^+$ and $z^-$, respectively, 
is given by
\begin{equation*}
\zeta^+(t;0,\overline{\eta})=e^{-\frac{it}{2}}\begin{pmatrix}
1\\0
\end{pmatrix}, \qquad \zeta^-(t;0,\overline{\eta})=e^{\frac{it}{2}}\begin{pmatrix}
0\\1
\end{pmatrix},
\end{equation*}
so that Theorem~\ref{theorem small time expansion principal symbol} gives us
\begin{equation}
\label{principal symbols example 3-sphere dirac}
\mathfrak{a}^+_0(t;0,\overline{\eta})=e^{-\frac{it}{2}}\begin{pmatrix}
1 & 0 \\
0 & 0
\end{pmatrix}, \qquad
\mathfrak{a}^-_0(t;0,\overline{\eta})=e^{\frac{it}{2}}\begin{pmatrix}
0& 0 \\
0 & 1
\end{pmatrix}.
\end{equation}

Let us now move to the subprincipal symbol.
Careful examination of formula \eqref{1 February 2019 formula 3} shows that
\begin{equation}
\label{einstein framing plus}
\overset{*}{K}=-g\,,
\end{equation}
which means that this particular framing has the `Einstein property',
namely, that the Hodge dual of contorsion is proportional to the metric.
Formula \eqref{einstein framing plus} implies that
\begin{equation}
\label{einstein framing corollary}
\nabla\overset{*}{K}=0.
\end{equation}
In view of
\eqref{einstein framing plus}
and
\eqref{einstein framing corollary},
Theorem~\ref{theorem invariant expression principal and subprincipal symbols dirac} gives us
\begin{equation}
\label{example anisotropic subps temp 1}
\begin{split}
\mathfrak{a}_{-1}^+(t;0,\eta)
&
=
\frac{1-it}{4\|\eta\|} \operatorname{Id}+O(t^2),
\end{split}
\end{equation}
\begin{equation}
\label{example anisotropic subps temp 2}
\begin{split}
\mathfrak{a}_{-1}^-(t;0,\eta)
&
=
-\frac{1-it}{4\|\eta\|} \operatorname{Id}-\frac{it}{2\|\eta\|}
\begin{pmatrix}
\eta_3 & \eta_1 -i\eta_2\\
\eta_2+i\eta_2 & -\eta_3
\end{pmatrix}
+O(t^2).
\end{split}
\end{equation}
In particular,
formulae \eqref{example anisotropic subps temp 1} and \eqref{example anisotropic subps temp 2} imply
\begin{equation}
\mathfrak{a}_{-1}^+(t;0,\overline{\eta})
=
\frac14 \begin{pmatrix}
1-it & 0 \\
0 & 1-it
\end{pmatrix}
+O(t^2),
\qquad
\mathfrak{a}_{-1}^-(t;0,\overline{\eta})
=
-\frac14 \begin{pmatrix}
1+it & 0\\
0 & 1-3it
\end{pmatrix}
+O(t^2).
\end{equation}

Let us now deal with $W_-\,$. Arguing as above, one obtains the following expressions for the principal symbols
\begin{equation}
\label{principal symbols example 3-sphere dirac minus}
\mathfrak{a}^+_0(t;0,\overline{\eta})=e^{\frac{it}{2}}\begin{pmatrix}
1 & 0 \\
0 & 0
\end{pmatrix}, \qquad
\mathfrak{a}^-_0(t;0,\overline{\eta})=e^{-\frac{it}{2}}\begin{pmatrix}
0& 0 \\
0 & 1
\end{pmatrix}.
\end{equation}
The subprincipal symbols are calculated in a similar fashion, only now we have
\begin{equation}
\label{einstein framing minus}
\overset{*}{K}=+g\,,
\end{equation}
compare with \eqref{einstein framing plus}.
Combining \eqref{einstein framing minus} with Theorem~\ref{theorem invariant expression principal and subprincipal symbols dirac} we get
\begin{equation}
\label{example anisotropic subps temp 1 minus}
\begin{split}
\mathfrak{a}_{-1}^+(t;0,\eta)
&
=
-\frac{1+it}{4\|\eta\|} \operatorname{Id}+O(t^2),
\end{split}
\end{equation}
\begin{equation}
\label{example anisotropic subps temp 2 minus}
\begin{split}
\mathfrak{a}_{-1}^-(t;0,\eta)
&
=
\frac{1+it}{4\|\eta\|} \operatorname{Id}-\frac{it}{2\|\eta\|}
\begin{pmatrix}
\eta_3 & \eta_1 -i\eta_2\\
\eta_2+i\eta_2 & -\eta_3
\end{pmatrix}
+O(t^2).
\end{split}
\end{equation}
In particular, formulae \eqref{example anisotropic subps temp 1 minus} and \eqref{example anisotropic subps temp 2 minus} imply
\begin{equation}
\mathfrak{a}_{-1}^+(t;0,\overline{\eta})
=
-\frac14 \begin{pmatrix}
1+it & 0 \\
0 & 1+it
\end{pmatrix}
+O(t^2),
\qquad
\mathfrak{a}_{-1}^-(t;0,\overline{\eta})
=
\frac14 \begin{pmatrix}
1-it & 0\\
0 & 1+3it
\end{pmatrix}
+O(t^2).
\end{equation}

Of course, the principal symbols of positive and negative propagators of $W_-$ at $(t;0,\overline{\eta})$ can also be obtained from \eqref{principal symbols example 3-sphere dirac} by means of the gauge transformation \eqref{5 February 2019 formula 2} evaluated at $z^\pm(t;0,\overline{\eta})$, 
\begin{equation}
\label{special G}
\left.G\right|_{(u,v,w)=z^\pm(t;0,\overline{\eta})}=\begin{pmatrix}
e^{\mp it} & 0\\
0 & e^{\pm it}
\end{pmatrix}.
\end{equation}
Namely, multiplying \eqref{principal symbols example 3-sphere dirac}
from the left by the Hermitian conjugate of
\eqref{special G},
we arrive at
\eqref{principal symbols example 3-sphere dirac minus}.

Finally, let us run a test for Theorem~\ref{theorem about third Weyl coefficient}. It is well known \cite{Bar1, Bar2, Sulanake, Trautman} that the eigenvalues
of the Dirac operator on the round 3-sphere are
\[
\pm \left(k+\frac12 \right), \qquad k=1,2,\ldots,
\]
with multiplicity $k(k+1)$. Therefore, in view of \eqref{hyperbolic approach equation 1 plus}, we have
\begin{equation}
\label{13 April 2019 equation 1}
\mathcal{F}_{\lambda\to t}[N_+'\ast \mu](y,t)
=\frac1{2\pi^2}
e^{-\frac{it}{2}}
\sum_{k=1}^{+\infty} k(k+1)e^{-ikt}.
\end{equation}
Note that the quantity $2\pi^2$ appearing in the RHS of \eqref{13 April 2019 equation 1} is the volume of the 3-sphere.
Taking the Fourier transform of the RHS of \eqref{13 April 2019 equation 1} we get
\begin{equation}
\label{13 April 2019 equation 2}
\begin{split}
\mathcal{F}^{-1}_{t\to \lambda}\left[\frac1{2\pi^2}
e^{-\frac{it}{2}}
\sum_{k=1}^\infty(k^2+k)e^{-ikt}\hat{\mu}(t) \right]&=\frac{1}{4\pi^3}\sum_{k=1}^{+\infty}\int_{-\infty}^{+\infty} e^{it(\lambda-\frac12-k)}\,(k^2+k)\,\hat{\mu}(t)\,\dr t\\
&=\frac{1}{4\pi^3}\sum_{k=-\infty}^{+\infty}\int_{-\infty}^{+\infty} e^{-itk}\,(k^2+k)\,\bigl( e^{it(\lambda-\frac12)}\hat{\mu}(t)\bigr)\,\dr t +O(\lambda^{-\infty})\\
&=\frac{1}{2\pi^2}\left((\lambda-\frac12)^2+(\lambda-\frac12)+O(\lambda^{-\infty}) \right)\\
&=\frac{1}{2\pi^2}\left(\lambda^2-\frac14+O(\lambda^{-\infty})\right).
\end{split}
\end{equation}
Combining \eqref{13 April 2019 equation 2} and \eqref{13 April 2019 equation 1} we arrive at
\begin{equation}
\label{13 April 2019 equation 3}
[N_+'\ast {\mu}](y,\lambda)=\frac{1}{2\pi^2}\left(\lambda^2-\frac14+O(\lambda^{-\infty})\right) \qquad \text{as }\lambda\to +\infty.
\end{equation}
Since $\mathcal{R}(y)=6$, formula \eqref{13 April 2019 equation 3} is in agreement with \eqref{theorem about third Weyl coefficient formula}.

\subsection{The case $M=\mathbb{S}^2\times\mathbb{S}^1$}
\label{The case 2-sphere times circle}

Let $M=\mathbb{S}^2\times\mathbb{S}^1$ be endowed with the metric $g=g_{\mathbb{S}^2}+d\varphi^2$, where $g_{\mathbb{S}^2}$ is the round metric on the 2-sphere. Let $y\in M$ be given. In this subsection we shall compute the small time expansion for the subprincipal symbols of the Dirac propagator $\widetilde{W}$ associated with a Levi-Civita framing at $y$. In this case, the result will not be isotropic in momentum $\eta$, because, unlike the previous example, $(\mathbb{S}^2\times\mathbb{S}^1,g)$ is not an Einstein manifold.

Without loss of generality, we assume that $y$ coincides with the north pole when projected onto $\mathbb{S}^2$. The exponential map $\exp_y:T_yM \to M$ is realised explicitly by
\begin{equation}
\label{exponential map S2 times S1 dirac}
(u,v,w) \mapsto (\theta=\sqrt{u^2+v^2}, \phi=\arctan(v/u), \varphi=w),
\end{equation}
where $(\theta,\phi)$ are standard spherical coordinates on $\mathbb{S}^2$.
Formula \eqref{exponential map S2 times S1 dirac} defines geodesic normal coordinates $(u,v,w)$ in a neighbourhood of $y$. In such coordinates, the metric $g$ reads
\begin{equation}
\label{12 June 2019 formula 2}
g(u,v,\textcolor{black}{w})=\frac{1}{u^2+v^2}
\begin{pmatrix}
u^2+\frac{v^2 \,\sin^2( \sqrt{u^2+v^2})}{u^2+v^2}
&
uv \left(1-\frac{\sin^2( \sqrt{u^2+v^2})}{u^2+v^2} \right)
&
0
\\
uv \left(1-\frac{\sin^2( \sqrt{u^2+v^2})}{u^2+v^2} \right)
&
v^2+\frac{u^2 \,\sin^2( \sqrt{u^2+v^2})}{u^2+v^2}
&
0
\\
0
&
0
&
u^2+v^2
\end{pmatrix}.
\end{equation}
We will assume that normal coordinates are chosen so that the Levi-Civita framing satisfies $\widetilde{e}_j{}^\alpha(y)=\delta_j{}^\alpha$. In this case, the Hamiltonian flows generated by the eigenvalues of $\widetilde{W}_\mathrm{prin}$ read, simply,
\[
z^\pm(t;0,\eta)=\pm t\frac{\eta}{\|\eta\|}, \qquad \xi^\pm(t;0,\eta)=\eta.
\]
The Ricci curvature of $g$ in normal coordinates \textcolor{black}{$(u,v,w)$} is given by
\begin{equation}
\begin{aligned}
\operatorname{Ric}(u,v,w)
&=
\frac{1}{u^2+v^2}
\begin{pmatrix}
u^2+\frac{v^2 \,\sin^2( \sqrt{u^2+v^2})}{u^2+v^2}
&
uv \left(1-\frac{\sin^2( \sqrt{u^2+v^2})}{u^2+v^2} \right)
&
0
\\
uv \left(1-\frac{\sin^2( \sqrt{u^2+v^2})}{u^2+v^2} \right)
&
v^2+\frac{u^2 \,\sin^2( \sqrt{u^2+v^2})}{u^2+v^2}
&
0
\\
0
&
0
&
0
\end{pmatrix}.
\end{aligned}
\end{equation}
Hence, Theorem~\ref{theorem small time expansion subprincipal symbols levi-civita dirac} tells us that
\begin{equation}
\label{small time expansion subprincipal symbols levi-civita dirac example S2 times S1}
\begin{split}
\widetilde{\mathfrak{a}}_{-1}^\pm(t;y,\eta)
&
=
\mp it
\left(
\frac{1}{12\|\eta\|}\,\widetilde{P}^\pm(y,\eta)
-
\frac{1}{8\|\eta\|^2}(\eta_1 \,\sigma^1(y)+\eta_2\,\sigma^2(y))
\right)
+
O(t^2)\,
\\
&
=
\mp \frac{it}{24 \|\eta\|^2} \left[\|\eta\|\operatorname{Id} 
+(-3\pm 1)\,\eta_\alpha\,\sigma^\alpha(y)
+3\,s^3\,\eta_\beta\,\widetilde{e}_3{}^\beta(y)
\right]
+O(t^2),
\end{split}
\end{equation}
where the $s^j$ and the $\sigma^\alpha$
are defined by formulae \eqref{Pauli matrices basic} and \eqref{Pauli matrices projection} respectively,
and $\widetilde{e}_3$ is the vector field $\partial/\partial\varphi$
(unit vector field along the positive direction of the circle $\mathbb{S}^1$).

Let us stress once again that, even though the intermediate steps depend on the choice of coordinates, the final result \eqref{small time expansion subprincipal symbols levi-civita dirac example S2 times S1} is a scalar matrix-function, thus independent of the choice of coordinates.
The only assumption involved in the derivation of formula
\eqref{small time expansion subprincipal symbols levi-civita dirac example S2 times S1}
is that we used a particular Levi-Civita framing at the point $y$,
one which respects the product structure of the manifold.
The presence of the vector field $\widetilde{e}_3$ in formula
\eqref{small time expansion subprincipal symbols levi-civita dirac example S2 times S1}
is a manifestation of anisotropy.

\section{Acknowledgements}

We are grateful to Yiannis Petridis for providing useful references.  Furthermore, we would like to thank an anonymous referee for insightful comments, in particular,  for suggesting the argument in Remark~\ref{remark Lalpha}.

\begin{appendices}

\section{The Weitzenb\"ock connection}
\label{Geometric properties of the Weitzenb\"ock connection}

In this appendix we recall the main properties of the Weitzenb\"ock connection and fix our sign conventions, which are chosen in agreement with \cite{nakahara}.

\

Let $M$ be an oriented Riemannian 3-manifold and let $\{e_j\}_{j=1}^3$ be a global orthonormal framing.

\begin{definition}
The \emph{Weitzenb\"ock connection} is the affine connection $\nabla^W$ on $M$ defined by the condition
\begin{equation}
\label{Weitzenbock defining condition}
\nabla^W_v (f^i\, e_i)=v(f^i)\, e_i\,,
\end{equation}
for every vector field $v$ and $f^i\in C^\infty(M;\mathbb{R})$, $i=1,2,3$.
\end{definition}

The Weitzenb\"ock connection is a curvature-free metric-compatible connection. Formula \eqref{Weitzenbock defining condition} implies
\[
0=\nabla^W_{e_k} e_j{}^\alpha=e_k{}^\beta\, \dfrac{\partial e_j{}^\alpha}{\partial x^\beta}+ e_k{}^\beta \,\Upsilon^\alpha{}_{\beta \gamma}\, e_j{}^\gamma,
\]
which, in turn, yields a formula for the Weitzenb\"ock connection coefficients $\Upsilon^\alpha{}_{\beta\gamma}$ in terms of the framing:
\begin{equation}
\label{weitzenbock connection coefficients definition}
\Upsilon^\alpha{}_{\beta\gamma}=- \,e^j{}_\gamma  \dfrac{\partial e_j{}^\alpha}{\partial x^\beta}= e_j{}^\alpha  \dfrac{\partial e^j{}_\gamma}{\partial x^\beta}\,.
\end{equation}
Here $e^j{}_\alpha:=\delta^{jk}\,g_{\alpha\beta}\,e_k{}^\beta$.
The torsion tensor associated with $\nabla^W$ is
\begin{equation}
\label{torsion definition}
T^\alpha{}_{\beta \gamma}= \Upsilon^\alpha{}_{\beta\gamma}-\Upsilon^\alpha{}_{\gamma \beta}
\end{equation}
and the curvature tensor vanishes identically.
The Weitzenb\"ock connection coefficients and the Christoffel symbols are related via the identity
\begin{equation}
\label{weitzenbock connection coefficients vs christoffel}
\Upsilon^\alpha{}_{\beta \gamma}=\Gamma^\alpha{}_{\beta\gamma}+\frac12\left(T^\alpha{}_{\beta\gamma}+T_\beta{}^\alpha{}_\gamma+T_\gamma{}^\alpha{}_\beta \right),
\end{equation}
see \cite[Eqn.~(7.34)]{nakahara}. The second summand on the RHS of \eqref{weitzenbock connection coefficients vs christoffel}
\begin{equation}
\label{contorsion definition}
K^\alpha{}_{\beta \gamma}:=\frac12\left(T^\alpha{}_{\beta\gamma}+T_\beta{}^\alpha{}_\gamma+T_\gamma{}^\alpha{}_\beta \right)
\end{equation}
is called \emph{contorsion} of $\nabla^W$. Note that the torsion tensor is antisymmetric in the second and third indices, $T^\alpha{}_{\beta\gamma}=-T^{\alpha}{}_{\gamma\beta}$, whereas the contorsion tensor is antisymmetric in the first and third ones, $K_{\alpha\beta\gamma}=-K_{\gamma\beta\alpha}$ (the first index was lowered using the metric). Torsion and contorsion can be expressed one in terms of the other and capture the geometric information encoded within the framing.

In dimension three antisymmetric tensors of order two are equivalent to vectors.
Therefore, we define
\begin{equation}
\label{hodge star torsion}
\overset{*}{T}{}_{\alpha\beta}:=\frac{1}{2}T_\alpha{}^{\mu\nu}\,E_{\mu\nu\beta}
\end{equation}
and
\begin{equation}
\label{hodge star contorsion}
\overset{*}{K}{}_{\alpha\beta}:=\frac{1}{2}K^\mu{}_\alpha{}^\nu\,E_{\mu\nu}{}_\beta\,,
\end{equation}
where
\begin{equation}
\label{Levi-Civita tensor}
E_{\alpha\beta\gamma}(x):=\rho(x)\,\varepsilon_{\alpha\beta\gamma}\,,
\end{equation}
$\rho$ is the Riemannian density and $\varepsilon$ is the totally antisymmetric symbol,
$\varepsilon_{123}:=+1$.
It is often convenient to use \eqref{hodge star torsion} and \eqref{hodge star contorsion} instead of $T$ and $K$
because the former have lower order -- two instead of three.

As a final remark, we observe that formulae
\eqref{hodge star torsion}, \eqref{hodge star contorsion} and \eqref{contorsion definition}
imply
\begin{equation}
\label{torsion star vs contorsion star}
\overset{*}{K}_{\alpha\beta}=\overset{*}{T}_{\alpha\beta}-\frac12 \overset{*}{T}{}^\gamma{}_\gamma\, g_{\alpha\beta}\,,
\end{equation}
\begin{equation}
\label{contorsion star vs torsion star}
\overset{*}{T}_{\alpha\beta}=\overset{*}{K}_{\alpha\beta}-\overset{*}{K}{}^\gamma{}_\gamma\, g_{\alpha\beta}\,.
\end{equation}

\section{Some techincal proofs}
\label{techincal proofs}

\subsection{Proof of Theorem~\ref{theorem covariant derivatives gauge transformation}}
\label{appendix proof covariant derivative gauge transformation}

In the following we work in normal coordinates centred at $y=0$ such that
\[
e_j{}^\alpha(0)=\widetilde{e}_j{}^\alpha(0)=\delta_j{}^\alpha.
\]
Since $G\in C^\infty(M;SU(2))$ and $G(0)=\mathrm{Id}$, there exist smooth real-valued functions $A_k$, $k=1,2,3$, such that $A_k(0)=0$ and
\begin{equation}
\label{G as exponential}
G(x)=e^{is^k\,A_k(x)}
\end{equation}
in a neighbourhood of $y=0$.
Differentiating \eqref{G as exponential} with respect to $x$ and evaluating the result at $0$, we obtain
\begin{equation}
\label{nabla g proof experssion}
G_{x^\alpha}(0)=i s^k\,F_{k\alpha},
\end{equation}
where
$F_{k\alpha}:=[A_k]_{x^\alpha}(0)$.

Now, differentiating \eqref{new framing generated by G} with respect to $x$ and evaluating the result at $0$, we obtain
\begin{equation}
\label{proof nabla G temp 1}
\begin{split}
\frac{\partial e_j{}^\alpha}{\partial x^\beta}(0)
&=\frac12 \operatorname{tr}\left[s_j\,G_{x^\beta}^*(0)\,s^k+s_j\,s^k\,G_{x^\beta}(0)\right]\widetilde{e}_k{}^\alpha(0)+\frac{\partial \widetilde{e}_k{}^\alpha}{\partial x^\beta}(0)\\
&=\frac12 \operatorname{tr}\left[[s_j\,s^k\,G_{x^\beta}(0)]^*+s_j\,s^k\,G_{x^\beta}(0\right]\widetilde{e}_k{}^\alpha(0)+\frac{\partial \widetilde{e}_k{}^\alpha}{\partial x^\beta}(0)\\
&=\operatorname{Re}\,\operatorname{tr}\left[s_j\,s^k\,G_{x^\beta}(0)\right]\widetilde{e}_k{}^\alpha(0)+\frac{\partial \widetilde{e}_k{}^\alpha}{\partial x^\beta}(0).
\end{split}
\end{equation}
Contracting \eqref{proof nabla G temp 1} with $e^j{}_\gamma(0)=\widetilde{e}^j{}_\gamma(0)=\delta^j{}_\gamma$, using \eqref{weitzenbock connection coefficients definition} and rearranging, we obtain
\begin{equation}
\label{proof nabla G temp 2}
\begin{split}
\widetilde{\Upsilon}^\alpha{}_{\beta\gamma}(0)-\Upsilon^\alpha{}_{\beta\gamma}(0)
&
=
\operatorname{Re}\operatorname{tr}\left[i\,s_j s^k s^l\right]\,F_{l\beta} \,\delta^j{}_\gamma\,\delta_k{}^\alpha
\\
&
=
-2\,\varepsilon_{\gamma}{}^{\alpha l}\,F_{l\beta}.
\end{split}
\end{equation}
In view of \eqref{torsion definition}, formula \eqref{proof nabla G temp 2} implies
\begin{equation}
\label{proof nabla G temp 3}
\begin{split}
T^\alpha{}_{\beta \gamma}(0)-\widetilde{T}^\alpha{}_{\beta \gamma}(0)=2\,\varepsilon_{\gamma}{}^{\alpha l}\,F_{l\beta}-2\,\varepsilon_{\beta}{}^{\alpha l}\,F_{l\gamma}.
\end{split}
\end{equation}
Contracting \eqref{proof nabla G temp 3} with $\frac12E_\sigma{}^{\beta\gamma}(y)=\frac12\varepsilon_\sigma{}^{\beta\gamma}$, cf.~\eqref{Levi-Civita tensor}, we get
\begin{equation}
\label{proof nabla G temp 4}
\begin{split}
\overset{*}{T}{}^\alpha{}_\sigma(0)-\overset{*}{\widetilde T}{}^\alpha{}_\sigma(0)
&
=
2 \varepsilon_\sigma{}^{\beta\gamma}\,\varepsilon_\gamma{}^{\alpha l} F_{l\beta}
\\
&
=
2\delta^{\beta l}\,F_{l\beta} \, \delta_\sigma{}^\alpha- 2 \delta_\sigma{}^l\, F_{l\alpha}.
\end{split}
\end{equation}
Inverting \eqref{proof nabla G temp 4} so as to express $F$ in terms of $[\overset{*}{T}-\overset{*}{\widetilde T}](0)$, we arrive at
\begin{equation}
\label{proof nabla G temp 5}
\begin{split}
-2F_{k\beta}
&
=
\delta_k{}^\alpha\,[\overset{\ast}{T}-\overset{\ast}{\widetilde T}]_{\alpha\beta}(y)-\frac12\,\delta_{k\beta}\,[\overset{\ast}{T}-\overset{\ast}{\widetilde T}]^\gamma{}_\gamma (0)\\
&
=\delta_k{}^\alpha \bigl[\overset{\ast}{K}-\overset{\ast}{\widetilde K}\bigr]_{\alpha\beta}(0).
\end{split}
\end{equation}
Substitution of \eqref{proof nabla G temp 5} into \eqref{nabla g proof experssion} gives 
\eqref{covariant derivative of gauge transformation G}.

\subsection{Proof of Theorem~\ref{theorem covariant derivatives gauge transformation Levi-civita}}
\label{appendix proof double covariant derivative gauge transformation}

Recall that according to formula \eqref{new framing generated by G Levi-civita} we have $G(y)=\mathrm{Id}$. 
In the following we work in a sufficiently small neighbourhood $\mathcal{U}$ of $y$ and we choose normal coordinates centred at $y=0$ such that $\widetilde e_j{}^\alpha(0)=e_j{}^\alpha(0)=\delta_j{}^\alpha$. 

Since $G\in C^\infty(M;SU(2))$ and $G(0)=\mathrm{Id}$, there exist smooth real-valued functions $A_k$, $k=1,2,3$, such that $v_k(0)=0$ and
\begin{equation}
\label{proof nabla nabla G temp 1}
G(x)=e^{is^k\,A_k(x)}
\end{equation}
in a neighbourhood of $y=0$.
Differentiating \eqref{proof nabla nabla G temp 1} twice with respect to $x$ and evaluating the result at zero we obtain
\begin{equation}
\label{proof nabla G temp 6 LC}
\begin{split}
G_{x^\alpha x^\beta}(0)
&
=
i s^k \,[A_k]_{x^\alpha x^\beta}(0)-\frac12 s^k\,s^j\,(F_{k\alpha}F_{j\beta}+F_{j\alpha}F_{k\beta})\\
=
&
i s^k\, H_{k\alpha\beta}-\delta^{jk}\,\mathrm{Id}\,F_{j\alpha}F_{k\beta}.
\end{split}
\end{equation}
Here $H_{k\alpha\beta}:=[A_k]_{x^\alpha x^\beta}(0)$ and $F_{k\alpha}:=[A_k]_{x^\alpha}(0)$. The task at hand is to express $H$ in terms of the contorsion tensor $K$ and its derivatives.

Differentiating the identity
\begin{equation*}
\label{proof nabla G temp 7 LC}
\begin{split}
\Upsilon^\alpha{}_{\beta\gamma}(x)&=e_k{}^\alpha(x) \dfrac{\partial e^k{}_\gamma}{\partial x^\beta}(x)\\
\end{split}
\end{equation*}
with respect to $x^\mu$, evaluating the outcome at $y=0$ and resorting to Lemma~\ref{lemma Levi-Civita framing}, we obtain
\begin{equation}
\label{proof nabla G temp 8 LC}
\begin{split}
\left[\Upsilon^\alpha{}_{\beta\gamma}\right]_{x^\mu}(0)
&
=
\dfrac{\partial e_k{}^\alpha}{\partial x^\mu}(0) \dfrac{\partial e^k{}_\gamma}{\partial x^\beta}(0)+e_k{}^\alpha(0) \dfrac{\partial^2 e^k{}_\gamma}{\partial x^\beta \partial x^\mu}(0)
\\
&
=
-\Upsilon^\alpha{}_\mu{}_\rho(0)  \, \Upsilon^\rho{}_{\beta\gamma}(0)
\\
&
+\delta_k{}^\alpha
\operatorname{Re} \operatorname{tr}[s^k \,G^*_{x^\beta x^\mu}(0)\,s_l+s^k\,G^*_{x^\beta}(0)\,s_l\,G_{x^\mu}(0)]\,\delta^l{}_\gamma\\
&+\delta_k{}^\alpha\,[\widetilde{e}^k{}_\gamma]_{x^\beta x^\mu}(0)
\\
&
=
-\Upsilon^\alpha{}_\mu{}_\rho(0)  \, \Upsilon^\rho{}_{\beta\gamma}(0)+
\delta_k{}^\alpha\delta^l{}_\gamma\operatorname{Re} \operatorname{tr}[s_l\,s^k \,G^*_{x^\beta x^\mu}(0)]
\\
&
+\delta_k{}^\alpha\delta^l{}_\gamma \operatorname{Re} \operatorname{tr}[s^k\,G^*_{x^\beta}(0)\,s_l \,G_{x^\mu}(0)]+\delta_k{}^\alpha\,[\widetilde{e}^k{}_\gamma]_{x^\beta x^\mu}(0).
\end{split}
\end{equation}

Straightforward calculations show that
\begin{equation}
\label{proof nabla G temp 9 LC}
\begin{split}
-\Upsilon^\alpha{}_\mu{}_\rho(0)  \, \Upsilon^\rho{}_{\beta\gamma}(0)
&
=-\operatorname{Re}\,\operatorname{tr}\left[s^\alpha\,G^*_{x^\mu}(0)s_\rho\right] \operatorname{Re}\,\operatorname{tr}\left[s^\rho\,G^*_{x^\beta}(0)s_\gamma\right]
\\
&
=
4\delta^\alpha{}_\gamma F^r{}_\beta F_{r\mu}-4\delta^{\alpha j}\,\delta_\gamma{}^k\,F_{j\beta} F_{k\mu},
\end{split}
\end{equation}
\begin{equation}
\label{proof nabla G temp 10 LC}
\delta_k{}^\alpha\delta^l{}_\gamma\operatorname{Re} \operatorname{tr}[s_l\,s^k\,G^*_{x^\beta x^\mu}(0)]=-2 \varepsilon^\alpha{}_{\gamma}{}^r\,H_{r\beta\mu}-2\delta^\alpha{}_\gamma\,F^r{}_\beta\,F_{r\mu}
\end{equation}
and
\begin{equation}
\label{proof nabla G temp 11 LC}
\delta_k{}^\alpha\delta^l{}_\gamma \operatorname{Re} \operatorname{tr}[s_l\,G^*_{x^\beta}(0)\,s^k \,G_{x^\mu}(0)]=
2(\delta^{\alpha k}\delta^j{}_\gamma+\delta^{\alpha j}\delta^k{}_\gamma)\,F_{j\mu}F_{k\beta}-2\delta^\alpha{}_\gamma\,F^r{}_\beta\,F_{r\mu}.
\end{equation}
Substituting \eqref{proof nabla G temp 9 LC}--\eqref{proof nabla G temp 11 LC} into \eqref{proof nabla G temp 8 LC} 
we obtain
\begin{equation}
\label{proof nabla G temp 12 LC}
\left[\Upsilon^\alpha{}_{\beta\gamma}\right]_{x^\mu}(0)= -2 \varepsilon^\alpha{}_{\gamma}{}^r\,H_{r\beta\mu}+2(\delta^{\alpha j}\delta^k{}_\gamma-\delta^{\alpha k}\delta^j{}_\gamma)\,F_{j\mu}F_{k\beta}+\delta_k{}^\alpha\,[\widetilde{e}^k{}_\gamma]_{x^\beta x^\mu}(0).
\end{equation}
Summing up \eqref{proof nabla G temp 12 LC} and \eqref{proof nabla G temp 12 LC} with indices $\beta$ and $\mu$ swapped, we arrive at
\begin{equation}
\label{proof nabla G temp 13 LC}
\left[\Upsilon^\alpha{}_{\beta\gamma}\right]_{x^\mu}(0)+\left[\Upsilon^\alpha{}_{\mu\gamma}\right]_{x^\beta}(0)= -4\varepsilon^\alpha{}_{\gamma}{}^r\,H_{r\beta\mu}+2\delta_k{}^\alpha\,[\widetilde{e}^k{}_\gamma]_{x^\beta x^\mu}(0).
\end{equation}
Now, formula \eqref{weitzenbock connection coefficients vs christoffel} and the fact that the Christoffel symbols vanish at $y=0$ imply
\begin{equation}
\label{proof nabla G temp 14 LC}
\varepsilon_\alpha{}^{\gamma}{}_\rho\,\Upsilon^\alpha{}_{\mu\gamma}(0)=\varepsilon_\alpha{}^{\gamma}{}_\rho\,K^\alpha{}_{\mu\gamma}(0)=2 \overset{*}{K}_{\mu\rho}(0).
\end{equation}
Hence, by contracting \eqref{proof nabla G temp 13 LC} with $\varepsilon_\alpha{}^{\gamma}{}_\rho$, substituting \eqref{proof nabla G temp 14 LC} in, and resorting to the identity
\[
\varepsilon_\alpha{}^{\gamma}{}_\rho \,\varepsilon^\alpha{}_{\gamma}{}^r=2\delta_\rho{}^r,
\]
we obtain
\begin{equation}
\label{proof nabla G temp 15 LC}
[\overset{*}{K}_{\beta\rho}]_{x^\mu}(0)+  [\overset{*}{K}_{\mu\rho}]_{x^\beta}(0) =-4\,\delta_\rho{}^r\,H_{r\beta\mu}+\varepsilon_\alpha{}^{\gamma}{}_\rho\,\delta_k{}^\alpha\,[\widetilde{e}^k{}_\gamma]_{x^\beta x^\mu}(0).
\end{equation}

We claim that
\begin{equation}
\label{proof nabla G temp 16 LC}
\varepsilon_\alpha{}^{\gamma}{}_\rho\,\delta_k{}^\alpha\,[\widetilde{e}^k{}_\gamma]_{x^\beta x^\mu}(0)=0.
\end{equation}
To see this, let us observe that formula \eqref{expansion Levi-Civita framing Dirac} implies
\[
\widetilde{e}^k{}_\gamma(x)=e^k{}_\gamma(0)-\frac16 e^k{}_\rho(0)\,R_{\gamma\tau}{}^\rho{}_\nu(0)\,x^\tau x^\nu +O(\|x\|^3), \qquad j=1,2,3,
\]
so that
\begin{equation}
\label{proof nabla G temp 17 LC}
\delta_k{}^\alpha\,[\widetilde{e}^k{}_\gamma]_{x^\beta x^\mu}(0)=-\frac16  \left(R_{\gamma\beta}{}^\alpha{}_\mu+R_{\gamma\mu}{}^\alpha{}_\beta \right)(0).
\end{equation}
The RHS of \eqref{proof nabla G temp 17 LC} is symmetric in $\alpha$ and $\gamma$, whereas $\varepsilon_\alpha{}^{\gamma}{}_\rho$ is antisymmetric in the same indices, so \eqref{proof nabla G temp 16 LC} follows.

All in all, \eqref{proof nabla G temp 6 LC}, \eqref{proof nabla G temp 15 LC} and \eqref{proof nabla G temp 16 LC} give us
\begin{equation}
\label{proof nabla G temp 18 LC}
\nabla_\alpha \nabla_\beta\, G(0)=-\frac{i}{4} [ \nabla_\alpha\overset{*}{K}_{\beta\rho}(0)+  \nabla_\beta \overset{*}{K}_{\alpha\rho}(0) ]\sigma^\rho(0)-\delta^{jk}\,\mathrm{Id}\,F_{j\alpha}F_{k\beta}.
\end{equation}
Finally, substitution of \eqref{proof nabla G temp 5} with $\widetilde{K}(0)=0$ (which is the case for the Levi-Civita framing) into \eqref{proof nabla G temp 18 LC} yields \eqref{double covariant derivative of gauge transformation G Levi-civita}.

\end{appendices}

\end{document}